\documentclass{amsart}
\usepackage{amssymb,amsmath,amsthm}
\usepackage{amsaddr}
\usepackage{amsfonts}
\usepackage{authblk}
\usepackage{graphicx}
\usepackage{epstopdf}
\usepackage{caption}
\usepackage{subcaption}
\usepackage{hyperref}
\usepackage{color}
\usepackage{marginnote}
\usepackage{float}
\usepackage{graphicx}
\usepackage{comment}
\usepackage{tikz-cd}
\usetikzlibrary{arrows}
\usetikzlibrary{intersections,positioning}
\tikzset{>=latex}
\usepackage{lipsum}

\newcommand{\C}{\mathbb{C}}

\newcommand{\R}{\mathbb{R}}
\newcommand{\Z}{\mathbb{Z}}
\newcommand{\Q}{\mathbb{Q}}
\newcommand{\N}{\mathbb{N}}

\newcommand{\J}{\mathcal{J}}

\newcommand{\F}{\mathcal{F}}

\renewcommand{\sl}{\text{sl}}

\newcommand{\lk}{\text{link}}

\newcommand{\rr}{\mathbb{R}}

\newcommand{\zz}{\mathbb{Z}}

\newcommand{\Sp}{\mathrm{Sp}}

\newcommand{\calc}{\mathcal{C}}
\newcommand{\cb}{\mathcal{B}}

\newcommand{\al}{\alpha}

\newcommand{\tht}{\theta}
\newcommand{\zt}{\zeta}

\newcommand{\vep}{\varepsilon}

\newcommand{\ey}{\frac{1}{2}}

\newcommand{\pr}{p_r}

\newcounter{newcounter}[section]
\numberwithin{equation}{section}
\numberwithin{newcounter}{section}
\numberwithin{figure}{section}
\numberwithin{footnote}{section}

\newtheorem{thm}[newcounter]{Theorem}
\newtheorem{defi}[newcounter]{Definition}
\newtheorem{prop}[newcounter]{Proposition}
\newtheorem{lem}[newcounter]{Lemma}
\newtheorem{cor}[newcounter]{Corollary}
\newtheorem{rem}[newcounter]{Remark}

\newtheorem{claim}[newcounter]{Claim}

\newcommand{\beg}{\begin}
\newcommand{\bea}{\beg{eqnarray}}
\newcommand{\eea}{\end{eqnarray}}
\newcommand{\authorfootnotes}{\renewcommand\thefootnote{\@fnsymbol\c@footnote}}%
\makeatother

\title[The Spatial Isosceles Three-Body Problem]
{A symplectic dynamics approach to the spatial isosceles three-body problem}

\begin{document}


\email{xjhu@sdu.edu.cn}

\email{liulei30@email.sdu.edu.cn/liulei@bicmr.pku.edu.cn}

\email{ywou@sdu.edu.cn}

\email{pas383@nyu.edu}

\email{yugw@nankai.edu.cn}

\maketitle

\begin{center}

  \normalsize
  \authorfootnotes
  Xijun Hu\textsuperscript{1}, Lei Liu\textsuperscript{1,2},
  Yuwei Ou\textsuperscript{1}, Pedro A. S. Salom\~ao\textsuperscript{3} and Guowei Yu\textsuperscript{4} \par \bigskip

  \textsuperscript{1}School of Mathematics, Shandong University  \par
  \textsuperscript{2}Beijing International Center for Mathematical Research, Peking University
  \par
  \textsuperscript{3}New York University - Shanghai,  \par
  \textsuperscript{4}Chern Institute of Mathematics and LPMC, Nankai University

\end{center}


\begin{abstract}
 We study the spatial isosceles three-body problem from the perspective of Symplectic Dynamics. For certain choices of mass ratio, angular momentum, and energy, the dynamics on the energy surface is equivalent to a Reeb flow on the tight three-sphere. We find a Hopf link formed by the Euler orbit and a symmetric brake orbit, which  spans an open book decomposition whose pages are annulus-like global surfaces of section. In the case of large mass ratios, the Hopf link is non-resonant, forcing the existence of infinitely many periodic orbits. The rotation number of the Euler orbit plays a fundamental role in the existence of periodic orbits and their symmetries. We explore such symmetries in the Hill region and show that  the Euler orbit is negative hyperbolic for an open set of parameters while it can never be positive hyperbolic. Finally, we address convexity and determine for each parameter whether the energy surface is strictly convex, convex, or non-convex. Dynamical consequences of this fact are then discussed.
\end{abstract}

\tableofcontents





\section{Introduction and main results} \label{sec: intro}

The spatial three-body problem is the study of the motion of three point masses in $\rr^3$ subjected to Newton's universal gravitational law. There exists an invariant subsystem in which two equal masses are symmetric to a fixed axis where the third body moves.  As the three bodies always form an isosceles triangle, this subsystem is called the spatial isosceles three-body problem.

The spatial isosceles three-body problem plays an important role in understanding the overall dynamics of the $n$-body problem. It  was used by Xia \cite{Xia92} in the five-body problem to find non-collision trajectories that escape to infinity in finite time, solving the Painlev\'e conjecture.

The Hamiltonian of the spatial isosceles three-body problem in  reduced form is given by
\begin{equation}\label{ham0}
H = \frac{p_r^2+p_z^2}{2} +\frac{\varpi^2}{2r^2} - \frac{1}{r} -\frac{4\alpha^{-1}}{(r^2+(1+2\alpha)z^2)^{1/2}},
\end{equation}
where $r\in \R^+, z\in \R$ are suitable coordinates on the plane determined by the three bodies, and $p_r,p_z\in \R$ are the corresponding momenta, see section \ref{reduce to two degree system}. The parameters $\varpi, \alpha >0$ are referred to as the angular momentum with respect to the axis of symmetry, and the mass ratio between the first two bodies and the third body.

There are two special related problems that have been mostly studied:
\begin{itemize}
    \item The restricted isosceles three-body problem, or  Sitnikov problem, where the third body has negligible mass and its motion along the axis of symmetry is ruled by a bounded Kepler motion of  the first two bodies.  Up to a suitable rescaling in $z$ the equations of motion can be recovered by taking $\alpha
    \to +\infty$.

    \item The planar isosceles three-body problem, where the angular momentum $\varpi$ is assumed to be zero. Its Hamiltonian can be recovered  directly from \eqref{ham0}.

\end{itemize}

The dynamics of the Sitnikov problem is very rich. As predicted by Chazy \cite{Chazy} and proved by Sitnikov \cite{Sit60}, it admits  oscillatory motions, i.e., unbounded trajectories of the massless body   that oscillate infinitely many times. The complexity of the Sitnikov problem was further discussed by  Alekseev  \cite{Alk69} and Moser  \cite{Ms73}. They used surfaces of section and symbolic dynamics to encode a rich variety of trajectories including oscillatory motions, periodic orbits, homoclinic and heteroclinic orbits, etc. Some other approaches concerning the existence and stability of periodic orbits in the Sitnikov problem are found in \cite{Zhang19,Zhang21, GNR18,LO2008,Ortega2016}.

In the planar isosceles three-body problem, collisions occur for a typical trajectory. Double collisions are usually regularized via Levi-Civita coordinates. Using McGehee's blow up techniques from \cite{McG73},  Devaney \cite{Dev80} found suitable coordinates to study the dynamics near triple collisions. Such coordinates were then used by Sim\'o and Martinez \cite{SM87}  to find chaotic dynamics near homoclinic and heteroclinic trajectories to triple collisions. We refer to  Moeckel, Montgomery and Venturelli \cite{MMV12} and Chen \cite{Chen13} for the existence of the so called brake orbits in the planar problem, i.e., periodic orbits  whose velocity vanishes twice along the period.

In the general spatial isosceles three-body problem determined by \eqref{ham0}, collisions do not occur. Alekseev \cite{Alk72} and Moeckel \cite{Mk84} generalized some of the previous results showing that the dynamics is still rich  for large values of $\alpha$ and for small values of $\varpi$, respectively.
The existence of periodic orbits using alternative methods were treated in \cite{CorLli04,Shi09}.

A common ingredient in many of the works above is a surface of section, i.e., a surface transverse to the flow whose first return map enables two-dimensional methods in dynamics. The ideal scenario is the presence of a global surface of section, that is an embedded surface bounded by periodic orbits, transverse to the flow in its interior, and so that every trajectory hits it forward and backward in time. The total dynamics is then encoded into an area-preserving surface diffeomorphism given by the first return map.

In this work we study the dynamics of the Hamiltonian \eqref{ham0} assuming
that the parameters $\varpi,\alpha$ are positive, the energy $h$ is negative, and the following conditions are satisfied
\begin{equation}\label{cond1}
\frac{1}{2} < \varpi^2|h| < \frac{(1+4\alpha^{-1})^2}{2}.
\end{equation}
These are precisely the conditions that make the energy surface
$$
\mathfrak{M}:=H^{-1}(h)\subset \R^4
$$
a regular sphere-like hypersurface. As we shall explain later, the dynamics on $\mathfrak{M}$ is actually determined by $\varpi^2h$ and $\alpha$ and we may fix $h=-1$ without loss of generality. The mechanical nature of $H$ implies that $\mathfrak{M}$  has contact-type and thus the Hamiltonian flow on $\mathfrak{M}$ is equivalent to the Reeb flow of a contact form on the tight three-sphere. 

Reeb flows on the tight three-sphere are equivalent to Hamiltonian flows on star-shaped hypersurfaces in $\R^4$. They have been extensively considered by many authors. Periodic orbits on such energy surfaces were studied by Rabinowitz \cite{Ra}, and Weinstein \cite{We}, who conjectured that any Reeb flow on a closed energy surface admits a periodic orbit. Hofer \cite{Hofer93} introduced finite energy pseudo-holomorphic curves to prove Weinstein conjecture on the three-sphere and Taubes \cite{Taubes} proved it for general Reeb flows in dimension $3$.

A remarkable result concerning Reeb flows on the tight-three sphere was proved by Hofer, Wysocki, and Zehnder \cite{convex} for dynamically convex contact forms, i.e., those contact forms whose periodic orbits have index  $\geq 3$. It can be summarized as follows.

\begin{thm}[Hofer-Wysocki-Zehnder \cite{convex}]\label{thm:convexHWZ} The Reeb flow of a dynamically convex contact form on the tight three-sphere admits a disk-like global surface of section bounded by an index-$3$ periodic orbit. In particular, the flow has either two or infinitely many periodic orbits.
\end{thm}

It is also proved in \cite{convex} that every strictly convex hypersurface in $\R^4$ induces a dynamically convex contact form on the tight three-sphere. The theory of pseudo-holomorphic curves, developed by Hofer, Wysocki, and Zehnder \cite{Hofer93,props1,props2,props3}, was used to prove Theorem \ref{thm:convexHWZ}. The methods in Symplectic Dynamics using pseudo-holomorphic curves have brought considerable insights to the study of Reeb flows, especially regarding the existence of periodic orbits and global surfaces of section. We shall discuss later some generalizations of Theorem \ref{thm:convexHWZ} for global surfaces of section with more than one boundary component.

Recall that global surfaces of section were first used by Poincar\'e \cite{Pc12} to prove the existence of infinitely many periodic orbits in the restricted circular planar three-body problem. He stated a general theorem for area-preserving twist maps of the annulus, that was ultimately proved by Birkhoff \cite{Bk13}, and became known as the Poincar\'e-Birkhoff Theorem. It asserts that if a homeomorphism of the closed annulus is homotopic to the identity map, preserves a finite area form, and twists the boundary components in opposite directions, then it must have at least two fixed points and thus infinitely many periodic orbits.

A direct benefit of a disk-like global surface of section is that the associated first return map preserves a finite area form and thus admits a fixed point. The orbit bounding the disk and the orbit corresponding to the fixed point form a Hopf link. Generalizations of the Poincar\'e-Birkhoff Theorem by Franks \cite{Fr90, Fr92} and Le Calvez \cite{Calvez06}, imply the existence of infinitely many periodic orbits provided a third periodic orbit exists. As a result, the flow has either two or infinitely many periodic orbits.

Finding global surfaces of section is in general a difficult task. Conley \cite{Con63}, McGehee \cite{McG69} and Kummer \cite{Kum1979} generalized Poincar\'e's works in the restricted three-body problem,  by finding global surfaces of section bounded by the retrograde and the direct orbits. Their methods, however, rely on perturbative arguments and thus are restricted to special situations. In contrast, the non-perturbative nature of Theorem \ref{thm:convexHWZ} widens the range of potential applications in Celestial Mechanics.

It is worth mentioning that checking the dynamical convexity of a concrete Reeb flow may also be intractable, except for some few special cases for which localizing periodic orbits and estimating their indices are feasible tasks. The alternative is to check whether the energy surface is strictly convex. To do that, one is led to some possibly intricate curvature estimates. In \cite{AFFHO12} and \cite{Sch}, strictly convex energy surfaces were found in the restricted three-body problem and the H\'enon-Heiles potential, respectively.  
More recent works on convexity of energy surfaces in Celestial Mechanics are found in  \cite{PS2,FKZ16,Kim2,Sa04}.
It is one of our goals to decide for every $\varpi,\alpha>0$ and $h<0$ whether the energy surface $\mathfrak{M}$ is strictly convex, convex or non-convex.


As mentioned above, if the Reeb flow admits a disk-like global surface of section, then there exists a pair of periodic orbits forming a Hopf link. In the dynamically convex case, it is even possible to show that the Hopf link bounds an annulus-like global surface of section. Alternatively, if dynamical convexity cannot be checked, the following linking condition is available: if every other periodic orbit is linked with the Hopf link, then it bounds an annulus-like global surface of section. Using this linking condition, we shall discuss the existence  of  an annulus-like global surface of section on $\mathfrak{M}$.

A non-resonance condition on the rotation numbers of the components of the Hopf link implies the twist condition of the annulus first return map and thus forces infinitely many periodic orbits by the Poincar\'e-Birkhoff Theorem. A version of the Poincar\'e-Birkhoff Theorem for Reeb flows on the tight three-sphere asserts that if the Hopf link satisfies a non-resonance condition, then there exist  infinitely many periodic orbits regardless it bounds a global surface of section. We shall discuss the existence of non-resonant Hopf links for large mass ratios.

As outlined above, the main purpose of this paper is to explore some dynamical aspects of the spatial isosceles three-body problem under the light of Symplectic Dynamics, see \cite{BH}.  In particular, we study the reduced Hamiltonian flow on the sphere-like energy surface $\mathfrak{M}$, which is equivalent to a Reeb flow on the tight three-sphere.  The Euler orbit in $z=0$ is a brake orbit, i.e., its velocity vanishes precisely twice along its period. Its mean index is shown to be greater than $4$. Since it  links with every other periodic orbit, it bounds a disk-like global surface of section. Using Birkhoff's shooting method, we find a $z$-symmetric brake orbit forming a Hopf link with the Euler orbit. This orbit is non-negatively linked with every other periodic orbit and its mean index is at least $2$. Hence the Hopf link bounds an annulus-like global surface of section (Theorem \ref{main1}). For large values of $\alpha$ and a suitable condition on $\varpi$, we check that the Hopf link is non-resonant (Theorem \ref{thm: nonresonant implies infy many}). This is accomplished by a suitable re-scaling of coordinates that leads to an integrable limiting system, where the non-resonance can be directly checked. We discuss how the non-resonance condition implies  infinitely many periodic orbits with different symmetries in the Hill region, such as brake orbits, $z$-symmetric orbits, etc (Theorem \ref{thm infy z brake}).  If the rotation number of the Euler orbit is rational, then there exist infinitely many periodic orbits of certain types in the Hill region (Theorem \ref{thm: rational rotation implies infy many}). Moreover, the Euler orbit is shown to be negative hyperbolic for an open set of parameters, while it can never be positive hyperbolic. The rotation number of the Euler orbit is then proved to be rational for a dense set of parameters (Theorem \ref{thm:Dk}). Finally, the convexity of the energy surface is addressed, and the parameters for which $\mathfrak{M}$ is strictly convex, convex, and non-convex are entirely identified (Theorem \ref{thm_convexity}). Combining it with some known facts from Symplectic Dynamics, we discuss some dynamical consequences of convexity.


\section{Main results}

The potential of the mechanical Hamiltonian $H=H(p_r,p_z,r,z)$ in \eqref{ham0} is denoted by
$$
V(r,z):=\frac{\varpi^2}{2r^2} - \frac{1}{r} -\frac{4\alpha^{-1}}{(r^2+(1+2\alpha)z^2)^{1/2}}, \quad (r,z) \in \R^+ \times \R.
$$
Here,  $(p_r,p_z,r,z)\in \R^2 \times \R^+ \times \R$ are canonical coordinates with symplectic form $\omega_0 = dp_r \wedge dr + dp_z \wedge dz$. The parameters $\varpi,\alpha >0$ are  referred to as the angular momentum and the mass ratio, respectively, see section \ref{reduce to two degree system} for more details.

The Hamiltonian vector field $X_H$ is determined by
$\omega_0(X_H, \cdot) = -dH,$
and Hamilton's equations become
\begin{equation}\label{ham}
\left\{\begin{aligned}
& \dot r   = p_r,\\
& \dot p_r  = \frac{\varpi^2}{r^3} -\frac{1}{r^2} -\frac{4\alpha^{-1}r}{(r^2+(1+2\alpha)z^2)^{3/2}},\\
& \dot z   = p_z,\\
 & \dot p_z   =  -\frac{4\alpha^{-1}(1+2\alpha)z}{(r^2+(1+2\alpha)z^2)^{3/2}}.
\end{aligned}
\right.
\end{equation}

The Hill region $\mathcal{H}\subset \R^2$ is the image of $\mathfrak{M}$ under the foot point projection $\Pi_{r,z}:(p_r,p_z,r,z)\mapsto (r,z),$ that is
$$
\mathcal{H} := \Pi_{r,z}(\mathfrak{M}) \subset \R^+ \times \R.
$$ Notice that $\mathcal{H} = \{(r,z): V(r,z) \leq h\}.$ Since $\mathfrak{M}$ is a sphere-like regular hypersurface,
 $$\partial \mathcal{H}= \{(r,z): V(r,z)=h\},$$ is a regular simple closed curve, called the zero velocity curve.
It is in one-to-one correspondence with the points $(p_r,p_z,r,z)\in \mathfrak{M}$ satisfying $p_r=p_z=0$. Every  $(r,z) \in \mathcal{H} \setminus \partial \mathcal{H}$ is the projection under $\Pi_{r,z}$ of an embedded circle in $\mathfrak{M}$, called the $S^1$-fiber over $(r,z)$, and  determined by $p_r^2 + p_z^2 = 2(h-V(r,z))>0$.

\begin{figure}[ht]
\centering
\includegraphics[width=0.9\textwidth]{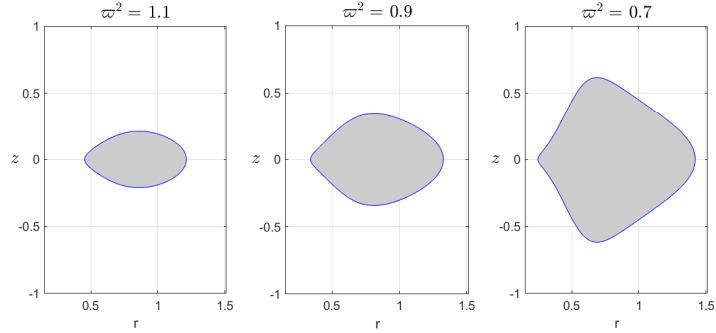}
 \caption{The Hill region for $\alpha=6$.}
 \label{hill regions}
\end{figure}

The Hamiltonian $H$ satisfies
$$
c^2H(c^{-1} p_r,c^{-1} p_z,c^2r,c^2z) = \widetilde H(p_r,p_z,r,z), \quad c>0,
$$
where $\widetilde H$ has the same form as $H$, with parameters $\alpha,\varpi$ replaced with $\alpha, \varpi c^{-1},$ respectively.  Hence the flow of $\widetilde H$ on $\widetilde H^{-1}(c^2h)$ is equivalent to the flow of $H$ on $H^{-1}(h), h<0$. This implies that we can once for all fix the energy
$$
h=-1.
$$

Since $H$ is a mechanical Hamiltonian, the hypersurface $\mathfrak{M}$ has contact type. This means that there exists a $1$-form $\lambda$ on $\mathfrak{M}$ satisfying $\lambda \wedge d\lambda \neq 0$, so that its Reeb vector field $R$, determined by  $d\lambda(R,\cdot) \equiv 0$ and $\lambda(R) =1$, is parallel to $X_H$.
The Reeb flow preserves $\lambda$ and, in particular, preserves the contact structure  $\xi = \ker \lambda\subset TM$. We shall make use of this extra structure to find periodic orbits and global surfaces of section.

A periodic trajectory $\zeta \subset \mathfrak{M}$ is called $z$-symmetric if its projection to the Hill region $\Pi_{r,z}(\zeta)\subset \mathcal{H}$ is symmetric with respect to the reflection $(r,z) \mapsto (r,-z)$.

A trivial knot $\zeta \subset (\mathfrak{M},\xi)$ is called a Hopf fiber if it is transverse to $\xi$ and its self-linking number is $-1$, that is the linking number between $\zeta$ and a perturbation of $\zeta$ by a non-vanishing constant section induced by a global trivialization of $\xi$ is $-1$. 
A link $L= \zeta \cup \zeta'\subset (\mathfrak{M},\xi)$ is called a Hopf link if it is formed by a pair of simply linked Hopf fibers $\zeta,\zeta'$.

We say that a periodic orbit $\zeta\subset \mathfrak{M}$ is a brake orbit if its projection  $\hat \zeta= \Pi_{r,z}(\zeta)$ to $\mathcal{H}$ intersects the zero velocity curve $\partial \mathcal{H}$. As a special case, the periodic orbit $\zeta_e \subset \mathfrak{M}$  projecting to  $\mathcal{H} \cap \{z=0\}$ is a simple brake orbit, called the Euler orbit. Every brake orbit necessarily intersects $\partial \mathcal{H}$ at precisely two distinct points. We call a brake orbit $\zeta$ simple if $\hat \zeta$ does not self-intersect. Any simple brake orbit is a Hopf fiber. A pair of simple brake orbits $\zeta,\zeta'$ whose projections $\hat \zeta,\hat \zeta'\subset \mathcal{H}$ intersect precisely at a single  point form a Hopf link on $(\mathfrak{M},\xi)$.

We denote by $\hat i(\zeta)$ the mean index of a periodic orbit $\zeta\subset \mathfrak{M}$. The rotation number of $\zeta$ is denoted
$$
\rho(\zeta):=\hat i (\zeta)/2.
$$
For a mechanical system, $\rho(\zeta)$ is always non-negative, and if $\zeta$ is a brake orbit, then $\rho(\zeta)\geq 1$. As we show in Proposition \ref{prop: rho_e}, the rotation number $\rho_e$ of the Euler orbit $\zeta_e$ is greater than $2$.

\begin{defi} \label{defi: GlobalSurface} Let $F\hookrightarrow \mathfrak{M}$ be an embedded compact surface. We say that $F$ is a global surface of section for the Hamiltonian flow  on $\mathfrak{M}$  if
\begin{enumerate}
\item[(i)] $X_H$ is tangent to $\partial F$, the boundary of $F$;
\item[(ii)] $X_H$ is transverse to $F^o,$ the interior of $F$;
\item[(iii)] The trajectory through every point $x \in \mathfrak{M} \setminus \partial F$ intersects $F^o$ infinitely many times forward and backward in time.
\end{enumerate}
\end{defi}
When a global surface of section $F\hookrightarrow \mathfrak{M}$ exists, the dynamics on $\mathfrak{M}$ is encoded by the diffeomorphism $g: F^o \to F^o$ given by the first return map. We shall see a global surface of section as a page of an open book decomposition of $(\mathfrak{M},\xi)$.

\begin{defi}\label{defi: Openbook}
An open book decomposition of $(\mathfrak{M},\xi)$ is a pair $\mathcal{O}=(B,\Phi),$ where  $B\subset (\mathfrak{M},\xi)$ is a transverse link, called the binding of $\mathcal{O}$, and $\Phi: \mathfrak{M}\setminus B\rightarrow S^1$ is a fibration so that each page $\Phi^{-1}(\theta),\ \theta\in S^1,$ is a properly embedded surface in $\mathfrak{M}\setminus B$ whose closure is a  compact embedded surface with boundary $B$. Moreover, there exists a defining contact form $\alpha$ on $(\mathfrak{M},\xi)$ so that its Reeb vector field $R_\alpha$ is transverse to every page $\Phi^{-1}(\theta)$, tangent to the binding $B$, and the orientation on $B$ induced by $R_\alpha$ coincides with the orientation induced by the pages. Here, $\mathfrak{M}$ is oriented by $\alpha \wedge d\alpha >0$, and the pages are co-oriented by $R_\alpha$.
\end{defi}

Our first result concerns the existence of brake-orbits and open book decompositions whose pages are  global surfaces of section.

\begin{thm}\label{main1}Assume that $\varpi,\alpha>0$ and $h<0$ satisfy   \eqref{cond1}, and let $\zeta_e\subset \mathfrak{M}\cap \{z=0\}$ be the Euler orbit.  Then
\begin{itemize}
    \item[(i)] $\zeta_e$ is the  binding of an open book decomposition of $(\mathfrak{M},\xi)$ whose pages are disk-like global surfaces of section.
    \item[(ii)] There exists a simple $z$-symmetric  brake orbit $\zeta_{z}\subset \mathfrak{M}$ so that the Hopf link $\zeta_e \cup \zeta_{z}\subset (\mathfrak{M}, \xi)$  is the binding of an open book decomposition  whose pages are annulus-like global surfaces of section.
\end{itemize}
\end{thm}

Although Theorem \ref{main1} can be proved using standard geometric methods, we shall present some more sophisticated results in Symplectic Dynamics that provide the desired open books as projections of finite energy foliations in the symplectization of $\mathfrak{M}$.


Our second result is about the existence of non-resonant Hopf links.

\begin{defi}Let $L=L_1\cup L_2\subset (\mathfrak{M},\xi = \ker \lambda)$ be a Hopf link formed by Reeb orbits whose linking number is $+1$. Let $\rho_1,\rho_2\in \R$ be the respective rotation numbers of $L_1,L_2$. We say that $L$ is non-resonant if $\rho_1^{-1} + \rho_2^{-1} \neq 1,$ or equivalently,
\begin{equation}\label{non-resonance}
\rho_1 - 1 \neq (\rho_2 - 1)^{-1}.
\end{equation}
\end{defi}

For mass ratio $\alpha \gg 0$, we shall prove the existence of an arbitrary number of non-resonant Hopf links as in Theorem \ref{main1}. The non-resonance condition \eqref{non-resonance} implies the existence of infinitely many periodic orbits with prescribed linking numbers with the components of the Hopf link. In fact, this condition is equivalent to a twist condition of the first return map to a page of the open book bounded by the Hopf link.

The following theorem considers large values of $\alpha$, i.e., the symmetric bodies are much heavier than the body  moving along the symmetry axis. As $\alpha$ gets larger,
 the number of brake orbits forming a non-resonant Hopf link with the Euler orbit is arbitrarily large.

\begin{thm}\label{thm: nonresonant implies infy many}
Let $N\in \N$ and  $C>1$ be large numbers, and let $\epsilon>0$ be small. Then
\begin{itemize}
    \item[(i)] There exists $\alpha_0=\alpha_0(N,C,\epsilon)>0,$ so that if $\alpha> \alpha_0$, and
     \begin{equation}\label{cond2}
    1<2\varpi^2<1+\alpha^{-2},
    \end{equation}
     then the energy surface $\mathfrak{M}$ admits at least $N$ distinct simple $z$-symmetric brake orbits $\zeta_{z_1},\ldots,\zeta_{z_N}$ so that
     $L_i = \zeta_{z_i} \cup \zeta_e$ is a Hopf link for every $i.$
     The rotation number of $\zeta_{z_i}$ satisfies
     $\rho_{z_i} > C,\  \forall i,$
     and the rotation number  of the Euler orbit satisfies
     $|\rho_e - (2\sqrt{2}+1)| < \epsilon.$
     In particular, $L_i$ is  non-resonant for every $i$.
     \item[(ii)] Let $\zeta_{z_i}$ be as in (i).  For every co-prime  integers $p,q>0$  satisfying
$$
 (\rho_{z_i}-1)^{-1}< p/q < \rho_e-1,
$$
there exists a periodic orbit $\zeta_{p,q} \subset \mathfrak{M}$ so that
$$
\lk(\zeta_{p,q},\zeta_e) = p \quad \mbox{ and } \quad \lk(\zeta_{p,q},\zeta_{z_i}) = q.
$$
\end{itemize}
\end{thm}

The proof of Theorem \ref{thm: nonresonant implies infy many} relies on  a suitable re-scaling of the variables $(p_r,p_z,$ $r,z)$ so that the new system of equations admits a limit as $\alpha \to +\infty$. The periodic orbits in Theorem \ref{thm: nonresonant implies infy many}-(i) are then obtained as natural continuations of similar periodic orbits of the limiting system.

Although the Hill region becomes unbounded in the $z$-direction as $\alpha$ goes to $+\infty$, the $z$-symmetric brake orbits obtained in Theorem \ref{thm: nonresonant implies infy many}-(i) are mostly concentrated near $z=0$. The periodic orbits $\zeta_{p,q}$ in Theorem \ref{thm: nonresonant implies infy many}-(ii) follow from the Poincar\'e-Birkhoff Theorem applied to the first return map of the annulus-like global surface of section bounded by the Hopf link. However, to obtain such orbits we shall instead  apply a more general version of the Poincar\'e-Birkhoff Theorem for Reeb flows on the tight three-sphere, which does not make use of any global surface of section.

Next,  we explore the reversibility and the $z$-symmetry of the equations of motion in \eqref{ham} and discuss the existence of new types of periodic orbits. We set
$$ \mathcal{B}_+ := \partial\mathcal{H} \cap \{(r, z): z \ge 0 \}, \quad
\mathcal{B}_- := \partial\mathcal{H} \cap \{(r, z): z \le 0 \}. $$
\begin{defi}
\label{defi: BrakeOrbits} Let $\zt(t)\in \mathfrak{M}$ be a brake orbit, $\hat \zeta(t) :=\Pi_{r,z}(\zeta(t)) \in \mathcal{H}$,  with $\hat \zt(a), \hat \zt(b) \in \partial \mathcal{H}$ and $\hat \zt(t) \notin \partial \mathcal{H}$, for every $t \in (a, b)$.
\begin{enumerate}
\item[(i)] We say that $\zt(t)$ is a type-I brake orbit, if $\hat \zt(a) \in \cb_+$ and $\hat \zt(b) \in \cb_-$, or $\hat \zt(a) \in \cb_-$ and $\hat \zt(b) \in \cb_+$.
\item[(ii)] We say that $\zt(t)$ is a type-II brake orbit, if $\hat \zt(a) \in \cb_+$ and $\hat \zt(b) \in \cb_+$, or $\hat \zt(a) \in \cb_-$ and $\hat \zt(b) \in \cb_-$.
\end{enumerate}
\end{defi}
Notice that, except for the Euler orbit, which is both of type I and II, every brake orbit is either of type I or II.

\begin{figure}[ht]
\centering
\includegraphics[width=0.7\textwidth]{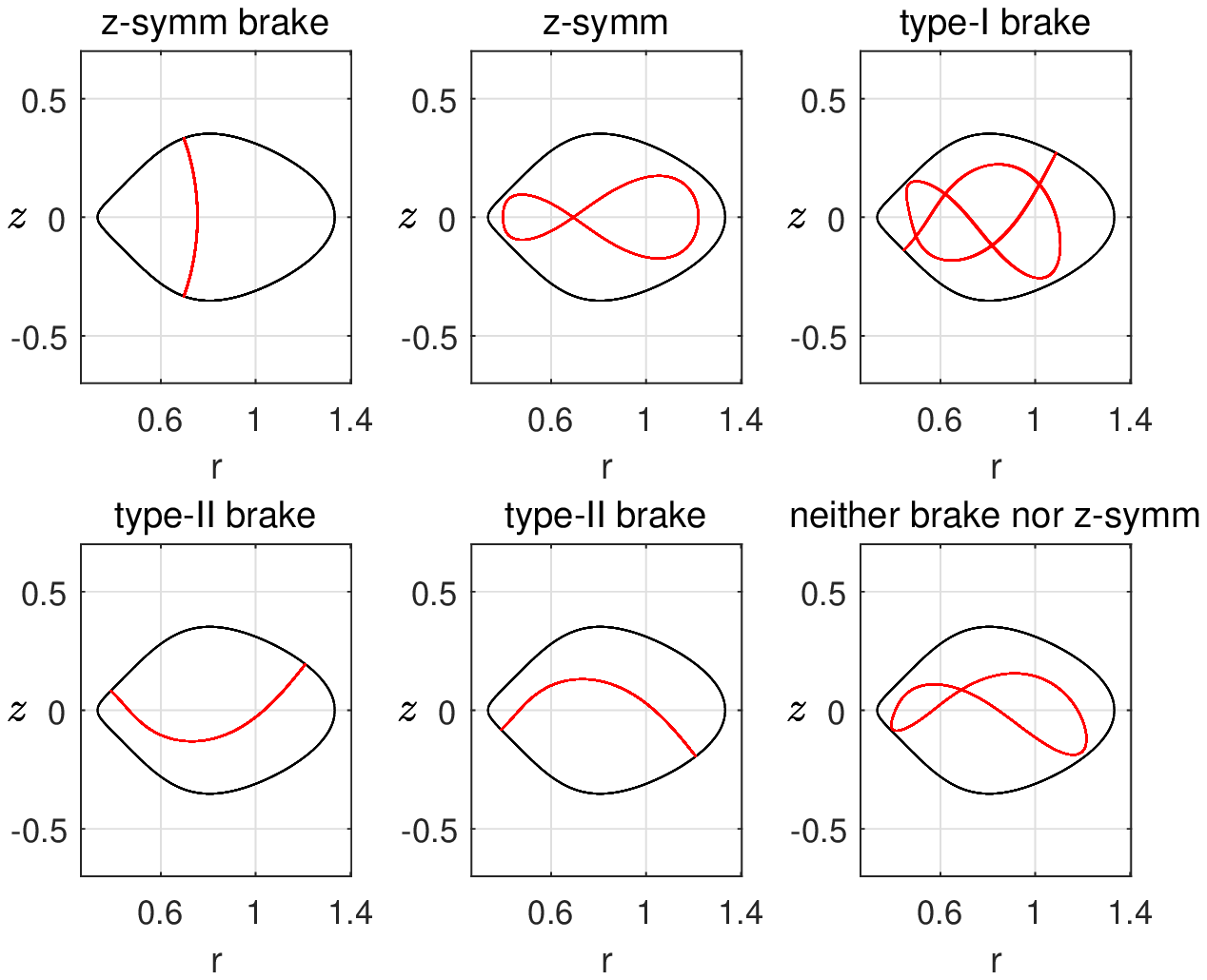}
 \caption{Different types of periodic orbits for $\beta = \mathfrak{e}=0.6$.}
 \label{picture of symmetric orbits}
\end{figure}

\begin{thm}\label{thm infy z brake}
Let $\zeta_z\subset \mathfrak{M} \setminus \zeta_e$ be a  $z$-symmetric simple brake orbit. If the Hopf link $\zt_e \cup \zeta_z$ is non-resonant,   then
\begin{enumerate}
\item[(i)] There exist infinitely many $z$-symmetric brake orbits;
\item[(ii)] There exist infinitely many type-I brake orbits which are not $z$-symmetric;
\item[(iii)] There exist infinitely many  type-II brake orbits;
\item[(iv)] There exist infinitely many $z$-symmetric periodic orbits that are not brake orbits.
\end{enumerate}
\end{thm}

The proof of Theorem \ref{thm infy z brake} relies on an intersection argument involving the iterates of certain curves in the disk-like global surface of section bounded by the Euler orbit, and the fact that the first return map satisfies a twist condition.

 Before stating the next result, we consider new parameters $\beta, \mathfrak{e}\in (0,1),$ defined as
$$
\beta := (1+4\alpha^{-1})^{-1} \ \ \ \ \mbox{ and } \ \ \ \ \mathfrak{e} := (1-2\varpi^2 \beta^2)^{1/2}.
$$
They are in one-to-one correspondence with $\alpha$ and $\varpi$.  Conditions \ref{cond1} with $h=-1$ are equivalent to requiring that the pair $(\beta,\mathfrak{e})$ belongs to
$$
\mathcal{D}=\{(\beta,\mathfrak{e}): 0<\beta^2 + \mathfrak{e}^2 <1,  \beta,\mathfrak{e}>0\}.
$$

\begin{thm}\label{thm: rational rotation implies infy many}
Let $\rho_e >0$ be the rotation number of $\zeta_e$. Then
\begin{itemize}
    \item[(i)] If $\rho_e \in \Q$ then  $\mathfrak{M}$ admits infinitely many $z$-symmetric orbits.
    \item[(ii)] If $\rho_e = p/q\in \Q$, where $p>0$ is odd and $q>0$ is even, then $\mathfrak{M}$ admits infinitely many $z$-symmetric brake orbits.
    \end{itemize}
\end{thm}

The existence of infinitely many periodic orbits as in Theorem \ref{thm: rational rotation implies infy many} can be proved for $(\beta,\mathfrak{e})$ in a dense subset of $\mathcal{D}$  containing a non-empty open set.
To prove it, we investigate the stability of the Euler orbit and find an open subset of parameters where the Euler orbit is negative hyperbolic.

\begin{thm}\label{thm:Dk}The following statements hold:
\begin{itemize}
     \item[(i)] For each  integer $k\geq 1$, the set $$
   \hat{\mathcal{E}}_k:=\{(\beta,\mathfrak{e})\in (0,1)\times [0,1): \rho_e = k \mbox{ and } \zeta_e \mbox{ is hyperbolic}\}$$
    is empty.

    \item[(ii)]
   For each integer $k\geq 2$, the set $$
   \mathcal{D}_{k+1/2}:=\{(\beta,\mathfrak{e})\in (0,1)\times [0,1): \rho_e = k+1/2 \mbox{ and } \zeta_e \mbox{ is hyperbolic}\}$$
    contains a non-empty open subset of $[0,1]\times [0,1)$. Moreover, both $\mathcal{D}_{5/2}$ and $\mathcal{D}_{7/2}$ contain non-empty open subsets of $\mathcal{D}.$

 \item[(iii)] The subset $\mathcal{D}_{\mathcal{R}}:=\{(\beta,\mathfrak{e})\in \mathcal{D}: \rho_e \in \Q\}$  is dense in $\mathcal{D}$ and contains a non-empty open set. Moreover, the subset $\mathcal{D}_{{\rm odd / even}}\subset \mathcal{D}$ of parameters $(\beta,\mathfrak{e})$ for which $\rho_e = p/q\in \Q,$ with $p>0$ odd and $q$ even, is dense in $\mathcal{D}$ and contains a non-empty open set.

    \end{itemize}
\end{thm}

\begin{figure}[ht]
\centering
\includegraphics[width=0.7\textwidth]{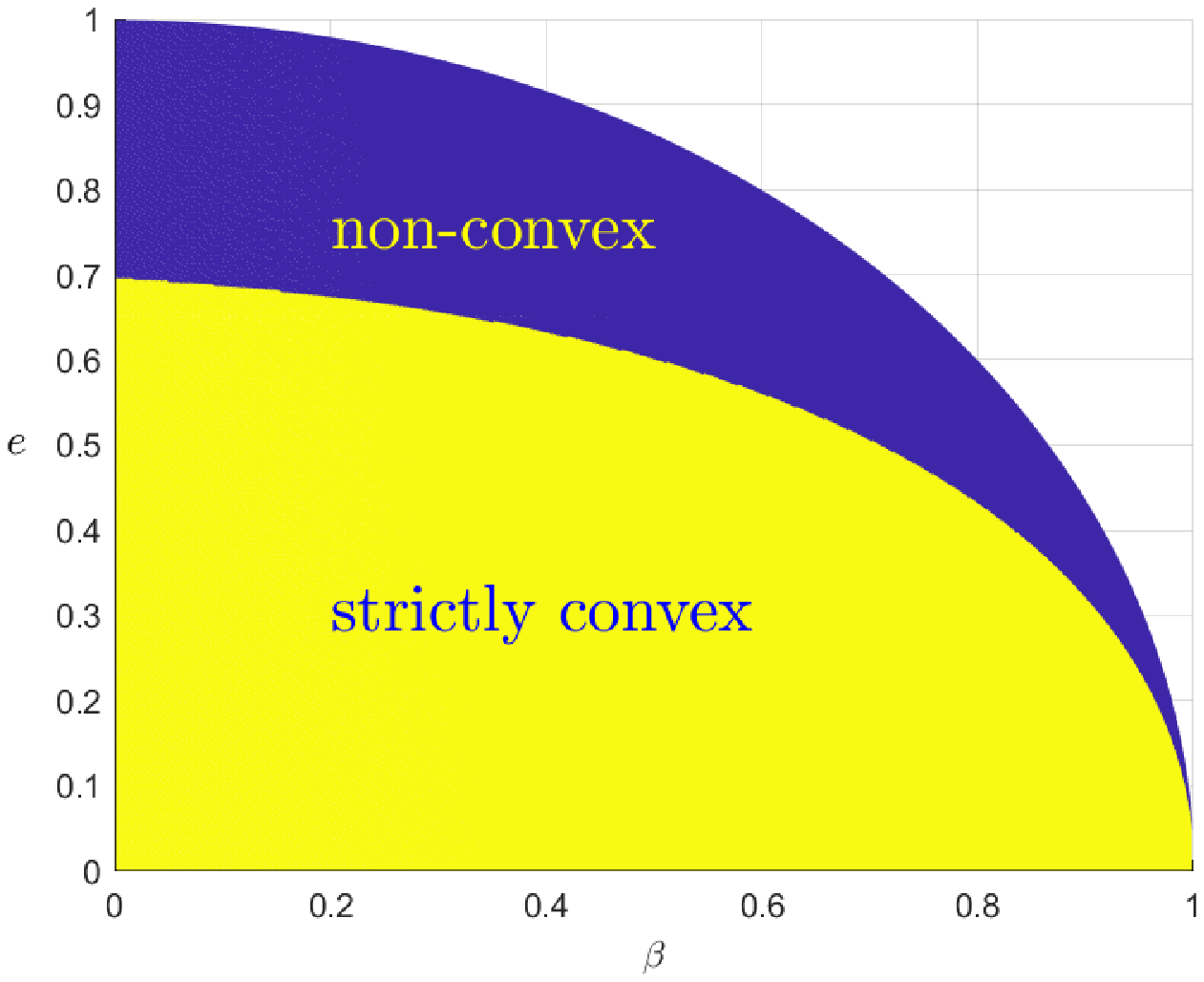}
 \caption{The range of convexity in the $(\beta,\mathfrak{e})$-coordinates.}
 \label{picture of convexity}
\end{figure}

Combining Theorems \ref{thm: rational rotation implies infy many} and \ref{thm:Dk} we obtain the following corollary.

\begin{cor}
    The set of parameters $(\beta,\mathfrak{e})\in \mathcal{D}$ for which $\mathfrak{M}$ carries infinitely many periodic orbits is dense and contains a non-empty open subset.
\end{cor}

Finally, we study the range of parameters for which the energy surface $\mathfrak{M}$ is strictly convex, convex, and non-convex. For each fixed $\beta \in (0,1)$, we show that there exists a special value
 $\mathfrak{e}_{\rm conv}(\beta)\in (0,1)$ so that $\mathfrak{M}$ is strictly convex if $\mathfrak{e} < \mathfrak{e}_{\rm conv}(\beta)$ and non-convex if $\mathfrak{e} > \mathfrak{e}_{\rm conv}(\beta)$, see Figure \ref{picture of convexity}.

\begin{thm}\label{thm_convexity} There exists a strictly decreasing and concave function $\mathfrak{e}_{\text{conv}}:[0,1] \to [0,1]$,  with
 $
 \mathfrak{e}_{\text{conv}}(0)= (7+\sqrt{17})/16$, $\mathfrak{e}_{\text{conv}}(1) =0$, and $(\beta,\mathfrak{e}_{\text{conv}}(\beta))\in \mathcal{D}, \forall \beta,$
 so that if $(\beta,\mathfrak{e})\in \mathcal{D}$, then
$$
\begin{aligned}
    \mathfrak{e} & < \mathfrak{e}_{\text{conv}}(\beta) \ \ \Leftrightarrow \ \ \mathfrak{M}  \mbox{ is strictly convex.}\\
     \mathfrak{e} & = \mathfrak{e}_{\text{conv}}(\beta) \ \ \Leftrightarrow \ \ \mathfrak{M}  \mbox{ is convex but not strictly convex.}\\
     \mathfrak{e} & > \mathfrak{e}_{\text{conv}}(\beta) \ \ \Leftrightarrow \ \ \mathfrak{M}  \mbox{ is not convex.}
\end{aligned}
$$
\end{thm}

Combining Theorem \ref{thm_convexity} with the main results from \cite{Hr, HS2,HSW} we obtain the following corollary.

\begin{cor}Let $(\beta,\mathfrak{e})\in \mathcal{D}$ satisfy $\mathfrak{e} < \mathfrak{e}_{\text{conv}}(\beta)$. Then
\begin{itemize}
\item[(i)] Every periodic orbit that is a Hopf fiber   is the binding of an open book decomposition whose pages are disk-like global surfaces of section. In particular, this holds true for any simple brake orbit.
\item[(ii)] Every pair of periodic orbits forming a Hopf link is the binding of an open book decomposition whose pages are annulus-like global surfaces of section. In particular, this holds true for a pair of brake orbits whose projections to the Hill region intersect precisely at a single point.
\end{itemize}
\end{cor}

\section{The spatial isosceles three-body problem}\label{reduce to two degree system}

\begin{figure}[ht]
\centering
\includegraphics[width=0.7 \textwidth]{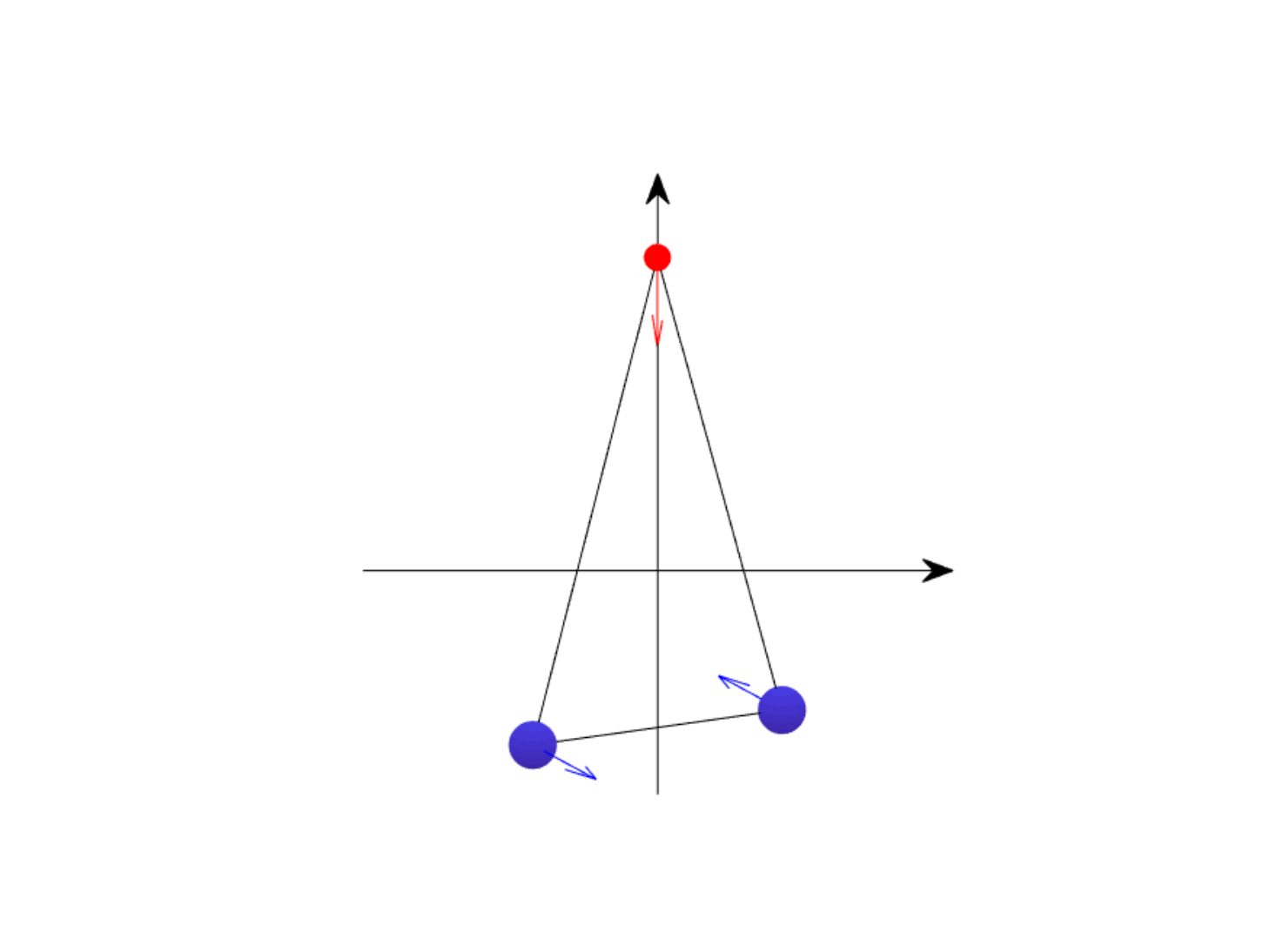}
\caption{The spatial isosceles three-body problem.}
\label{fig:IsoThree}
\end{figure}

Let $m_i>0$ and $q_i\in \rr^3$ be the mass and the position of the $i$-th body, $i=1,2,3$. Assume that $m:=m_1=m_2$, that $q_1 \neq q_2$ are symmetric with respect to the $z$-axis, and that $q_3$ moves along the $z$-axis. Let $q_1=(x, y,z) \in \rr^3$ and $\dot{q}_1=(\dot x, \dot y, \dot z)$. Then $q_2=(-x, -y, z)$ and $\dot q_2=(-\dot x, -\dot y, \dot z)$. If  the center of mass stays fixed at the origin, then $q_3=(0, 0, -2 \al z)$ and $\dot q_3=(0, 0, - 2\al \dot z)$, where $\al = m/m_3$ is the mass ratio.

The motion of the first body entirely determines the dynamics of the system, and the Lagrangian writes as
\begin{equation}
\label{eq_Lagxyz} \begin{aligned}
L(q,\dot q) & = m \big( \dot x^2 + \dot y^2 + (1+2\al) \dot z^2 \big) \\
& + \ey m^2\left(\frac{1}{(x^2+y^2)^{1/2}}+\frac{4\al^{-1}}{(x^2 +y^2 +(1+2\al)^2 z^2)^{1/2}} \right),
\end{aligned}
\end{equation}
where $q=(x,y,z)$.

Introduce the diagonal matrix $M=\text{diag}(2m, 2m, 2m(1+2\al))$, and let
$$  \xi = \left( \begin{array}{c}
\xi_1 \\ \xi_2 \\ \xi_3
\end{array} \right)
:= M^{\ey} \left( \begin{array}{c}
x \\ y \\ z
\end{array} \right) \; \text{ and } \; \dot{\xi} = \left( \begin{array}{c}
\dot \xi_1 \\ \dot \xi_2 \\ \dot \xi_3
\end{array} \right)
= M^{\ey} \left( \begin{array}{c}
\dot x \\ \dot y \\ \dot z
\end{array} \right).
$$
The previous Lagrangian becomes
\begin{equation}
\label{eq_LagXi} L(\xi, \dot \xi) = \ey |\dot \xi|^2 +  \frac{m^{5/2}}{\sqrt{2}}
 \left( \frac{1}{(\xi_1^2 + \xi_2^2)^{1/2}} + \frac{4 \al^{-1}}{(\xi_1^2 +\xi_2^2+ (1+2\al)\xi_3^2)^{1/2}} \right)
\end{equation}

Due to the rotational symmetry with respect to the $z$-axis, it is more convenient to work with cylindrical coordinates $(r, \tht,z) \in \R^+ \times \rr/ 2\pi \zz \times \rr$, defined by
$$ \xi_1= r\cos \tht, \ \ \ \ \ \xi_2 = r  \sin \tht, \ \ \ \ \ \xi_3 =z. $$
This further transfers the Lagrangian to
$$ L(r, \tht, z, \dot{r}, \dot{\tht}, \dot{z}) = \ey ( \dot r^2 +\dot z^2 + r^2 \dot \tht^2 ) + U(r,z), $$
where the $\theta$-independent potential function is
\begin{equation}
\label{eq_Urz} U(r,z) = \frac{m^{5/2}}{\sqrt{2}}  \left( \frac{1}{r} + \frac{4 \al^{-1}}{(r^2 + (1+2\al)z^2)^{1/2}} \right).
\end{equation}
 By the Legendre transform
\begin{equation}
\label{eq_Legendre} \pr = \dot r, \ \ \ \ \  p_z = \dot{z}, \  \ \ \ \ p_{\tht} = r^2 \dot \tht,\end{equation}
we get the corresponding Hamiltonian
\begin{equation}
\label{eq_Ham}  H(\pr, p_z, p_{\tht}, r, z, \tht) = \ey \left(\pr^2 + p_z^2 + p_{\tht}^2r^{-2}  \right)- U(r,z).
\end{equation}

Since  $H$ is independent of $\tht$, the angular momentum $p_\theta$ is constant along the trajectories. As a result, we fix it as a parameter
\begin{equation}
\label{eq: AngMom} \varpi := p_{\tht} = r^2 \dot \tht,
\end{equation}
 and obtain a two-degree-of-freedom mechanical Hamiltonian
\begin{equation}
\label{eq_HamOm} H =  \frac{\pr^2 + p_z^2}{2} + V(r, z), \quad \quad  (\pr, p_z, r, z) \in \rr^2 \times \rr^+ \times \rr,
\end{equation}
with corresponding potential
\begin{equation}
\label{eq_Vrphi} V(r, z) =\frac{\varpi^2}{2r^2 } - \frac{m^{5/2}}{\sqrt{2}}  \left( \frac{1}{r} + \frac{4 \al^{-1}}{(r^2 + (1+2\al)z^2)^{1/2}} \right).
\end{equation}


Multiplying $V$ by a positive constant does not change the structure of the dynamics on corresponding energy surfaces. Hence we can assume that $m^{5/2}/\sqrt{2}=1$, and the potential becomes
\begin{equation}\label{eq_Vrphi1}
V(r, z) =\frac{\varpi^2}{2r^2 }-\frac{1}{r}-\frac{4 \al^{-1}}{(r^2 + (1+2\al)z^2)^{1/2}}, \quad \quad \forall (r,z)\in \R^+ \times \R,
\end{equation}
where the new parameter $\varpi^2\sqrt{2}/m^{5/2}$ is still denoted by $\varpi^2$, and the energy $h$ is replaced with $hm^{5/2}/\sqrt{2}$.

\begin{prop}[Moeckel { \cite[Proposition 2.1]{Mk84}}]
\label{prop_Mcomp} Let $\varpi,\alpha>0$ and $h<0$, and let $\mathfrak{M}=H^{-1}(h)\subset \R^4$, where $H$ is the Hamiltonian in \eqref{eq_HamOm}, with potential $V$  as in \eqref{eq_Vrphi1}.  The following assertions hold:
\begin{itemize}
\item[(i)] If $1 < 2\varpi^2|h| < (1+4\alpha^{-1})^2$,
then $\mathfrak{M}$ is a sphere-like regular hypersurface.
\item[(ii)] If $2\varpi^2|h|>(1+4\alpha^{-1})^2$, then $\mathfrak{M}=\emptyset.$

\item[(iii)] If  $2\varpi^2|h|=(1+4\alpha^{-1})^2$, then $\mathfrak{M}=\{(0,0,\varpi^2(1+4\alpha^{-1})^{-1},0)\}$.

\item[(iv)] If $2\varpi^2|h|\leq 1$, then  $\mathfrak{M}$ is unbounded in $z$ direction.
\end{itemize}
\end{prop}
\begin{proof}
If $(p_r,p_z,r,z)\in \mathfrak{M}$ then $V(r,z) \leq h$. Conversely, if $V(r,z)\leq h,$ there exists $(p_r,p_z)\in \R^2$ so that $(p_r,p_z,r,z) \in \mathfrak{M}$.

For fixed $\varpi,\alpha>0$, the function $V$ satisfies the following properties that can be easily checked:
\begin{itemize}
    \item[(a)] For any fixed $r>0$, the function $z\mapsto V(r,z),z\in \R,$ is $z$-symmetric, and strictly increasing on $[0,+\infty).$

    \item[(b)] For any fixed $z\in \R$, the function $r\mapsto V(r,z), r>0,$ has a unique critical point $r_z>0$, which is a global minimum. The function $z\mapsto r_z$ is $z$-symmetric and strictly increasing on $[0,+\infty)$.

    \item[(c)]  $V(r,0) =\varpi^2/(2r^2)-(1+4\alpha^{-1})/r$ is such that $V(r,z) > V(r,0),$ for every $z\neq 0$.  Moreover, $V(r,0)$ attains its minimum value $-(1+4\alpha^{-1})^2/(2\varpi^2)$  precisely at $r_0:=\varpi^2(1+4\alpha^{-1})^{-1}$.

    \item[(d)] $V(r,+\infty):= \varpi^2/(2r^2)-1/r, r>0,$ is such that $V(r,z) < V(r,+\infty), \forall z\in \R,$ and its minimum value $-1/(2\varpi^2)$ is attained at $r_\infty:=\lim_{z\to +\infty} r_z = \varpi^2>r_0$.

    \item[(e)] $\lim_{r \to 0^+} V(r,z)=+\infty$ and $\lim_{r \to +\infty} V(r,z) = 0^-,$ uniformly in $z$.

    \item[(f)] $(r_0,0)$ is a global nondegenerate minimum point of $V$. Indeed, $\partial_{rz}V(r_0,0)=0$ and $\partial_{rr}V(r_0,0) \partial_{zz}V(r_0,0)>0$. In particular, for every $h$ in the interval $(-(1+4\alpha^{-1})^2/(2\varpi^2),-1/(2\varpi^2)),$ the level set $\{V=h\}$ is a bounded smooth $z$-symmetric simple closed curve.
\end{itemize}

From the above properties of $V$, we immediately conclude that
$$
\mathfrak{M} = \emptyset\ \Leftrightarrow\ V(r_0,0) > h \Leftrightarrow -(1+4\alpha^{-1})^2/(2\varpi^2)>h.
$$

If  $2\varpi^2|h| = (1+4\alpha^{-1})^2,$ then $V>h$ except at $r=r_0=\varpi^2(1+4\alpha^{-1})^{-1}$ and $z=0$. In that case, $\mathfrak{M} = \{(0,0,r_0,0)\}$.

If $2\varpi^2|h|\leq 1$, then $V(r_\infty,+\infty)\leq h$. This implies that $V(r_\infty,z)< h$ for every $z\in\R$ and thus $\mathfrak{M}$ is unbounded in $z$.

Finally, assume that $1 < 2\varpi^2|h| < (1+4\alpha^{-1})^2$. Similar to $V$, the only critical point of $H$ is $(0,0,r_0,0)$, which is a nondegenerate global minimum. We have $H(0,0,r_0,0) = -(1+4\alpha^{-1})^2/(2\varpi^2).$ By standard Morse theory, as $h$ increases from below $-(1+4\alpha^{-1})^2/(2\varpi^2)$ to slightly above $-(1+4\alpha^{-1})^2/(2\varpi^2)$, $\mathfrak{M}=H^{-1}(h)$ changes from the empty set to a sphere-like regular hypersurface. From (f), $\mathfrak{M}$ is bounded and thus a sphere-like hypersurface for every $h$ in that interval.
\end{proof}

\section{Reeb flows and open book decompositions}

Let $\lambda$ be a contact form on a closed connected three-manifold $M$, i.e., $\lambda$ is a $1$-form on $M$ so that  $\lambda \wedge d\lambda$ never vanishes. Denote by $\xi := \ker \lambda\subset TM$ the contact structure, and by $R_\lambda$ the Reeb vector field of $\lambda$,  determined by
$$
d\lambda(R_\lambda,\cdot) \equiv 0 \quad \mbox{ and } \quad \lambda(R_\lambda)\equiv 1.
$$
The flow of $R_\lambda$, denoted $\varphi_t,t\in \mathbb{R},$ is called the Reeb flow of $\lambda$. The pair $(M,\xi = \ker \lambda)$ is a co-oriented contact manifold.

Denote by $\mathcal{P}(\lambda)$ the set of periodic orbits of $\varphi_t$, identifying those which differ by a time shift and have the same period. A periodic orbit (also called a Reeb orbit) is then a pair $P=(x,T)$, where $x:\mathbb{R} \to M$ is a periodic trajectory of $\varphi_t$ and $T>0$ is a period of $x$. Let $x_T:=x(T \cdot).$
We say that $P$ is unknotted if $T$ is the smallest positive period of $x$ and $x(\mathbb{R})\subset M$ is a trivial knot, i.e., there exists an embedding $u:\mathbb{D} \to M$ so that $e^{i2 \pi t} = x_T(t), \forall t\in [0,1].$ Here, $\mathbb{D}=\{z\in \mathbb{C}:|z| \leq 1\}$ is the closed unit disk.   The mapping $u$ is called a capping disk for $P$.

We shall consider the tight $3$-sphere $(S^3,\xi_0)$. Here,
$$
S^3 = \{|z_1|^2+|z_2|^2=1\}\subset \mathbb{C}^2 \equiv \mathbb{R}^4,
$$
where $z_i=x_i+\sqrt{-1}y_i\in \mathbb{C}, i=1,2,$ are complex coordinates, the standard symplectic form is $\omega_0 = dx_1 \wedge dy_1 + dx_2 \wedge dy_2,$
and the contact structure is $$\xi_0:=\ker \lambda_0\subset TS^3,$$ where $\lambda_0$ is the contact form on $S^3$ given by the restriction of the Liouville form
$\frac{1}{2}\sum_{i=1}^2 x_i dy_i- y_i dx_i$ to $TS^3.$ The Reeb vector field of $\lambda_0$ is $R_{\lambda_0}(z) = 2 \sqrt{-1} z=2(-y_1+\sqrt{-1}x_1,-y_2+\sqrt{-1}x_2),  \forall z \in S^3,$ and all of its trajectories are periodic with least period $\pi$.

We consider contact forms on $S^3$ of the form  $\lambda=f\lambda_0$, where $f:S^3 \to \mathbb{R}^+$ is smooth. Notice that $\ker \lambda = \ker \lambda_0$ and $d\lambda|_{\xi_0} = f d\lambda_0|_{\xi_0}.$ The bundle $(\xi_0,d\lambda) \to S^3$ is symplectic and admits a symplectic trivialization $\Phi_\lambda: (\xi_0,d\lambda) \to S^3 \times (\mathbb{C}, dx \wedge dy)$. Let $K\subset (S^3,\xi_0)$ be a trivial knot transverse to $\xi_0$, and let $X$ be a non-vanishing section of $\xi_0$. Use $X$ and an exponential map to push $K$ to a trivial knot $K'$ which does not intersect $K$, and is transverse to a capping disk  $u:\mathbb{D} \to S^3$ for $K$. We define the self-liking number $\mathrm{sl}(K)$ as the algebraic intersection number
$$K' \cdot u \in \mathbb{Z},$$  where $K'$  has the orientation induced by $K$,  $u$ is oriented by $dx \wedge dy$, and $S^3$ is oriented by $\lambda \wedge d\lambda>0$.

\begin{defi}We say that a smooth knot $K \subset (S^3,\xi_0)$ is a Hopf fiber if $K$ is a trivial knot transverse to $\xi$ and its self-linking number is $-1$.  \end{defi}

Let $S^1 = \{z\in \mathbb{C}:|z|=1\}$. The Hopf fiber $L_1 =S^1 \times \{0\} \subset \mathbb{C}^2$ is the binding of an open book decomposition $\mathcal{O}_1$ whose disk-like pages are
$$
\begin{aligned}
u_c:{\mathbb{D}}^o & \to S^3 \setminus L_1,  \ c\in S^1, \\
z & \mapsto (z,c \sqrt{1-|z|^2}),
\end{aligned}
$$
where ${\mathbb{D}}^o = \mathbb{D} \setminus S^1$. To see that $\mathrm{sl}(L_1)=-1$, observe first that the contact structure $\xi_0$ is spanned by the $d\lambda_0$-symplectic basis $\{X_1,X_2\}$ given by
$$
\begin{aligned}
X_1 & = x_2 \partial_{x_1} - y_2 \partial_{y_1} - x_1 \partial_{x_2} + y_1 \partial_{y_2},\\
X_2 & = y_2 \partial_{x_1} + x_2 \partial_{y_1} - y_1 \partial_{x_2} - x_1 \partial_{y_2}.
\end{aligned}
$$
Parametrize $L_1$ as $(x_1,y_1,x_2,y_2) = (\cos t,\sin t,0,0), t\in [0,2\pi]$. The winding number of $X_1|_{L_1}\in \xi_0|_{L_1}$ with respect to $N_0\equiv-\partial_{x_2}\in \xi_0|_{L_1}$, normal to the page $u_0$, coincides with $\mathrm{sl}(L_1)$. We compute
$$
\text{sl}(L_1) = \text{wind}(X_1|_{L_1},N_0) = \text{wind}(t \mapsto (-\cos t,\sin t),t\in [0,2\pi])=-1.
$$

\begin{defi}
We say that a link $L=L_1 \cup L_2 \subset (S^3,\xi_0)$ is a Hopf link if $L_1$ and $L_2$ are simply linked Hopf fibers.
\end{defi}

We will consider Hopf links whose components are Reeb orbits oriented by the flow, and whose linking number is $+1$.  Let $L_1$ be as before, and let $L_2:= \{0\} \times S^1$. The  Hopf link $L:=L_1 \cup L_2$ is the binding of an open book decomposition $\mathcal{O}_2$ whose annulus-like pages are
$$
\begin{aligned}
v_c: {\mathbb{D}}^o\setminus \{0\} & \to S^3 \setminus L, \ c\in S^1,\\
z & \mapsto (z,c\sqrt{1-|z|^2} \bar z/|z|).
\end{aligned}
$$

It is simple to check  that the orientations of $L_1$ and $L_2$  induced by the pages of $\mathcal{O}_1$ and $\mathcal{O}_2$ coincide with the orientations induced by the flow of $\lambda_0$.  Any embedded annulus representing $[v_c]\in H_2(S^3,L)$ is called a Birkhoff annulus bounded by $L$.  This notation is reminiscent of the annulus-like global surfaces of section found by Birkhoff for positively curved geodesic flows on the $2$-sphere.

Let $\lambda=f\lambda_0$ be a contact form on $(S^3,\xi_0)$. Its Reeb flow $\phi_t,\ t\in \mathbb{R},$  preserves $\xi_0$. This follows from Cartan magic formula  $L_{R_\lambda} \lambda = i_{R_\lambda} d\lambda + di_{R_\lambda} \lambda \equiv 0$. Denote by $\Phi:\xi_0 \to S^3 \times \mathbb{C}$ the $d\lambda$-symplectic trivialization of $\xi_0$ induced by  $\{X_1/\sqrt{f},X_2/\sqrt{f}\}$. Let $P=(x,T) \in \mathcal{P}(\lambda)$ be a Reeb orbit.  Then the linearized flow $d\phi_t:\xi_{x(0)} \to \xi_{x(t)}$ determines a path of $2 \times 2$ symplectic matrices $\Psi(t)\in \text{Sp}(2), t\in [0,T],$ starting from the identity. This path determines the Conley-Zehnder index and the rotation number of $P$ (half of the mean index), denoted by $\mathrm{CZ}(P)$ and $\rho(P)=\hat i(P)/2$, respectively. See Appendix \ref{sec: Maslov-type index}.

The following theorems will be useful in the proof of Theorem \ref{main1}.

\begin{thm}[Hryniewicz-Salom\~ao \cite{HS2}]\label{thm_L_connected}
Let $\lambda=f\lambda_0$ be a contact form on $S^3$ and let $P \subset S^3$ be a closed Reeb orbit. Assume that $P$ is a Hopf fiber in $(S^3,\xi_0)$ and that the following conditions hold:
\begin{itemize}
\item[(a)] $\rho(P)>1$.
\item[(b)] Every closed Reeb orbit $P'\subset S^3 \setminus P$ is linked with $P$.
\end{itemize}
Then $P$  is the binding of an open book decomposition whose disk-like pages are global surfaces of section.
\end{thm}

Recall that the existence of a disk-like global surface of section as in Theorem \ref{thm_L_connected} implies that the flow has either two or infinitely many periodic orbits. Indeed, since the first return map preserves a finite area form, Brouwer's translation theorem implies the existence of at least one fixed point. Then Franks' generalization of the Poincar\'e-Birkhoff Theorem implies that a second periodic point of the first return map forces the existence of infinitely many periodic points.

The next theorem provides conditions for a pair of periodic orbits to bind an annulus-like global surface of section.

\begin{thm}[Hryniewicz-Salom\~ao-Wysocki \cite{HSW}]\label{thm_Birkhoff}
Let $\lambda=f\lambda_0$ be a contact form on $(S^3,\xi_0)$ and let $L=L_1\cup L_2$ be a Hopf link formed by a pair of Reeb orbits whose linking number is $+1$. Assume
that
\begin{itemize}
\item[(a)] $\rho(L_i)>0, i=1,2.$
\item[(b)] Every periodic orbit $P' \subset S^3\setminus L$  has non-zero intersection number with a Birkhoff annulus bounded by $L$.
\end{itemize}
Then $L$ is the binding of an open book decomposition whose pages are global surfaces of section.
\end{thm}

As mentioned before, the existence of an annulus-like global surface of section as in Theorem \ref{thm_Birkhoff} implies the existence of infinitely many periodic orbits provided a third Reeb orbit exists.

The following theorem gives a criterion for the existence of infinitely many periodic orbits without the use of global surfaces of section. It requires the existence of a Hopf link formed by a pair of Reeb orbits whose rotation numbers satisfy a non-resonance condition.

\begin{thm}[Hryniewicz-Momin-Salom\~ao \cite{HMS15}]\label{thm_PB}
Let the contact form $\lambda = f \lambda_0$, $f>0$, on $(S^3,\xi_0)$  admits a pair of Reeb orbits  $L_i$, $i=1,2$, forming a Hopf link, and let the real numbers $\eta_0,\eta_1$ be defined as
$\eta_i = \rho(L_i) - 1,\ i=1,2,$
where $\rho(L_i)$ is the rotation number of $L_i$. Let $0 \neq (p,q)\in \mathbb{Z}^2$ be a prime pair of integers, i.e., there exists no integer $k>1$ such that $(p/k,q/k) \in \mathbb{Z}^2$.  If
$$
(1,\eta_1)<(p,q)<(\eta_0,1) \quad  \mbox{ or }  \quad (\eta_0,1)<(p,q)<(1,\eta_1),
$$
then there exists $P \in \mathcal{P}(\lambda)$ such that
 $$
 \text{link}(P,L_1) = p \quad \mbox{ and } \quad \text{link}(P,L_2) = q.
 $$
\end{thm}
In the statement above, the inequality $a < b$, with $a,b\in \R^2 \setminus \{(x,y):x\leq 0, y\leq 0\}$, means that the arguments of $a$ and $b$ satisfy $-\frac{\pi}{2} < \text{arg}(a)<\text{arg}(b)< \pi$.

\subsection{Pseudo-holomorphic curves and finite energy foliations}

The open book decompositions given in Theorems \ref{thm_L_connected} and \ref{thm_Birkhoff} are obtained as projections of finite energy foliations in the symplectization of the contact manifold. In this section, we briefly introduce these concepts.

Let $\lambda$ be a contact form on a closed three-manifold $M$. The four-manifold $\R \times M$ is naturally equipped with the symplectic form $d(e^a\lambda)$, and  is called the symplectization of $(M,\xi=\ker \lambda)$. Here, $a$ is the $\R$-coordinate.
Since $d\lambda$ restricts to $\xi$ as a symplectic form, there exists a complex structure $J:\xi \to \xi, J^2=-I,$ so that $d\lambda(\cdot, J \cdot)$ is a positive-definite inner product. The space of such $J$'s, denoted $\J(\lambda)$, is non-empty and contractible in the $C^\infty$-topology. For each $J\in \J(\lambda)$,  we consider the almost-complex structure $\tilde J: T(\R \times M) \to T(\R \times M),\tilde J^2 = -I,$ satisfying
$$
\tilde J|_\xi = J \quad \mbox{ and } \quad \tilde J \cdot \partial_a = R_\lambda.
$$

Given a closed connected Riemann surface $(\Sigma,j)$ and a finite set $\Gamma \subset \Sigma$, we call a map $\tilde u=(a,u):\Sigma^o:=\Sigma \setminus \Gamma \to (\R \times M, \tilde J)$ a finite energy $J$-holomorphic curve if it satisfies the Cauchy-Riemann-type equation
$$
d\tilde u \circ j = \tilde J(u) \circ du,
$$
and its Hofer energy satisfies
\begin{equation}\label{finiteenergy}
0< E(\tilde u):= \sup_{\varphi\in \Lambda} \int_{\Sigma^o} \tilde u^*d(\varphi(a) \lambda)<+\infty.
\end{equation}
Here, $\Lambda$ stands for the set of smooth non-decreasing functions $\varphi:\R \to [0,1]$. The points in $\Gamma$ are called punctures. Due to the $\R$-invariance of $\tilde J$, any shift of $\tilde u$ in the $\R$-direction is also $J$-holomorphic.

The simplest finite energy curve is a trivial cylinder: given a closed Reeb orbit $P=(x,T)$, let $x_T:=x(T\cdot)$. Then $\tilde u:(\R \times \R / \Z,i) \to (\R \times M,\tilde J)$, given by $\tilde u(s+it)=(a(s,t),u(s,t)) := (sT, x_T(t))$, is $\widetilde J$-holomorphic and has energy $E(\tilde u) = T$.

In general, \eqref{finiteenergy} implies that $\tilde u$ is well-behaved near  a puncture $z\in \Gamma$: if $a$ is bounded near $z$, then $z$ is removable, that is $\tilde u$ smoothly extends over $\Sigma^o \cup \{z\}$, otherwise either $a\to +\infty$ or $a\to -\infty$.  We call $z$  positive or negative according to each case, respectively. Let us assume all punctures are non-removable. Let $Z_+:=[0,+\infty) \times \R / \Z$ be holomorphic polar coordinates on a punctured disk about $z\in \Gamma$. Write $\tilde u(s,t) = (a(s,t),u(s,t)),\ (s,t) \in Z_+$ near $z$. Then the limit $\lim_{s \to +\infty} \int_{\{s\} \times \R / \Z} u^*\lambda = \pm T$ exists for some $T>0$, and its sign coincides with the sign of $z$. Given any sequence $s_n \to +\infty$, there exists a subsequence, also denoted $s_n$, and a periodic orbit $P=(x,T)$, depending on $s_n$, so that $u(s_n,\cdot) \to x(T \cdot)$ in $C^\infty$ as $n\to \infty$. $P$ is called an asymptotic limit of $\tilde u$ at $z$. If $P$ is nondegenerate, then $P$ is the unique asymptotic limit of $\tilde u$ at $z$ and $u(s,\cdot) \to x(T \cdot)$ in $C^\infty$  as $s \to +\infty$.

A finite energy foliation $\F$ is a foliation of $\R \times M$ so that each leaf $F\in \F$ is the image of a finite energy $J$-holomorphic curve $\tilde u=(a,u): (\Sigma^o,j) \to \R \times M$. Each puncture of $\tilde u$ has a unique asymptotic limit, and the closure of $u(\Sigma^o) \subset M$ is a compact embedded surface. The binding  of $\F$ is the set of asymptotic limits of all leaves and is formed by a finite set of Reeb orbits  $P_1,\ldots, P_l$, called binding orbits. The trivial cylinders over the binding orbits are regular leaves and every other leaf is asymptotic to some of these trivial cylinders at its punctures. $\F$ is invariant under shifts in the $\R$-direction, and its projection to $M$ is a singular foliation $\F_M$ whose singular set is formed by the binding orbits. The regular leaves of $\F_M$ are required to be transverse to the flow. In some special cases, $\F_M$ determines an open book decomposition of $M$ whose pages are global surfaces of section.  The theory of pseudo-holomorphic curves in symplectizations was developed by Hofer, Wysocki, and Zehnder in \cite{props1,props2,props3}. Finite energy foliations for star-shaped hypersurfaces in $\R^4$ were first studied  in \cite{convex,fols}.  

\section{Proof of Theorem \ref{main1}}

\subsection{Proof of Theorem \ref{main1}-(i)}

First, recall that the sphere-like energy surface $\mathfrak{M}\subset \R^4$ has contact type.  Indeed, every regular energy surface of a mechanical system has contact type, see Theorem 4.8 in \cite{HZ94} for a proof. Hence there exists a contact form $\lambda$ on $\mathfrak{M}$ so that its Reeb vector field  is a positive multiple of the Hamiltonian vector field. In addition, $\mathfrak{M}$ admits a strong symplectic filling in the sense that it is the boundary of a compact symplectic manifold $(W,\omega_0)$, with $W = \{H \leq h\} \subset \R^4$ and standard symplectic form $\omega_0=dp_r \wedge dr + dp_z \wedge dz$. Moreover, there exists a Liouville vector field $Y$ ($L_Y \omega_0 = \omega_0$) defined on a neighborhood of $\mathfrak{M}\subset \R^4$ so that $Y$ is outward transverse to $\mathfrak{M}$ and $i_Y \omega_0|_\mathfrak{M} = \lambda$. The strong filling implies that $\xi = \ker \lambda$ is tight  in the sense that it does not admit an embedded disk $\mathcal{D}\hookrightarrow \mathfrak{M}$ so that $\partial \mathcal{D}$ is tangent to $\xi$, and $T_z\mathcal{D} \neq  \xi, \forall z\in \partial \mathcal{D}$. Since, up to a diffeomorphism, there exists only one tight contact structure on  $\mathfrak{M} \simeq S^3$, we conclude that $(\mathfrak{M},\xi=\ker \lambda)$ is contactomorphic to the tight $3$-sphere $(S^3,\xi_0)$, i.e., there exists a diffeomorphism $g: S^3\to \mathfrak{M}$ so that $g_*\xi_0 = \xi$. The contact form $g^*\lambda$, still denoted $\lambda$, has contact structure $\xi_0$ and its Reeb flow is equivalent to the Hamiltonian flow on $\mathfrak{M}$.

We claim that the Euler orbit $\zeta_e\subset \mathfrak{M}$ is unknotted and has self-linking number  $-1$. To prove the claim we first find a co-oriented embedded compact disk $F \hookrightarrow \mathfrak{M}$ bounded by $\zeta_e$  whose interior is positively transverse to the flow with respect to the co-orientation. This disk is explicitly given by
\begin{equation}\label{Eulerplane}
F:= \{(p_r,p_z,r,z) \in \mathfrak{M}: z=0, p_z \geq 0\}.
\end{equation}
The disk $F$ is embedded since the projection $\hat \zeta_e = \Pi_{r,z} ( \zeta_e)\subset \mathcal{H}\cap \{z=0\}$ does not self-intersect. Its boundary  $\partial F  = \{(p_r,p_z,r,z): z=0, p_z=0\}$ coincides with $\zeta_e$. An interior point $w=(p_r,p_z,r,0) \in F \setminus \partial F$ is such that $p_z>0$ and $(r,0) \in \hat \zeta_e \setminus \partial \mathcal{H}$. Moreover,
\begin{equation}\label{basis}
T_wF = \text{span}\{p_z \partial_{p_r} - p_r \partial_{p_z},-\partial_rV(r,0)\partial_{p_z} + p_z\partial_r\}.
\end{equation}
Since the Hamiltonian vector field is $$X_H=-\partial_rV \partial_{p_r} -\partial_zV \partial_{p_z} + p_r \partial_r + p_z \partial_z,$$ we conclude that $F \setminus \partial F$ is transverse to the flow. The $1$-form $\lambda|_{F\setminus \partial F}>0$ naturally induce a co-orientation on $F$.

Consider the $d\lambda$-positive frame $\{X_1,X_2\} \subset T\mathfrak{M}$ given by
\begin{equation}\label{frameX1X2}
\begin{aligned}
X_1 & = \partial_zV \partial_{p_r} - \partial_rV \partial_{p_z} +p_z \partial_r - p_r \partial_z,\\
X_2 & = -p_z \partial_{p_r} + p_r \partial_{p_z} + \partial_zV \partial_r - \partial_rV \partial_z.
\end{aligned}
\end{equation}

Notice that $\text{span}\{X_1,X_2\}$ is transverse to $X_H$ and thus the frame $\{X_1,X_2\}$ induces a trivialization of $\xi$ under the projection $T\mathfrak{M} \to \xi$ along $X_H$. The non-vanishing vector field
$$
\partial_{p_z}|_{\zeta_e} =\frac{-\partial_rV(r,0)X_1 + p_rX_2}{\partial_rV(r,0)^2+p_r^2}
$$
is tangent to $F$ along $\zeta_e$. The self-linking number of $\zeta_e$ coincides with minus the winding number of $\partial_{p_z}|_{\zeta_e}$ with respect to the frame $\{X_1,X_2\}|_{\zeta_e}$ along the period of $\zeta_e$. Since the winding number of $(-\partial_rV(r,0),p_r)|_{\zeta_e}\in \C \setminus 0$  around the origin is $+1$, we conclude that $\sl(\zeta_e) = -1.$

Now let $w(t)=(p_r(t),p_z(t),r(t),z(t))\in \mathfrak{M}\setminus \zeta_e$ be a periodic trajectory. Denote $p_z(t) + \sqrt{-1} z(t) = \rho(t) e^{\sqrt{-1} \eta(t)}, \forall t,$ for some continuous functions $\eta(t)\in \R$ and $\rho(t)>0$. It follows from \eqref{ham} that
$$
\dot \eta  = \cos^2 \eta +g(r,z) \sin^2 \eta,
$$
where $$g(r,z) = \frac{4\alpha^{-1}(1+2\alpha)}{(r^2+(1+2\alpha)z^2)^{3/2}}, \quad \forall (r,z).$$
Since $ g(r,z) > g_{\min}, \forall (r,z)\in \mathcal{H}$, for some $g_{\min}>0$, we conclude that
\begin{equation}\label{etamin}
\dot \eta> \eta_{\min},
\end{equation}
for a suitable constant $\eta_{\min}>0$. Hence $z(t)$ vanishes  and $w(t)$ transversely intersects $F\setminus \partial F$ infinitely many times forward and backward in time.

Let us check that such intersections are positive. Consider the orientation of $\mathfrak{M}$ induced by $\lambda \wedge d\lambda >0$. Since $i_{X_H} \omega_0 = - dH$ and the Liouville vector field $Y$ points outward $W$ at $\mathfrak{M}=\partial W$, we conclude that $\lambda(X_H) = \omega_0(Y,X_H)=dH \cdot Y >0.$ Then $X_H$ fits the co-orientation of $\lambda$, i.e. $X_H$ intersects $F\setminus \partial F$ positively. We conclude that every intersection of $w(t)$ with $F \setminus \partial F$ is positive, and thus $w(t)$ positively links with $\zeta_e = \partial F.$

Notice that $\lambda$ also induces an orientation of $F$ via $\iota_{X_H}(\lambda\wedge d\lambda)|_F>0$. It turns the basis
\eqref{basis} into a positive basis since $d\lambda(\partial_{p_r},\partial_r)=\omega_0(\partial_{p_r},\partial_r)>0$. This orientation of $F$ induces the oriented basis $\{\partial_{p_r}, \partial_r\}$ on the $(p_r,r)$-plane under the natural projection, and thus it induces on $\zeta_e=\partial F$ the  orientation   induced by the flow.

By Proposition \ref{prop: rho_e}, we have $\rho(\zeta_e)>2$. Then Theorem \ref{thm_L_connected} implies that $\zeta_e$ is the binding of an open book decomposition $\mathcal{O}_1$ whose disk-like pages are global surfaces  of section. Moreover, for every $J\in \J(\lambda)$, there exists a finite energy foliation on $\R \times \mathfrak{M}$ by $J$-holomorphic curves whose projection to $\mathfrak{M}$ is $\mathcal{O}_1$. This completes the proof of Theorem \ref{main1}-(i). \qed

\begin{rem}\label{rem_openbook} The disk $F\hookrightarrow \mathfrak{M}$ defined in \eqref{Eulerplane} is itself the page of an open book decomposition $\mathcal{O}$ whose pages are disk-like global surfaces of section. To see this, consider the projection $\Pi_{p_z,z}:\R^4\to \R^2$ to the $(p_z,z)$-plane and observe  that  $\Pi_{p_z,z}^{-1}(0,0)$ coincides with the Euler orbit $\zeta_e$. Now fix a point $(p_r,p_z,r,z)\in \mathfrak{M}$ so that $(p_z,z) \neq 0$.
The function $\tau \mapsto H(p_r,\tau p_z,r,\tau z),\ \tau>0,$ is strictly increasing and converges to a number in $(-1,+\infty]$ as $\tau \to +\infty$. Hence, for every $\tau \in [0,1),$ we have $H(p_r,\tau p_z,r,\tau z) < -1$. For a suitable $\hat p_r> p_r$ depending on $\tau$, $H(\hat p_r, \tau p_z, r, \tau z) = -1$. This implies that $(\hat p_r,\tau p_z, r,\tau z) \in \mathfrak{M} \Rightarrow (\tau p_z, \tau z) \in \Pi_{p_z,z} (\mathfrak{M})$.  We conclude that $\Pi_{p_z,z}(\mathfrak{M})$  is a star-shaped domain $\mathcal{D}_{p_z,z}\subset \R^2$ with smooth boundary. The pre-image of a boundary point $(p_{z,0},z_0)\in\partial \mathcal{D}_{p_z,z}$ coincides with a single point $(0,p_{z,0},r(z_0),z_0)$, where $r(z_0)>0$ is the unique minimum of $V(\cdot,z_0)$, see Proposition \ref{prop_Mcomp}.

Now notice that each radial segment $\ell$ issuing from $(0,0)$ to a point $(p_{z,0},z_0) \in \partial \mathcal{D}_{p_z,z}$ is the projection under $\Pi_{p_z,z}$ of an embedded disk bounded by $\zeta_e$. Indeed, since for each fixed $(p_{z,1},z_1)\in \ell \setminus \{0,(p_{z,0},z_0)\}$, the function $h:(p_r,r) \mapsto H(p_r,p_{z,1},r,z_1)$ has a unique critical point at $(0,r(z_1))$, which is a minimum and satisfies $H(0,p_{z,1},r(z_1),z_1)< -1$. Moreover, since $V(r,z)>V(r,0), \forall r>0, \forall z\neq 0,$ we conclude that $h>-1$ when restricted to $\Pi_{p_r,r}(\zeta_e)$. Hence $(p_{z,1},z_1)$ is the projection under $\Pi_{p_z,z}$ of a topological circle in $\mathfrak{M}$, which is in one-to-one correspondence with a topological circle in the interior of   $\Pi_{p_r,r}(\zeta_e)$, see Figure \ref{fig on z pz}.   Taking the union of all such circles in $\mathfrak{M}$, we obtain a smooth embedded disk $\Sigma_\ell \subset \mathfrak{M}$ bounded by the $\zeta_e$.

\begin{figure}[ht]
\centering
\includegraphics[width=0.7 \textwidth]{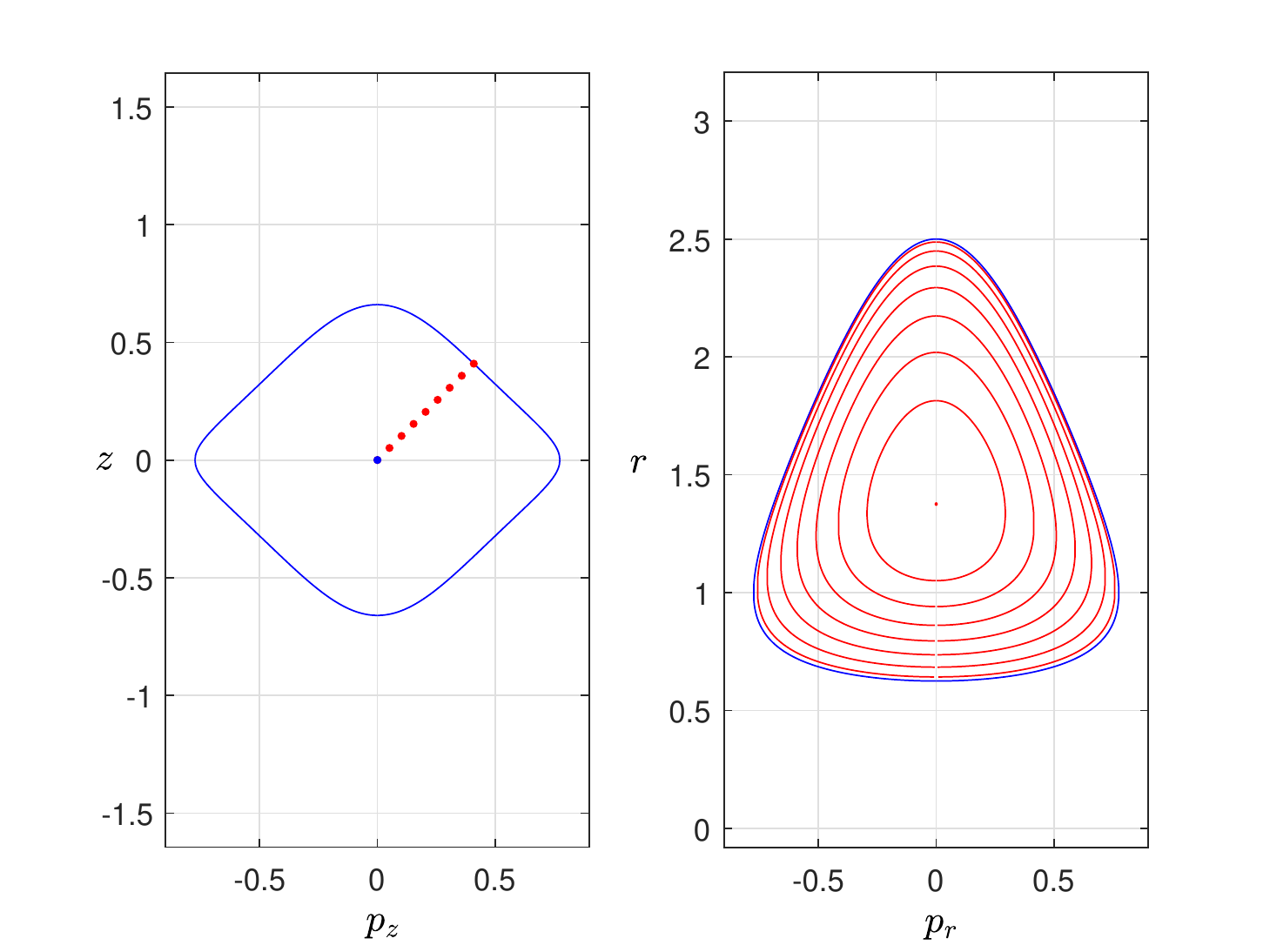}
\caption{Projection of the disk-like global surface of section to the $(p_z,z)$-plane (left) and the $(p_r,r)$-plane (right).}
\label{fig on z pz}
\end{figure}

The uniform estimate \eqref{etamin} shows that  $\Sigma_\ell$ is a disk-like global surface of section.  Finally, the family of all such disks $\Sigma_\ell$ determine an open book decomposition $\mathcal{O}$ of $\mathfrak{M}$ with binding orbit $\zeta_e$ and pages
\begin{equation}\label{equ: family of pages}
\Sigma_s :=\Pi_{p_z,z}^{-1}(\ell_s),\quad \ell_s := e^{2\pi s \sqrt{-1}}\R^+ \cap  \mathcal{D}_{p_z,z}, s\in \mathbb{R}/\mathbb{Z}.
\end{equation}

Observe that $F=\Sigma_0$ is a page of $\mathcal{O}$. See \cite{Kim3} for global surfaces of section admitting symmetries.
\end{rem}


\subsection{Proof of Theorem \ref{main1}-(ii)}

We start with Birkhoff's shooting method to prove the existence of a simple $z$-symmetric brake orbit $\zeta_z$ whose linking number with $\zeta_e$ is $+1$. Then we show that every periodic orbit in the complement of $L:=\zeta_e \cup \zeta_z$ is positively linked with a Birkhoff annulus bounded by $L$.

Let $x_e(t) = (p_{re}(t),0,r_e(t),0)\in\mathfrak{M}$ be the $T_e$-periodic solution whose image is $\zeta_e$. Assume that $r_e(0) = (r_{\min},0) \in \partial \mathcal{H}$, $r_e(T_e/2) = (r_{\max},0) \in \partial \mathcal{H}$ and $r(t) \in \mathcal{H} \setminus \partial \mathcal{H}, \forall t\in (0,T_e/2)$.

Consider the global frame $\{X_1,X_2\}\subset T\mathfrak{M}$ as in \eqref{frameX1X2}. This frame  induces a trivialization of $\xi$. Observe that
$$
X_1|_{\zeta_e} = -\partial_rV(r_e,0) \partial_{p_z} - p_r \partial_z \quad \mbox{ and } \quad X_2|_{\zeta_e} = p_r \partial_{p_z} - \partial_rV(r_e,0) \partial _z,
$$
giving
$$
\partial_z|_{x_e} = -\frac{p_r X_1 + \partial_rV(r_e,0)X_2}{p_r^2+\partial_rV(r_e,0)^2}.
$$
Let $F\subset \{z=0\}$ be the disk bounded by $\zeta_e$ defined as in \eqref{Eulerplane}. Let
 $$
 \eta_0:= X_2(x_e(0))=-\partial_rV(r_{\min},0)\partial_z.
 $$
 Notice that $\eta_0$ is transverse to $F$ and its projection  to the Hill region  is upward tangent to $\partial \mathcal{H}$. Let
 $$
 \eta(t)=:a_1(t)X_1+a_2(t)X_2, \quad \eta(0)=\eta_0,
 $$
 be the solution of the linearized flow along $\zeta_e$ projected to the plane spanned by $X_1,X_2$.  Then
 $$
 \eta \cdot \partial_z =(a_1,a_2) \cdot(-p_{re},-V_r(r_e,0)).
 $$
 We know from Proposition \ref{prop: rho_e} that the rotation number of $\zeta_e$ is $ >2$. This implies that the variation in the argument of  $(a_1(t),a_2(t))\in \C \setminus \{0\}, t\in [0,T_e/2],$ is  $>2\pi$. Since the variation in the argument of $(-p_{re}(t),-\partial_rV(r_e(t),0))\in \C \setminus \{0\}, t\in [0,T_e/2],$ is $\pi$, we conclude that there exists a least $t_0 \in (0,T_e/2)$ so that $\eta(t_0) \cdot \partial _z = 0$ and $\eta(t) \cdot \partial_z$ changes from positive to negative at $t_0$. In particular, any solution $x(t)=(p_r(t),p_z(t),r(t),z(t))\in \mathfrak{M}$ satisfying $(r(0),z(0)) \in \partial \mathcal{H},$ with $z(0)>0$ small and $r(0)$ close to $r_{\min}$, necessarily crosses $z=0$ for the first time at time $\bar t_0<T_e/2$, with $p_z(\bar t_0)<0$ and $p_r(\bar t_0) >0.$

 Arguing as above, we conclude that any solution $x(t)=(p_r(t),p_z(t),r(t),z(t))\in \mathfrak{M}$ satisfying $(r(0),z(0)) \in \partial \mathcal{H},$ with $z(0)>0$ small and $r(0)$ close to $r_{\max}$, also crosses $z=0$ for the first time at time $\bar t_0<T_e/2$, with $p_z(\bar t_0)<0$ and $p_r(\bar t_0) <0.$ We know from \eqref{ham} that all solutions $x(t) \in \mathfrak{M}$ starting from $\partial \mathcal{H} \cap \{z>0\}$ transversely hit $\{z=0\}$ for the first time with uniformly bounded hitting time, see Remark \ref{rem_openbook}. Hence, by the intermediate value theorem and the continuous dependence of solutions relative to initial conditions, we find $(r_*,z_*) \in \partial \mathcal{H},$ with $z_*>0,$ so that the solution $\zeta_z(t)$ starting from $(0,0,r_*, z_*)$ hits $\{z=0\}$ perpendicularly at  $t=T_0$, and $z(t)>0,\ \forall t\in [0, T_0)$. The $z$-symmetry and the reversibility of \eqref{ham} imply that $\zeta_z(t)$ is a $z$-symmetric periodic orbit with least period $4T_0$. The monotonicity of $z(t)$ in the interval $[0,2T_0]$ implies that $\zeta_z(t)$ is a simple brake orbit. The Reeb orbit $\zeta_z$ is the desired $z$-symmetric simple brake orbit.

 Now we claim that
\begin{lem} \label{claim of Hopf fiber}
$\zeta_z$ is a Hopf fiber and $L:=\zeta_e \cup \zeta_z$ is a Hopf link.
\end{lem}

\begin{proof}
Consider a parametrization $\zeta_z(t)=(p_{r}(t),p_{z}(t),r(t),z(t)) \in \mathfrak{M}$ of $\zeta_z$ so that $(r(0),z(0)),(r(2T_0),z(2T_0)) \in \partial\mathcal{H}$ and $(r(t),z(t)) \in \mathcal{H} \setminus \partial \mathcal{H}, \forall t\in (0,2T_0).$ The global frame in \eqref{frameX1X2} restricts to $\zeta_z$ as
\begin{equation}\label{equ: X1X2 along bz orbit}
 \begin{aligned}
 X_1|_{\zeta_z} & = \partial_zV(r,z) \partial_{p_r} -\partial_rV(r,z) \partial_{p_z}+p_{z} \partial_r - p_{r} \partial_z,\\
 X_2|_{\zeta_z} &  = -p_{z} \partial_{p_r} + p_{r} \partial_{p_z} + \partial_zV(r,z) \partial_r -\partial_rV(r,z) \partial_z.
 \end{aligned}
\end{equation}

Let $\hat \zeta_z: = \Pi_{r,z}(\zeta_z)$ be the projection of $\zeta_z$ to the Hill region $\mathcal{H}$. Then  $\mathcal{H} \setminus \hat \zeta_z$ has two components, say  $\mathcal{H}_1$ and $\mathcal{H}_2$. Since $\hat \zeta_z$ is  the image of a regular curve, we can choose a normal vector $\vec n$ along $\hat \zeta_z$ pointing outward $\mathcal{H}_1$. As in the case of the Euler orbit, we consider the disk $F_z \hookrightarrow \mathfrak{M}$ consisting of all $(p_r,p_z,r,z) \in \mathfrak{M}$ so that $(r,z) \in \hat \zeta_z$ and $\vec n \cdot (p_r,p_z) \geq 0$. Since $\hat \zeta_z$ has no self-intersections, $F_z$ is an embedded disk bounded by $\zeta_z$ and its interior is positively transverse to the flow. Denote $\vec n = n_r \partial_r + n_z \partial_z$. Then the non-vanishing vector field $(n_r \partial_{p_r}+n_z \partial_{p_z})|_{\zeta_z}$ is tangent to $F_z$ along $\zeta_z$. Its projection $\vec N$ to $\text{span}\{X_1,X_2\}$ writes as
 $$
 \vec N = \frac{1}{|X_1|^2} [(n_r\partial_zV(r,z) - n_z\partial_rV(r,z)) X_1 + (-n_r p_{z} + n_z p_{r})X_2].
 $$
 Let $a:=n_r\partial_zV(r,z) - n_z\partial_rV(r,z)$ and $b:=-n_r p_{z} + n_z p_{r}$. Both $a$ and $b$ depend on $t$. In the interval $[0,2T_0]$,  $\hat \zeta_z$ is traversed from  $(r(0),z(0))\in \partial \mathcal{H}$ to $ (r(2T_0),z(2T_0))\in \partial \mathcal{H}$. We observe that $b(0)=0$ and, replacing $\vec n$ with $-\vec n$ if necessary, we can assume that $a(0)>0$. We then have $b(t)>0,\ \forall t\in (0,2T_0),\ b(2T_0)=0$ and $a(2T_0)<0.$ Similarly, while traversing $\hat \zeta_z$ in the opposite direction from $(r(2T_0),z(2T_0))\in \partial \mathcal{H}$ to $(r(0),z(0))\in \partial \mathcal{H}$, we observe that $b(t)<0,\ \forall t \in (2T_0,4T_0)$, $a(4T_0)>0$ and $b(4T_0)=0$. Hence the winding of $t\mapsto (a(t),b(t)),\ t\in [0,4T_0]$ is $+1$ implying that the winding of $N$ in the frame $\{X_1,X_2\}$ along the period of $\zeta_z$ is $+1$. We conclude that $\sl(\zeta_z) = -1,$
 and hence $\zeta_z \subset (\mathfrak{M},\xi)$ is a Hopf fiber. By construction, the linking number between $\zeta_z$ and $\zeta_e$ is $+1$, and thus $L:=\zeta_e \cup \zeta_z$ is a Hopf link.
 \end{proof}

 Here we give a brief proof of the following important fact.
 \begin{lem}\label{claimbrake}
 The rotation number of any simple brake orbit is at least 1.
 \end{lem}
\begin{proof}Indeed, observe that a neighborhood of $\zeta_z$ in $F_z$ is an embedded strip whose interior is positively transverse to the flow. As we have checked above, the winding number of $\vec N$ in the global frame $\{X_1, X_2\}$ along the period of $\zeta_z$ is $+1$. Here, $\vec N$ is  outward tangent to $F_z$ along $\zeta_z$. Assume by contradiction that $\rho(\zeta_z)<1$. Then the linearized solutions rotate negatively with respect to a frame aligned to $\vec N$ and hence there exist solutions sufficiently close to $\zeta_z$ that negatively intersect $F_z$. This contradicts the fact that $F_z\setminus \zeta_z$ is positively transverse to the flow and the lemma is proved.
\end{proof}

 It remains to show that every periodic orbit positively links with a Birkhoff annulus $\Sigma \hookrightarrow \mathfrak{M}$ bounded by $L$. The orientation of $\Sigma$ is the one that induces on $\zeta_e$ and $\zeta_z$ the flow orientation. Observe that the orientations of $F$ and $F_z$ also induce on $\zeta_e$ and $\zeta_z$ the flow orientation, respectively. Denote by $\bar F$ and $\bar F_z$ the disks $F$ and $F_z$ with the opposite orientation. Then $\mathcal{C}:=\Sigma \cup \bar F \cup \bar F_z\subset S^3$ is a closed $2$-chain. In particular, if $P'\subset \mathfrak{M} \setminus L$ is a Reeb orbit, then the algebraic intersection number  $\text{int}(P',\mathcal{C})$ vanishes. In particular, $\text{int}(P',\Sigma) = \text{int}(P',F) + \text{int}(P',F_z) \geq 1 + 0=1$ since $F$ is a global surface of section and $F_z$ is positively transverse to the flow. Theorem \ref{thm_Birkhoff} and Lemma \ref{claimbrake} implies that the Hopf link $L=\zeta_e \cup \zeta_z$ is the binding of an open book decomposition whose pages are annulus-like global  surfaces of section. The proof of Theorem \ref{main1}-(ii) is complete. \qed

\section{Non-resonant Hopf links for large mass ratios}
In this section, we prove Theorem \ref{thm: nonresonant implies infy many}, i.e., we show that for $\alpha>0$ sufficiently large and $\varpi>0$ satisfying
\begin{equation}\label{pap req}
1<2\varpi^2<1+\alpha^{-2},
\end{equation}
the energy surface   $\mathfrak{M}$ carries an arbitrary number of simple $z$-symmetric brake orbits $\zeta_z$ with arbitrarily large rotation numbers. Since the rotation number of $\zeta_e$  is arbitrarily close to $2\sqrt{2} +1$ for $\alpha$ large, the Hopf link $L=\zeta_e \cup \zeta_z$ is non-resonant, that is $\rho(\zeta_e)^{-1} + \rho(\zeta_z)^{-1} \neq 1$. In particular,  the Poincar\'e-Birkhoff theorem for Reeb flows on the tight three-sphere, see Theorem \ref{thm_PB}, implies the existence of infinitely many periodic orbits with prescribed linking numbers with $\zeta_e$ and $\zeta_z$.

The proof of Theorem \ref{thm: nonresonant implies infy many} relies on a suitable re-scaling of our original system for which a limit exists as $\alpha \to +\infty$. To define this rescaling, consider the positive numbers
$$
n:= \frac{1+4\alpha^{-1}}{2} \quad \mbox{ and } \quad m:=\frac{\sqrt{(1+4\alpha^{-1})^2-2\varpi^2}}{2},
$$
and observe that $n\to \frac{1}{2}$ and $m \to 0$ as $\alpha \to +\infty$.

Recall that the Hill region satisfies $\mathcal{H} \subset [r_{\min},r_{\max}] \times \R,$ where
$$
r_{\min}  = n - m \quad \mbox{ and } \quad r_{\max}= n+m.
$$
Moreover, $\mathcal{H} \cap \{z=0\} = [r_{\min},r_{\max}]\times \{0\}$.

The change of coordinates $\psi: (p_r,p_z,r,z) \mapsto (p_v,p_u,u,v)$ given by
\begin{equation}\label{ch_cooridnates}
p_v:=\frac{p_r}{m}, \quad p_u:=\frac{p_z}{m}, \quad
v := \frac{r-n}{m}, \quad
u := \frac{z}{m},
\end{equation}
is symplectic up to the constant factor $\frac{1}{m^2}$. The equations of motion in the new coordinates become
\begin{equation}\label{motion_uv}
\left\{
\begin{aligned}
& \dot v  = p_v\\
& \dot p_v  = -(\frac{n-\varpi^2}{mv} + 1) \frac{v}{(mv+n)^3} - \frac{1}{\alpha m}\frac{4(mv+n)}{((mv+n)^2+(1+2\alpha)m^2u^2)^{3/2}}\\
& \dot u = p_u\\
& \dot p_u  = -\frac{(8+4\alpha^{-1})u}{((mv+n)^2+(1+2\alpha)m^2u^2)^{3/2}}
\end{aligned}
\right.
\end{equation}
and coincide with Hamilton's equations of
\begin{equation}\label{hamiltonianK}
\begin{aligned}
K :=\frac{1}{m^2}H(mp_v,mp_u,mv+n,mu)+\frac{1}{m^2} =\frac{p_v^2+p_u^2}{2}+ U(v,u),
\end{aligned}
\end{equation}
where
$$
U(v,u)=\frac{\varpi^2-2(mv+n) + 2(mv+n)^2}{2m^2(mv+n)^2} -\frac{4\alpha^{-1}}{m^2((mv+n)^2+(1+2\alpha)m^2u^2)^\frac{1}{2}}.
$$

The Hill region $\mathcal{H}$ in $(v,u)$-coordinates is a $u$-symmetric disk-like  region $$\mathcal{H}' =\{(v,u): U(v,u) \leq 0\} \subset \{(v,u):-1 \leq v \leq 1\},$$ so that $\mathcal{H}' \cap \{v=0\} = [-1,1].$ The energy surface $\mathfrak{M}$ is transformed into a sphere-like hypersurface
$$
\mathfrak{M}'=K^{-1}(0)\subset \R^4.
$$

The next lemma shows that $K$ converges locally to a non-trivial Hamiltonian $K_\infty$ as $\alpha \to +\infty$. This fact will be useful to find solutions of $K$ that approximate solutions of $K_\infty$.

\begin{lem}\label{lem_Kinfty}
Let
$$
K_\infty(p_v,p_u,v,u) := \frac{p_v^2+p_u^2}{2}+ 4v^2 - \frac{2}{(1/4 + 4u^2)^{1/2}}, \quad \quad \forall (p_v,p_u,v,u)\in \R^4.
$$
Then $K \to K_\infty$ in $C^\infty_{\text{loc}}$ as $\alpha \to \infty$.
\end{lem}

\begin{proof}
Write $U(v,u)= C(v) + D(u,v),$ where
$$
\begin{aligned}
C(v) & := \frac{\varpi^2-2(mv+n) + 2(mv+n)^2}{2m^2(mv+n)^2}\\
& = \frac{2\varpi^2-2n}{4m^2(mv+n)^2} + \frac{(2n-1)}{2m^2(mv+n)}+\frac{(2n-1)v}{2m(mv+n)^2}+\frac{v^2}{(mv+n)^2}, \ \forall v\neq -\frac{n}{m},
\end{aligned}
$$
and
$$
D(u,v): =-\frac{4\alpha^{-1}}{m^2((mv+n)^2+(1+2\alpha)m^2u^2)^{1/2}}, \quad \forall u,\forall v \neq -\frac{n}{m}.
$$
Notice that \eqref{pap req} implies
$$
-\alpha^{-1} < \alpha(1-2\varpi^2)<0,
$$
and thus
\begin{equation}\label{limitparam}
\begin{aligned}
\alpha m^2 & =\frac{1}{4}(\alpha(1-2\varpi^2) + 8  +16\alpha^{-1})  \to 2, \\
\frac{2n-1}{2m^2} & = \frac{2}{\alpha m^2} \to 1,\\
\frac{2\varpi^2-2n}{4m^2} & =-\frac{\alpha(1-2\varpi^2)+4}{\alpha(1-2\varpi^2) + 8 + 16\alpha^{-1}} \to -\frac{1}{2},
\end{aligned}
\end{equation}
as $\alpha \to +\infty$. Moreover,
$g(v):=mv+n, \forall v\in \R,$ converges  in $C^\infty_{\text{loc}}$ to the constant function  $\frac{1}{2}$ as  $\alpha \to +\infty$.

We conclude that $U(v,u)$ converges in $C^\infty_{\text{loc}}$ to the function $4v^2-\frac{2}{(1/4 + 4u^2)^{1/2}}$ as $\alpha \to +\infty$, proving the lemma.
\end{proof}

 The Hamiltonian $K_\infty$ given in Lemma \ref{lem_Kinfty} decouples as
 $$
 K_\infty = K_1(p_v,v) + K_2(p_u,u),
 $$
and its  flow is determined by the equations
\begin{equation}\label{motion_uv limit}
\left\{
\begin{aligned}
& \dot v  = p_v\\
& \dot p_v  = - 8v\\
& \dot  u  = p_u\\
& \dot p_u  = -\frac{8u}{(1/4+4u^2)^{3/2}}.
\end{aligned}
\right.
\end{equation}

 We are especially interested in the dynamics of \eqref{motion_uv limit} restricted to the unbounded energy surface
 $$
 \mathfrak{M}'_\infty:= K_\infty^{-1}(0).
 $$
 Since $K_1 \geq 0$, we have $\mathfrak{M}'_\infty \subset \{K_2 \leq 0\}.$

  The functions
\begin{equation}\label{initv}
v_\infty(t) = v_0\cos (2\sqrt{2}t), \quad p_{v,\infty}(t) = -2\sqrt{2}v_0 \sin (2\sqrt{2}t), \quad t\in \R,
\end{equation}
solve the first two equations in \eqref{motion_uv limit} with initial conditions  $v_\infty(0)=v_0$ and $p_{v,\infty}(0)=0$.

In the $(p_u,u)$-plane, the solutions admit the Hamiltonian
$$
K_2 = \frac{p_u^2}{2} - \frac{2}{(1/4+4 u^2)^{1/2}}.
$$
Its unique critical point at $(0,0)$ is a global minimum with a critical value $-4$. For every $c\in (-4,0)$, $K_2^{-1}(c)$ is a regular closed curve surrounding the origin. This curve corresponds to a periodic orbit $(p_{u,\infty}(t),u_\infty(t))$ of $K_2$.  Such a periodic orbit approaches $(0,0)$ as $c\to -4$ and is unbounded in the $u$-direction as $c \to 0$. We may assume that $p_{u,\infty}(0)=0$ and $u_\infty(0)=u_0 >0.$ We denote by $T_\infty>0$ the least $t>0$ so that $u_\infty(t)=0$. The existence of  $T_\infty$ follows from the fact that the periodic orbit $(p_{u,\infty}(t),u_\infty(t))$ surrounds the origin.

\begin{lem}\label{lem_T_infty}
Given initial conditions $p_{u,\infty}(0)=0$ and $u_\infty(0)=u_0>0,$ let $T_\infty=T_\infty(u_0)>0$ be the least $t>0$ so that $u_\infty(t)=0$. Then $T_\infty$ is strictly increasing from $\frac{\pi}{16}$ to $+\infty$ as $u_0$ increases from $0^+$ to $+\infty$.
\end{lem}
\begin{proof}
Using the conserved quantity $K_2$ and the initial conditions, we compute
$$
\dot u_\infty(t) =-2\left( \frac{1}{\sqrt{\frac{1}{4} + 4u_\infty(t)^2}} - \frac{1}{\sqrt{\frac{1}{4} + 4u_0^2}}\right)^{\frac{1}{2}}, \quad \forall t\in[0,T_\infty].
$$
We see from the equation above that $|\dot u_\infty|$  is  uniformly bounded from above by $2\sqrt{2}$. Hence $T_\infty \to +\infty$ as $u_0 \to +\infty$.

To prove the monotonicity of $T_\infty$ with respect to $u_0$, we first integrate $\dot u$ in $[0,T_\infty]$ to express the hitting time $T_\infty$ as the following improper integral
\begin{equation}\label{equ: the hitting time}
\begin{aligned}
T_\infty & = \frac{1}{2}\int_0^{u_0}\left( \frac{1}{\sqrt{\frac{1}{4} + 4u^2}} - \frac{1}{\sqrt{\frac{1}{4} + 4u_0^2}}\right)^{-\frac{1}{2}} \, du.
\end{aligned}
\end{equation}
Consider the coordinate change
$$
v = \frac{1}{\sqrt{\frac{1}{4} + 4u^2}} \Leftrightarrow u = \frac{1}{2} \sqrt{\frac{1}{v^2} - \frac{1}{4}}.
$$
Then
$$
T_\infty= \frac{1}{2}\int_{v_0}^2 \frac{1}{\sqrt{v-v_0}}\frac{1}{\sqrt{4-v^2}}\frac{1}{v^2}\, dv,
$$
where $0<v_0 = \frac{1}{\sqrt{\frac{1}{4} +4u_0^2}}<2.$
Now consider the linear coordinate change
$
v=\frac{2+v_0}{2} + \frac{2-v_0}{2}w
$
to obtain
$$
T_\infty=\frac{1}{2} \int_{-1}^1 \frac{1}{\sqrt{1-w^2}} \frac{1}{\sqrt{\frac{6+v_0}{2} + \frac{2-v_0}{2}w}}\frac{1}{(\frac{2+v_0}{2} + \frac{2-v_0}{2} w)^2} \, dw.
$$
We compute
$$
\begin{aligned}
\lim_{u_0\to 0^0}T_{\infty}(u_0)=\lim_{v_0\to 2^-}T_\infty(v_0)=\frac{1}{16}\int^1_{-1}\frac{1}{\sqrt{1-w^2}}dw=\frac{\pi}{16},
\end{aligned}
$$
and
\begin{equation}\label{equ: diff of the hitting time}
\begin{aligned}
\frac{\partial T_\infty}{\partial v_0}  = & -\frac{1}{8}\int_{-1}^1 \frac{1-w}{\sqrt{1-w^2}} \frac{1}{(\frac{6+v_0}{2} + \frac{2-v_0}{2}w)^\frac{3}{2}}\frac{1}{(\frac{2+v_0}{2} + \frac{2-v_0}{2} w)^2} \, dw\\
 &
-\frac{1}{2}\int_{-1}^1 \frac{1-w}{\sqrt{1-w^2}} \frac{1}{\sqrt{\frac{6+v_0}{2} + \frac{2-v_0}{2}w}}\frac{1}{(\frac{2+v_0}{2} + \frac{2-v_0}{2} w)^3} \, dw\\
& <0.
\end{aligned}
\end{equation}
Since $\frac{\partial v_0}{\partial u_0}<0,$ we conclude that $\frac{\partial T_\infty}{\partial u_0}>0,$ and thus $T_\infty$ is strictly increasing.
\end{proof}


\subsection{The limiting linearized flow} Let us consider the linearized flow  determined by
$$
\dot y = DX_{K_\infty}(x_\infty(t)) y,
$$
where $x_\infty(t) = (p_{v,\infty}(t),v_\infty(t),p_{u,\infty}(t),u_\infty(t))$ is a solution to \eqref{motion_uv limit} with initial conditions $p_{v,\infty}(0)=p_{u,\infty}(0) =0$, $v_\infty(0)=v_0$ and $u_\infty(0)=u_0>0$. Here, $X_{K_\infty}$ is the Hamiltonian vector field of $K_\infty$, and $$y=(a,b,c,d)\equiv a\partial_{p_v} + b\partial_{v} +c\partial_{p_u} + d\partial_u.$$ It follows from \eqref{motion_uv limit}, that
$$
\left(\begin{array}{cc} \dot a \\ \dot b \end{array} \right)= \left(\begin{array}{cc}0 & -8 \\
1 & 0 \\
\end{array} \right)\left(\begin{array}{cc}  a \\  b \end{array} \right)
\quad \mbox{ and } \quad \left(\begin{array}{cc} \dot c \\ \dot d \end{array} \right)= \left(\begin{array}{cc}0 & g(u_\infty(t)) \\
1 & 0 \\
\end{array} \right)\left(\begin{array}{cc}  c \\  d \end{array} \right),
$$
where $g(u):=\frac{\partial}{\partial u} \left(-\frac{8u}{\sqrt{\frac{1}{4} + 4 u^2}} \right)$.

Let us assume that $x_\infty$ is a $u$-symmetric brake orbit. This means that $x_\infty$ first hits $u=0$ perpendicularly at time $T_\infty>0$, that is $u_\infty(t)>0, \forall t\in[0,T_\infty),$ $p_u(T_\infty)<0$ and $p_v(T_\infty)=0.$ In particular, $T_\infty=\frac{k\pi}{2\sqrt{2}}$ for some $k\geq 1$ and $x_\infty$ is periodic with least period  $4T_\infty$.

Since $K_\infty$ is decoupled, the rotation number of $x_\infty$ is the sum of the rotation numbers on the respective invariant symplectic planes $(p_u,u)$ and $(p_v,v)$. In the $(p_u,u)$-plane, considering a non-vanishing vector pointing in the direction of the flow, we see that the contribution over the period $4T_\infty$ is $+1$. In the $(p_v,v)$-plane, using that the linearized flow rotates with mean angular velocity $2\sqrt{2}$, the contribution over the period $4T_\infty$ is  $\frac{4T_\infty \cdot 2\sqrt{2}}{2\pi}.$ Hence the rotation number of $x_\infty$ is
\begin{equation}\label{rho_x_infty}
\rho(x_\infty) = 1 + \frac{4\sqrt{2}}{\pi} T_\infty=1 + 2k\geq 3.
\end{equation}

We aim at finding $u$-symmetric brake orbits in the energy surface $\mathfrak{M}'_\infty$  starting from the zero velocity curve
$$
\ell_\infty:=\{(v,u)\in \R^2: 4v^2 = 2/\sqrt{1/4 +4u^2}\}.
$$
We shall consider the branch $\ell_\infty^+ \subset \ell_\infty$ satisfying $v>0$. Observe that $\ell_\infty^+$ is the graph of a smooth bounded function $v=\hat v_\infty(u)>0, u \in \R$.

\begin{lem}\label{lem_sequence} There exists a sequence $x_\infty^{(i)},i\in \N$ of $4T_\infty^{(i)}$-periodic $u$-symmetric brake orbits
$$
\begin{aligned}
x_\infty^{(i)}(t)  = (p_{v,\infty}^{(i)}(t),v_\infty^{(i)}(t),p_{u,\infty}^{(i)}(t),u_\infty^{(i)}(t))\in \mathfrak{M}'_\infty, \quad \forall t,\\
x_\infty^{(i)}(t + 4T_\infty^{(i)})  = x_\infty^{(i)}(t), \quad \forall t,
\end{aligned}
$$
with initial conditions
\begin{equation}\label{init_conditions}
p_{v,\infty}^{(i)}(0)=p_{u,\infty}^{(i)}(0)=0, \quad u_\infty^{(i)}(0)=u_0^{(i)}>0, \quad v_\infty^{(i)}(0)=\hat v_\infty(u_0^{(i)})>0,
\end{equation}
so that
$$
\begin{aligned}
 u_\infty^{(i)}(t)  > 0, \forall t\in [0,T_\infty^{(i)}), \quad u_\infty^i(T_\infty^{(i)}) =0,\\ p_{u,\infty}^{(i)}(T_\infty^{(i)})  <0, \quad p_{v,\infty}^{(i)}(T_\infty^{(i)}) =0,\\
 \lim_{i\to \infty} u_0^{(i)}  =  \lim_{i \to \infty} T_\infty^{(i)} =  \lim_{i\to \infty} \rho(x_\infty^{(i)}) = +\infty.
\end{aligned}
$$
\end{lem}

\begin{proof}
We may choose initial conditions as in \eqref{init_conditions} so that the first hit to $\{u=0\}$ is perpendicular. Indeed, since $v_\infty^{(n)}(t)= \hat v(u_0) \cos (2\sqrt{2}t), \forall t,$  this occurs if the first hitting time $T_\infty(u_0)$ to $\{u=0\}$ coincides with $\frac{k \pi}{2\sqrt{2}}$ for some $k\in \N$. Since $\ell_\infty^+$ is the graph of a function defined on the whole $u$-axis, and $T_\infty(u_0)$ is strictly increasing to $+\infty$ as $u_0\to +\infty$, see Lemma \ref{lem_T_infty}, we find a sequence of  $u$-symmetric brake orbits $x_\infty^{(i)}\in \mathfrak{M}'_\infty$ with $u_0^{(i)},T_\infty^{(i)} \to +\infty$ as $i\to \infty$, as desired.  The limit $\rho(x_\infty^{(i)})\to +\infty$ as $i\to \infty$  follows  from \eqref{rho_x_infty}.
\end{proof}

We are ready to prove Theorem \ref{thm: nonresonant implies infy many}.

\subsection{Proof of Theorem \ref{thm: nonresonant implies infy many}}
\begin{proof}
Fix $N\in \Z^+,$  $C\gg 1$ large, and  $\epsilon>0$ small. Consider the sequence $x_\infty^{(i)}=(p_{v,\infty}^{(i)},p_{u,\infty}^{(i)},v_\infty^{(i)},u_\infty^{(i)})\in \mathfrak{M}'_\infty$ of $u$-symmetric brake orbits as in Lemma \ref{lem_sequence}.  Since $\ell_\infty^+$ is the graph of $v=\hat v_\infty(u)>0,$ and the flow of $K_\infty$ in the $(p_v,v)$-plane satisfies \eqref{motion_uv limit},  Lemma \ref{lem_T_infty} implies that there exist trajectories
$$
x_{\infty,\pm}^{(i)}=(p_{v,\infty,\pm}^{(i)},p_{u,\infty,\pm}^{(i)},v_{\infty,\pm}^{(i)},u_{\infty,\pm}^{(i)})\in \mathfrak{M}'_\infty, \quad i=1,2,\ldots,
$$
with $x_{\infty,\pm}^{(i)}(0)$ arbitrarily close to $x_\infty^{(i)}(0)$, so that
$$
\begin{aligned}
v_{\infty,\pm}^{(i)}(0)  = \hat v_\infty(u_{\infty,\pm}^{(i)}(0)), \quad u_{\infty,-}^{(i)}(0) & < u_{\infty}^{(i)}(0) < u_{\infty,+}^{(i)}(0),\\
p_{v,\infty,+}^{(i)}(T_{\infty,+}^{(i)})\cdot p_{v,\infty,-}^{(i)}(T_{\infty,-}^{(i)}) & <0,
\end{aligned}
$$
where the first hitting time $T_{\infty,\pm}^{(i)}$ to $\{u=0\}$ is arbitrarily close to $T_\infty^{(i)}$.

Fix $M>0$ large. By Lemma \ref{lem_Kinfty}, we know that the Hamiltonian $K$ converges in $C^\infty_{\rm loc}$ to $K_\infty$ as $\alpha \to +\infty$. Hence, on the compact set $\mathcal{K}_M:=[0,2]\times [-M,M]\subset \R^2,$ the zero velocity curve $\ell^+$ associated with $\mathfrak{M}'$ is the graph of a function $v=\hat v(u), u\in [-M,M],$ that converges in $C^\infty$ to $\hat v_\infty$ as $\alpha \to +\infty$. In particular, for each $i$ satisfying  $(v_{\infty,\pm}^{(i)}(0),u_{\infty,\pm}^{(i)}(0))\in \mathcal{K}_M$ and $\alpha$ sufficiently large we find Hamiltonian trajectories
$$
x^{(i)}_\pm=(p_{v,\pm}^{(i)},p_{u,\pm}^{(i)},v_\pm^{(i)},u_\pm^{(i)})\in \mathfrak{M}',
$$  starting from points arbitrarily close to $x_{\infty,\pm}$, and satisfying the following conditions
$$
v_{\pm}^{(i)}(0) = \hat v(u_\pm^{(i)}(0)), \quad u_{-}^{(i)}(0)< u_{\infty}^{(i)}(0) < u_{+}^{(i)}(0).
$$
Moreover, their first hitting time $T_{\pm}^{(i)}$ to $\{u=0\}$ is arbitrarily close to $T_{\infty,\pm}^{(i)}$, and
\begin{equation}
p_{v,+}^{(i)}(T_{+}^{(i)})\cdot p_{v,-}^{(i)}(T_{-}^{(i)}) <0.
\end{equation}

 We know that every trajectory of $X_K$ starting from a point in $\{u>0\}$ hits $\{u=0\}$ transversely, with uniformly bounded hitting time to $\{u=0\}$ if $u_0$ lies in a fixed bounded set. Hence, for   $\alpha$ sufficiently large, there exists a trajectory
$$
x^{(i)}=(p_{v}^{(i)},p_{u}^{(i)},v^{(i)},u^{(i)})\in \mathfrak{M}',
$$
starting from a point arbitrarily close to $x_\infty^{(i)}(0)$, and satisfying
$$
v^{(i)}(0) = \hat v(u^{(i)}(0)), \quad u_{-}^{(i)}(0)< u^{(i)}(0) < u_{+}^{(i)}(0), \quad p_v(T^{(i)})=0.
$$
Here, $T^{(i)}$ is the  first hitting time  to $\{u=0\}$.
Such an orbit is a $u$-symmetric brake orbit with period $4T^{(i)}$.
Since $K\to K_\infty$ in $C^\infty_{\rm loc}$ as $\alpha \to +\infty$, we have $T^{(i)} \to T^{(i)}_\infty$ and  $x^{(i)} \to x^{(i)}_\infty$ in $C^\infty$ as $\alpha \to +\infty$. By continuity of the rotation number under small perturbations,  $\rho(x^{(i)}) \to \rho(x^{(i)}_\infty) = 1 + \frac{4\sqrt{2}}{\pi}T^{(i)}_\infty$ as $\alpha \to +\infty$.

Finally, take  $N_0>N$, and fix $M>0$ sufficiently large so that the number of $u$-symmetric brake orbits $x^{(i)}_\infty\in \mathfrak{M}'_\infty$, starting from a point in $\mathcal{K}_M$ and with rotation number $>C$, is  $N_0.$ Take $\alpha_0\gg 0$ sufficiently large so that for $\alpha >\alpha_0$, there exist   $N_0$ $u$-symmetric brake orbits $x^{(i)}\in \mathfrak{M}'$ arbitrarily close to those $x_\infty^{(i)}$, and with rotation number $>C$. The corresponding trajectories
$$
\zeta_{z_i}(t):= \psi^{-1}(x^{(i)}(t)) \in \mathfrak{M},
$$
are the desired simple $z$-symmetric brake orbits. Their rotation numbers coincide with those of $x^{(i)}$. Since  $\rho_e=\rho(\zeta_e)$  is $\epsilon$-close to $1+2\sqrt{2}>1$ for $\alpha$ sufficiently large, and $\rho_{z_i}=\rho(\zeta_{z_i})>C \gg 0$, we have $\rho(\zeta_{z_i})^{-1} + \rho(\zeta_e)^{-1} \neq 1.$ This implies that
the Hopf link $L_i=\zeta_{z_i}\cup \zeta_e \subset \mathfrak{M}$ is non-resonant for every $i$. This proves (i). Item (ii) immediately follows from  Theorem \ref{thm_PB}. This finishes the proof of Theorem \ref{thm: nonresonant implies infy many}.\end{proof}

\begin{rem}\label{general setting}
More generally, there exists  a family of limiting systems as $\alpha \to +\infty$. Fix $c\in (0,+\infty)$, and take $\varpi^2=1/2+(4-c)\alpha^{-1}$. Then
\begin{equation*}
\begin{aligned}
4m^2 & =2c\alpha^{-1}+16\alpha^{-2}\rightarrow 0,\\ \alpha m^2 & =(2c+16\alpha^{-1})/4\to c/2,\\
\frac{2\varpi^2-2n}{4m^2} & =-\frac{c-2}{c+8\alpha^{-1}}\rightarrow \frac{2-c}{c},\\
\frac{2n-1}{2m^2} & =\frac{2}{\alpha m^2}\to \frac{4}{c},
\end{aligned}
\end{equation*}
as $\alpha \to +\infty$.
The limiting Hamiltonian becomes
\begin{eqnarray*}
K_\infty=\frac{p_r^2+p_z^2}{2}+\frac{4(4-c)}{c}+4v^2-\frac{8}{c\sqrt{\frac{1}{4}+cu^2}}.
\end{eqnarray*}
The behavior on the $(p_v,v)$-plane is the same as before. On the $(p_u,u)$-plane  we have
\begin{eqnarray*}
\left\{\begin{aligned}
& \dot  u  = p_u\\
& \dot p_u  = -\frac{8u}{(1/4+cu^2)^{3/2}}.
\end{aligned}
\right.
\end{eqnarray*}
 Consider the initial conditions $p_{u,0}=p_{v,0}=0$ and
$$
(v_0,u_0)\in l_\infty=\{(v,u):\ 4(4-c)/c+4v^2=8c^{-1}(1/4+cu^2)^{-1/2}\}.
$$
We have the following four cases:
\begin{itemize}
\item[(a)] $c\in(0,4)$. In this case, $8c^{-1}(1/4+cu_0^2)^{-1/2}\in(4(4-c)/c,16/c)$, which implies that
$v_0\in [-1,1]$ and $u_0\in [-u_{\text{max}},u_{\text{max}}]$, where $u_{max}=\frac{\sqrt{8-c}}{2(4-c)}$.
\item[(b)] $c\in (4,\infty)$. In this case, $v_0\in[-1,-\sqrt{\frac{c-4}{c}})\cup(\sqrt{\frac{c-4}{c}},1]$, where $v_0$ tends to $\sqrt{\frac{c-4}{c}}$ as $|u_0|\rightarrow +\infty$.
\item[(c)] $c\rightarrow 0$. Since $(v_0,u_0)\in l_\infty$ satisfies
$$(4+(v^2-1)c)^2=\frac{16}{1+4cu^2}=16-64cu^2+16(4cu^2)^2+\cdots,$$
  $l_\infty$  converges to the ellipsoid $v^2+8u^2=1$ as $c\to 0$. In particular, $u(t)$ converges to $u_0\cos(8t)$ in $C^\infty_{\text{loc}}$, and thus the rotation number of the periodic orbit $(v,u)=(v_0\cos(2\sqrt{2}t),0)$ converges to $1+2\sqrt{2}$, which coincides the rotation number of Euler orbit with $(\beta,\mathfrak{e})=(1,0)$.
\item[(d)] $c=4$. This case coincides with the situation we considered before, i.e. $v_0\in [-1,0)\cup(0,1]$ and $u_0\in(-\infty,+\infty)$.
\end{itemize}
As in Lemma \ref{lem_T_infty}, the hitting time function is
\begin{equation*}
T_\infty(u_0)=\frac{\sqrt{c}}{4}\int^{u_0}_0\big(\frac{1}{\sqrt{\frac{1}{4}+cu^2}}-\frac{1}{\sqrt{\frac{1}{4}+cu_0^2}}\big)^{\frac{1}{2}}du,
\end{equation*}
which is also a strictly increasing function on $[0,+\infty)$ from $\frac{16}{\pi}$ to $+\infty$ (or to $T_\infty(u_{max})$ on $[0,u_{max}]$). Analogous to Lemma \ref{lem_sequence}, there exists a finite number of $u$-symmetric brake orbits in $K_\infty^{-1}(0)$ for any $c\in(0,4)$. Once $c\geq 4$,  there exist infinitely many $u$-symmetric brake orbits in the limiting system. Finally, using the perturbation argument in the proof of Theorem \ref{thm: nonresonant implies infy many}, one can find a similar amount of $z$-symmetric simple brake orbits in $\mathfrak{M}$.
\end{rem}
\begin{rem}
The estimate \eqref{rho_x_infty} implies that $\mathfrak{M}_\infty'=K_\infty^{-1}(0)$ is actually dynamically convex. However,  $\mathfrak{M}_\infty'$ is not convex in coordinates $(p_v,p_u,v,u)$, see Figure \ref{picture of K2}.
Indeed, we find a canonical mapping so that $\mathfrak{M}_\infty'$ becomes strictly convex.
\begin{figure}[ht]
\centering
\includegraphics[width=0.7\textwidth]{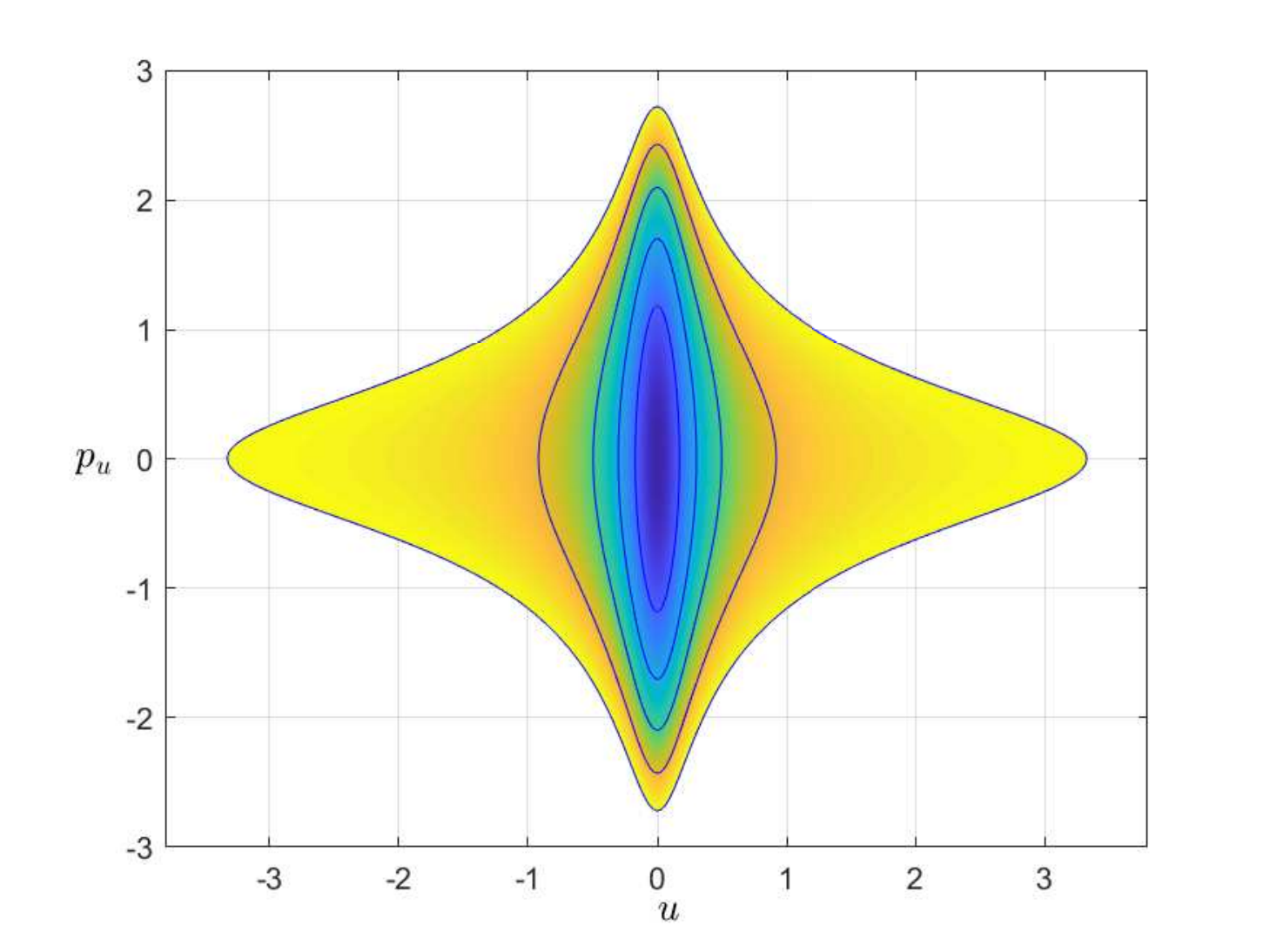}
\caption{ The level sets of $K_2$.}
\label{picture of K2}
\end{figure}
Firstly, since $K_2$ corresponds to an integrable system with 1-dimensional freedom, under the action-angle coordinates $(I,\phi)$, one can rewrite $K_2(p_u,u)$ as $h_u(I)$ and we obtain $\dot I=0,\ \dot\phi\equiv\partial_Ih_u$. The symplectic change of coordinates
$$(I,\phi)\rightarrow (x,y)=((I/\pi)^{1/2}\cos(2\pi\phi),(I/\pi)^{1/2}\sin(2\pi\phi)),$$
changes the Hamiltonian to $h_u(\pi(x^2+y^2))$. Direct computations show that
\begin{equation}\label{equ: hessen of h_u}
\nabla^2h_u(x,y)=(2\pi)^2h_u''(I)\begin{pmatrix}x^2 & xy\\ xy & y^2\end{pmatrix}+ 2\pi h_u'(I)\begin{pmatrix}1 & 0 \\ 0 & 1 \end{pmatrix}.
\end{equation}
Let $(p_u,u)(t)$ be a solution on $K_2^{-1}(h_u)$, where $h_u\in(-4,0)$. The period $T$ becomes
$$T=4T_\infty=2\int_{0}^{u_0}(\frac{1}{\sqrt{\frac{1}{4}+4u^2}}-\frac{1}{\sqrt{\frac{1}{4}+4u_0^2}})^{-\frac{1}{2}}du$$
where $u_0=(h_u^{-2}-4^{-2})^{\frac{1}{2}}$ (i.e. $h_u=-\frac{2}{\sqrt{\frac{1}{4}+4u_0^2}}$) and $T_\infty$ is the hitting time function in (\ref{equ: the hitting time}). Since $T$ also satisfies $T(I)={\dot \phi}^{-1}=(\partial_I h_u)^{-1}>0$, we have
$$\begin{aligned}
h_u''(I)=\frac{d(T^{-1})}{dI}=-\frac{\partial_{u_0} T}{T^{2}}\cdot \partial_{h_u}u_0\cdot h_u'(I)=\frac{4\partial_{u_0} T}{T^3h_u^2\sqrt{16-h_u^2}},
\end{aligned}$$
where $\partial_{u_0} T>0$, see Lemma \ref{lem_T_infty}.
Finally, \eqref{equ: hessen of h_u} becomes
$$
\begin{aligned}
\nabla^2h_u(x,y)=\frac{16\pi^2 \partial_{u_0} T}{T^3h_u^2\sqrt{16-h_u^2}}\begin{pmatrix}x^2 & xy\\ xy & y^2\end{pmatrix}+ \frac{2\pi}{T}I_2,
\end{aligned}
$$
which is positive definite for every $h_u\in(-4,0)$. Using  \eqref{equ: diff of the hitting time}, we obtain $\nabla^2h_u(x,y)\rightarrow 32I_2$ as $h_u\rightarrow -4$, and $\nabla^2h_u(x,y)\sim O(T^{-1})I_2\rightarrow 0$ as $h_u\rightarrow 0$.
\end{rem}
\section{Symmetric periodic orbits}\label{sec:twist condition and rational rotation}

In this section, we explore the symmetries of the first return map to the disk-like global surface of section in order to prove Theorems  \ref{thm infy z brake} and \ref{thm: rational rotation implies infy many}.


Let $\mathcal{O}=\{\Sigma_{s}\subset \mathfrak{M},\ s\in\mathbb{R}/\mathbb{Z}\},$ be the open book decomposition with disk-like pages
$$
\Sigma_s : = \{(p_r,p_z,r,z)\in \mathfrak{M}: p_z + zi \in \R^+e^{2\pi si}\}, \quad \forall s \in \R / \Z,
$$
bounded by the Euler orbit $\zeta_e$, as in Remark \ref{rem_openbook}. Recall that $\Sigma_0\subset \{z=0\}$ coincides with $F$, see \eqref{Eulerplane}, and each $\Sigma_s$ is a global surface of section for the Hamiltonian flow in $\mathfrak{M}$.

Denote by $\Pi_{p_r,r}:\R^4 \to \R^2$ the  projection  $(p_r,p_z,r,z) \mapsto (p_r,r)$. Then $$\Pi_{p_r,r}(\Sigma_s)=\Pi_{p_r,r}(\Sigma_0), \quad \forall s,$$ is a disk-like region
$$
\Upsilon:=\Pi_{p_r,r}(\mathfrak{M})\subset \R^2,
$$
bounded by $\Pi_{p_r,r}(\zeta_e) \equiv \zeta_e$.

Given $\zeta \in \Sigma_0^o$ and $s\in[0,1]$, there exists a smallest $t^+_s(\zeta) \geq 0,$  such that $\varphi_{t^+_s(\zeta)}(\zeta) \in \Sigma_s^o$. We call  $t^+_s$ the $s$-hitting time  and
\begin{equation}\label{equ: s-page reaching map}
g_s:=\varphi_{t^+_s(\cdot)}(\cdot) : \Sigma_0^o \to \Sigma_s^o,\ \forall s\in[0,1],
\end{equation}
the $s$-hitting map (a diffeomorphism, indeed) from $\Sigma_0^o$ to $\Sigma_s^o$. Notice that $t^+_0 \equiv 0$ and $g_0=\text{Id}$ is the identity map. The mapping $g_1$ is the so-called Poincar\'e first return map.

Since the tangent space $T_{\zeta_e}\Sigma_s$ is generated by the vectors
\begin{equation}\label{Tsigma_0 on the bdy}
\dot \zeta_e \quad \mbox{ and } \quad \cos(2\pi s)\partial_{p_z}+\sin(2\pi s)\partial_z,
\end{equation}
and the linearized flow along $\zeta_e$ positively  rotates the vectors in the $(p_z,z)$-direction,  we can continuously extend $g_s$ to a homeomorphism  $\Sigma_0 \to \Sigma_s$, also denoted $g_s$. In general, this extension is not always possible, see \cite[Chapter 9]{FvK18}.

\begin{figure}[ht]
\centering
\includegraphics[width=0.7\textwidth]{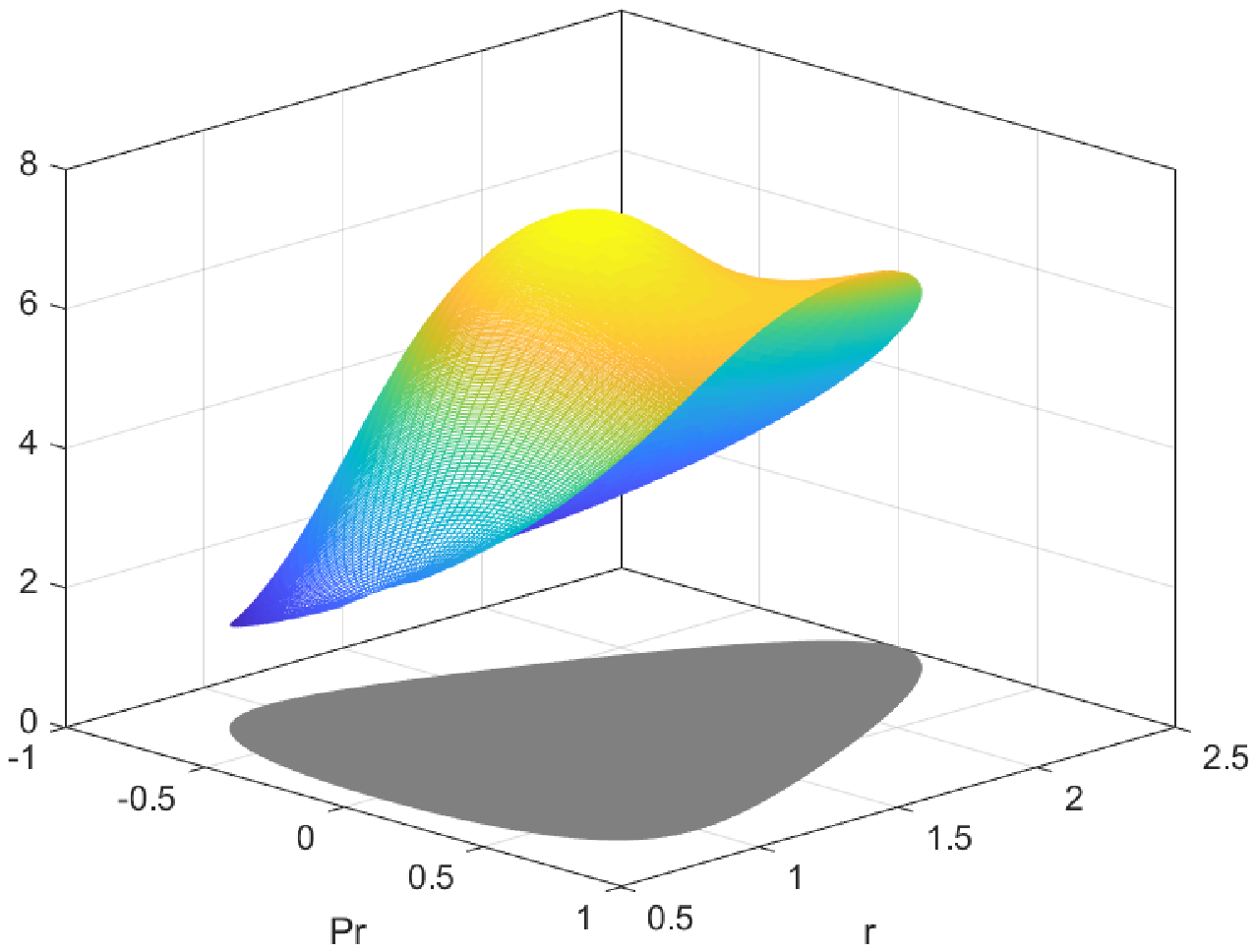}
\caption{The graph of $t^+_1$ over the $(p_r, r)$-plane with $\beta = \mathfrak{e}=0.6$.}
\label{fig:t+}
\end{figure}

Let $\mathcal{P}_s:=\Pi_{p_r,r}|_{\Sigma_s}: \Sigma_s\to \Upsilon$. Then $\mathcal{P}_s$ is a homeomorphism that restricts to a symplectomorphism $\Sigma_s^o \to \Upsilon \setminus \zeta_e$ which preserves the symplectic form $dp_r \wedge dr$. This follows from $$
dp_z \wedge dz|_{\Sigma_s} \equiv 0, \quad \forall s.$$
In particular,
\begin{equation}\label{def of tilde gs}
\tilde g_{s}:=\mathcal{P}_s\circ g_s\circ \mathcal{P}_0^{-1}: (\Upsilon, dp_r\wedge dr)\to (\Upsilon, dp_r\wedge dr)
\end{equation}
is a homeomorphism that restricts to a symplectomorphism $\Upsilon \setminus \zeta_e \to \Upsilon \setminus \zeta_e$.


Let $g_{s_1,s_2}:=g_{s_2}\circ g^{-1}_{s_1},\ 0 \leq s_1\leq s_2 \leq 1$, be the hitting map from $\Sigma_{s_1}$ to $\Sigma_{s_2}$, and let $\tilde g_{s_1,s_2}:=\tilde g_{s_2}\circ \tilde g^{-1}_{s_1}$. The $z$-symmetry of the potential $V$ implies that $\tilde{g}_{1/2} = \tilde{g}_{1/2,1}$. For simplicity, we write $\bar{g}:=\tilde g_{1/2}$ and $\check g:=\tilde g_1 =\tilde g_{1/2,1} \circ \tilde g_{1/2} = \tilde g_{1/2}^2 = \bar g^2$.

If $\zeta_z \subset \mathfrak{M}\setminus \zeta_e$ is  a $z$-symmetric simple brake orbit with $\zeta_z(0)=\zeta_z(T)\in \Sigma_0^o$, then $\Pi_{p_r,r}(\zeta_{z}(0))=\Pi_{p_r,r}(\zeta_{z}(T/2))\in \{0\} \times \R$ is a fixed point of $\bar{g}$. 
\begin{figure}[ht]
\centering
\includegraphics[width=0.7\textwidth]{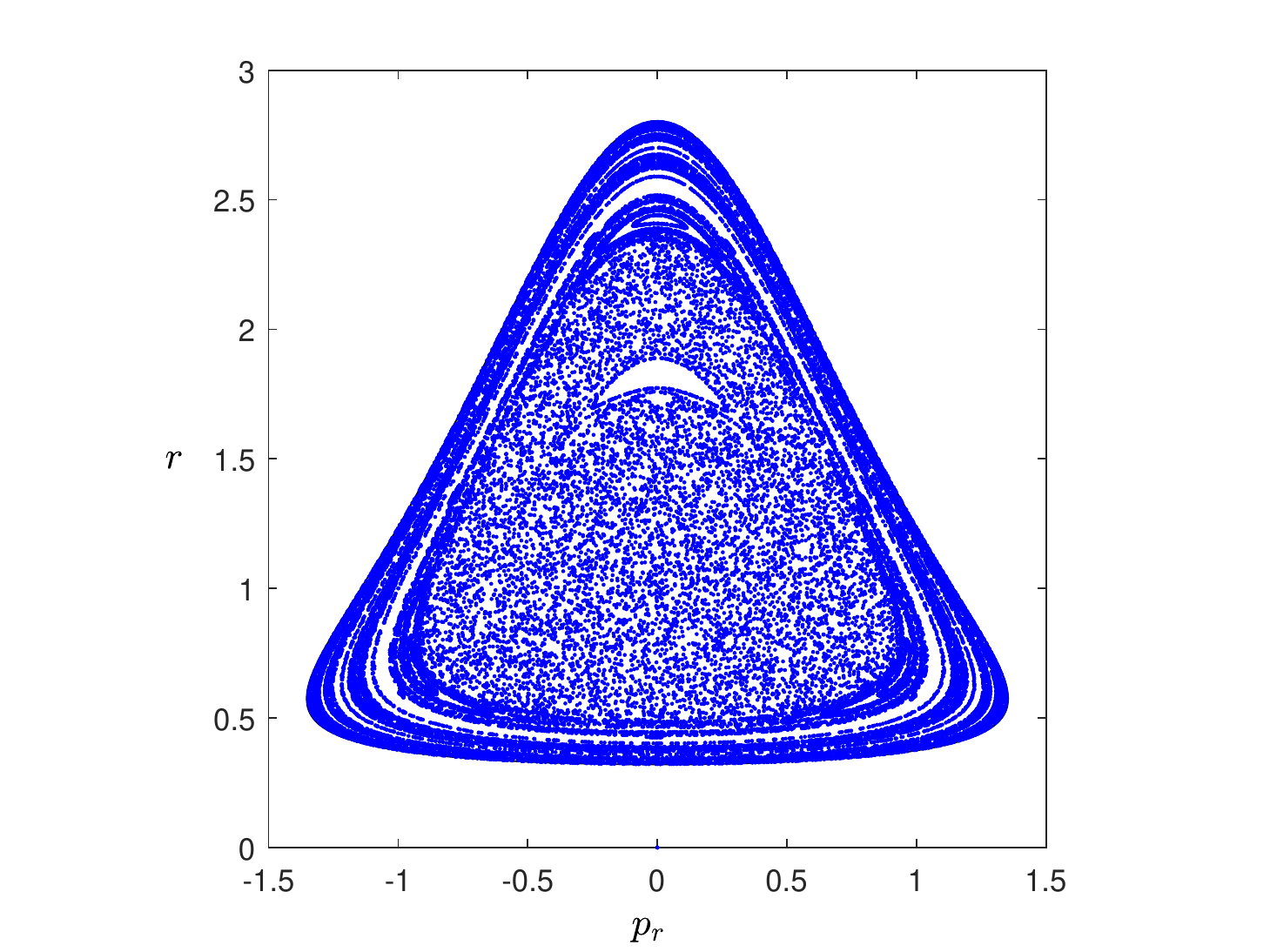}
\caption{The disk-like region $\Upsilon \subset \R^2$ bounded by $\Pi_{p_r,r}(\zeta_e)$ and some trajectories of the first return map $\bar g_1$  for $\beta=0.6$ and $\mathfrak{e}=0.795$.}
\label{fig: the dynamics of check g}
\end{figure}

Consider the involution $\mathcal{N}:(p_r,r) \mapsto (-p_r, r)$, and let $f: \Upsilon \to \Upsilon$ be an orientation preserving homeomorphism of the compact disk-like  region $\Upsilon \subset \R^2$. Following \cite{kang18},
 we assume that $f$ is reversible, that is  $f \circ \mathcal{N} \circ f=\mathcal{N}$. Let $z\in \Upsilon$ be a periodic point of $f$ with least period $k>0$, and let ${\rm Orb}(f,z) = \{z,f(z),\ldots,f^{k-1}(z)\}\subset \Upsilon$ be the orbit through $z$ under $f$. We call ${\rm Orb}(f,z)$ symmetric if $\mathcal{N}({\rm Orb}(f,z))={\rm Orb}(f,z).$ Observe that the reversibility of $f$ implies that if ${\rm Orb}(f,z) \cap {\rm Fix}(\mathcal{N}) \neq \emptyset,$ then ${\rm Orb}(f,z)$ is symmetric.
We say that a symmetric orbit ${\rm Orb}(f,z)$ is odd if there exists an odd number $l=l(z)>0$ so that $f^l(z) = \mathcal{N}(z)$. One easily checks that this definition does not depend on the point in ${\rm Orb}(f,z)$.

\begin{prop} \label{prop reversible g bar}
$\bar{g}:\Upsilon \to \Upsilon$ is a reversible mapping,  i.e. $\bar{g} \circ \mathcal{N} \circ \bar{g}=\mathcal{N}.$
\end{prop}
\begin{proof}
We check the reversibility condition on the interior of $\Upsilon$. The reversibility of $\bar g$ then follows from $\mathcal{N}$ and the continuity of  $\bar g$.

Let $q\in\Upsilon^o=\Upsilon \setminus \partial \Upsilon$, let $\zeta(t)$ be the orbit starting from $\mathcal{P}^{-1}_0(q)$, and let $t_{1}>0$ be the first hitting time of $\zeta(t)$ to $\Sigma_{1/2}$,  i.e. $t_1 = t^+_{1/2}(\zeta(0))$ and $\zeta(t_{1})=g_{1/2}(\zeta(0))$. Let
\begin{equation}
\zt_{1}(t) := R_1 \cdot  \zt(t_1 -t),\quad \mbox{ where } R_1=\mathrm{diag}(-1,-1,1,1).\nonumber
\end{equation}
Then $\zeta_1(t)\in \mathfrak{M}, \forall t,$ is also a Hamiltonian trajectory of $H$, and
$$
\begin{aligned}
\zt_{1}(0)& =\mathcal{P}^{-1}_{0} \circ \mathcal{N} \circ \mathcal{P}_{1/2}(\zeta(t_1))=\mathcal{P}^{-1}_{0}\circ \mathcal{N}\circ \mathcal{P}_{1/2}\circ  g_{1/2}(\zeta(0))=\mathcal{P}^{-1}_{0}\circ \mathcal{N}\circ \bar g(q),\\
\zt_{1}(t_{1})& =\mathcal{P}^{-1}_{1/2}\circ \mathcal{N}\circ \mathcal{P}_0(\zt(0))=\mathcal{P}^{-1}_{1/2}\circ \mathcal{N}(q).
\end{aligned}
$$
Therefore, $t_{1}$ is also the first hitting time of $\zt_{1}(t)$ to $\Sigma_{1/2}$, and
$$
\mathcal{P}^{-1}_{1/2} \circ \mathcal{N}(q)=\zt_{1}(t_{1})=g_{1/2}(\zt_{1}(0))=g_{1/2} \circ \mathcal{P}^{-1}_{0}\circ \mathcal{N}\circ \bar g(q).
$$
Applying $\mathcal{P}_{1/2}$ to the identity above and using that
 $\bar g = \bar g_{1/2} = \mathcal{P}_{1/2} \circ g_{1/2} \circ \mathcal{P}_0^{-1}$, we obtain  $\mathcal{N}(q)=\bar g\circ \mathcal{N}\circ \bar g(q),$ as desired. 
\end{proof}

The proposition above implies that $\bar{g}^k$ is reversible for any $k\in\mathbb N$. Next, we provide the relation between symmetric points of $\bar{g}$ and $z$-symmetric (or brake) periodic orbits on $\mathfrak{M}$.

\begin{prop}\label{prop: symmetric orbits and periodic points}
Let $\zeta\subset \mathfrak{M}\setminus \zeta_e$ be a periodic orbit  with $\zeta(0)=\mathcal{P}_0^{-1}(p)\in\Sigma_0^o, p \in \Upsilon\setminus \partial \Upsilon$. Then
\begin{itemize}
\item[(i)] $\zt$ is a brake orbit $\Leftrightarrow$ ${\rm Orb}(\bar g,p)$ is an odd symmetric periodic orbit.

\item[(ii)] $\zt$ is  $z$-symmetric  and hits $\{z=p_r=0\}\Leftrightarrow$ ${\rm Orb}(\check g,p)$ is  symmetric.


\item[(iii)] $\zt$ is a $z$-symmetric brake orbit $\Leftrightarrow$ ${\rm Orb}(\check g,p)$ is a symmetric periodic orbit  with odd minimal period.
\end{itemize}
\end{prop}

\begin{rem} There may exist $z$-symmetric periodic orbits $\zeta(t) \subset \mathfrak{M},\zeta(0)=\mathcal{P}^{-1}(p),p\in \Upsilon,$ which do not hit $\{z=p_r=0\}$ and for which ${\rm Orb}(\check g,p)$ and ${\rm Orb}(\bar g,p)$ are not symmetric.
\end{rem}

\begin{proof}
(i) Let $\zeta(t)=(p_r,p_z,r,z)(t)$ be a brake orbit with $\zeta(0) = \mathcal{P}_0^{-1}(p)\in \Sigma_0^o, p\in \Upsilon$. As before, let $R_1: \mathfrak{M} \to \mathfrak{M}$ be the reflection $R_1(p_r,p_z,r,z) = (-p_r,-p_z,r,z)$. Then there exist $n\in \N$ and unique least times  $0<t_1<t_1' < \ldots t_{n-1} < t_{n-1}' < t_n$ such that
$$
\begin{aligned}
\zeta(t_i) & \in \Sigma_{1/2}^o, \quad & \forall i=1,\ldots,n,\\
\zeta(t_i') & \in \Sigma_0^o, \quad & \forall i=1,\ldots, n-1,\\
\zeta(t_n) & = R_1(\zeta(0)). &
\end{aligned}
$$
In this case $\bar g^{2n-1}(p) = \mathcal{N}(p)$. Hence ${\rm Orb}(\bar g,p)$ is symmetric and odd.

Now let ${\rm Orb}(\bar g,p)$ be symmetric and odd, and let $l>0$ be an odd number so that $\bar g^l(p) = \mathcal{N}(p).$ Denote $p=(p_r(0),r(0))$ and hence $\bar g^l(p) = (-p_r(0),r(0))$. Since $l$ is odd, there exists $t_l>0$ such that $(-p_r(0),r(0)) = \mathcal{P}_{1/2}(\zeta(t_l)).$ Hence $\zeta(t_l)=(-p_r(0),-p_z(0),r(0),0)=R(\zeta(0))$. The reversibility of the flow implies that $\zeta$ is a brake orbit. In fact,  $p_r(t_l/2)=p_z(t_l/2) = 0.$

(ii) Let $\zeta(t)=(p_r,p_z,r,z)(t)$ be a $z$-symmetric periodic orbit that hits $\{z=p_r=0\}$. We may assume that $\zeta(0) = \mathcal{P}_0^{-1}(p)\in \Sigma_0^o$, where $z(0)=p_r(0)=0$. Let $t_1>0$ be such that $z(t_1)=0$.
The $z$-symmetry and the reversibility of the flow imply that $z(-t_1)=0$, $r(-t_1)=r(t_1),$ $p_r(-t_1) =-p_r(t_1)$ and $p_z(t_1)=p_z(-t_1)$. This implies that ${\rm Orb}(\check g,p)$ is symmetric.

Let $p=(p_r(0),r(0))\in \Upsilon$ be such that ${\rm Orb}(\check g,p)$ is symmetric, and let $\zeta(t)=(p_r,p_z,r,z)(t) \in \mathfrak{M}$ be the trajectory satisfying $\zeta(0)=(p_r(0),p_z(0),r(0),0)=\mathcal{P}^{-1}_0(p)$. If $p_r(0)=0$ then $\zeta$ hits $\{z=p_r=0\}$ and hence the $z$-symmetry and the reversibility of the flow implies that $\zeta$ is $z$-symmetric.  If $p_r(0)\neq 0$ then since ${\rm Orb}(\check g,p)$ is symmetric, we find $t_1>0$ such that $z(t_1)=0$, $r(t_1)=r(0)$, $p_r(t_1)=-p_r(0)$ and $p_z(t_1)=p_z(0)>0$. The $z$-symmetry and the reversibility of the flow imply that $z(t_1/2)=0$ and $p_r(t_1/2)=0$. Thus we fall into the previous case and thus $\zeta$ hits $\{z=p_r=0\}$ and is $z$-symmetric.

(iii) Let $\zeta(t)=(p_r,p_z,r,z)(t)\in \mathfrak{M}$ be a $z$-symmetric brake orbit. Then there exist $t_2>t_1>0$ such that $z(t_1)=-z(t_2)$, $r(t_1)=r(t_2)$ and $p_r(t_1)=p_r(t_2)=p_z(t_1)=p_z(t_2)=0.$ By the $z$-symmetry and the reversibility of the flow, we obtain $z((t_1+t_2)/2)=p_r((t_1+t_2)/2)=0$. Reasoning as in (ii), we conclude that ${\rm Orb}(\check g,p)$ is symmetric, where $p=\mathcal{P}_0(\zeta(0)).$ Since $\zeta$ is a brake orbit, we find an odd number $l>0$ so that $\bar g^l(p) = p$. This implies $\check g^l(p)=p,$ and the minimal period of $p$ is also odd.

Let $p=\mathcal{P}_0(\zeta(0))\in \Upsilon \setminus \partial \Upsilon$ be such that ${\rm Orb}(\check g,p)$ is symmetric with odd minimal period $l>0$. In particular, ${\rm Orb}(\check g,p)$ has precisely an odd number of points and because ${\rm Orb}(\check g,p)$ is symmetric, it hits $p_r=0$ an odd number of times.  We know from (ii) that $\zeta$ is $z$-symmetric and there exists $t_1>0$ so that $z(t_1)=0$ and $p_r(t_1)=0$. If there exists a least $a>0$ with $z(t_1+a)=p_r(t_1+a)=0$ and $r(t_1+a)\neq r(t_1)$ then the $z$-symmetry and the reversibility of the flow imply that $2a>0$ is a period of $\zeta$. In this case ${\rm Orb}(\check g,p)$ intersects $\{p_r=0\}$ at precisely two distinct points, a
contradiction. Hence ${\rm Orb}(\check g,p)$ intersects $\{p_r=0\}$ at precisely one single point $\mathcal{P}_0(\zeta(t_1)).$ Again by the $z$-symmetry and the reversibility of the flow, this implies that there exists $a>0$ such that $z(t_1+a)=0$, $p_r(t_1+a)=0$ and $p_z(t_1+a)=-p_z(t_1).$ This force $\zeta$ to be a brake orbit.
\end{proof}

\begin{rem} We can use  Propositions \ref{prop reversible g bar} and \ref{prop: symmetric orbits and periodic points}-(iii) to prove the existence of a $z$-symmetric brake orbit $\zeta_z \subset \mathfrak{M}$ forming a Hopf link with the Euler orbit $\zeta_e$,
 see Theorem \ref{main1}-(ii).  Indeed, observe that $\bar g\circ \mathcal{N}$ and $\mathcal{N}$ are both orientation reversing involutions, i.e. $(\bar g \circ \mathcal{N})^2=\mathcal{N}^2=\mathrm{Id}$. Hence their fixed sets $\text{Fix}(\bar g \circ \mathcal{N})$ and $\text{Fix}(\mathcal{N})$ are properly embedded curves in  $\Upsilon\setminus \partial \Upsilon$, see \cite[Lemma 9.5.1]{FvK18}. Each such curve separates $\Upsilon\setminus \partial \Upsilon$ into two regions of equal area since both maps reverse the finite area form $dp_r \wedge dr$ on $\Upsilon$.  Hence  $\text{Fix}(\bar{g}\circ \mathcal{N})\cap \text{Fix}(\mathcal{N})\neq \emptyset$ and such an intersection point  $p\in \Upsilon \setminus \partial \Upsilon$ is such that $\text{Orb}(\check g,p)=\text{Orb}(\bar g,p)=\{p\}$ has period $1$.
\end{rem}

Let $\zeta_z\subset \mathfrak{M}$ be a $z$-symmetric simple brake orbit, whose existence is provided in Theorem \ref{main1}-(ii). The link $\zeta_z \cup \zeta_e$ is a Hopf link bounding an annulus-like global surface of section. We explore the disk-like global surface of section bounded by the Euler orbit, as in the previous section. In that case $\zeta_z$ gives rise to a fixed point $p\in \{p_r=0\}\subset \Upsilon$ of $\bar g$. Notice that $p$ is also a fixed point of $\check g$. We assume that the Hopf link $\zeta_z \cup \zeta_e$ is non-resonant. As we shall see below, this non-resonance condition implies a twist condition of $\check g$ which eventually provides the desired periodic points.

Recall that the zero velocity curve $\partial\mathcal{H}\subset \R^2$ splits as $\mathcal{B}_+\cap \mathcal{B}_-$, where $\mathcal{B}_+\subset \{z\geq 0\}$  and $\mathcal{B}_-\subset \{z\leq 0\}$. The projection of $\zeta_z(t)$ to the $(r,z)$ plane is a smooth $z$-symmetric simple arc connecting the interior of $\mathcal{B}_+$ to the interior of $ \mathcal{B}_-$. The end-points of this arc separate $\mathcal{B}_+$ and $\mathcal{B}_-$ into two closed sub-arcs $\mathcal{B}_\pm = \mathcal{B}_{\pm,1} \cup \mathcal{B}_{\pm,2}$, see Figure \ref{fig: CurvesB}.

\begin{figure}[ht]
 \centering
\includegraphics[width=0.7\textwidth]{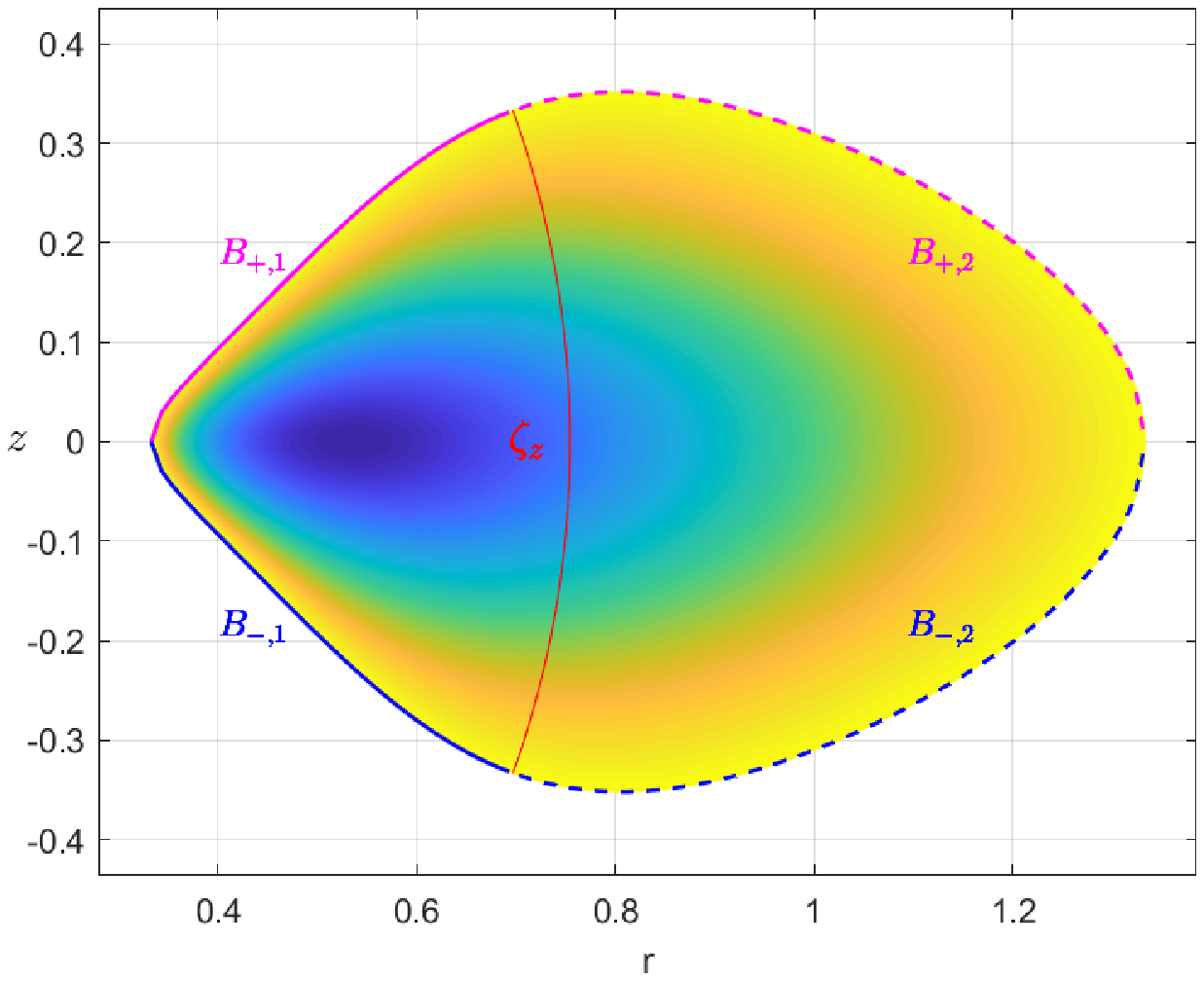}
\caption{ $\cb_{\pm, i}$ with $\beta = \mathfrak{e}=0.6.$ }
\label{fig: CurvesB}
\end{figure}

We see $\mathcal{B}_\pm$ and $\mathcal{B}_{\pm,i},i=1,2,$ as subsets of $\mathfrak{M}$ and forward flow them until they hit $\Sigma_{0}$ for the first time. In coordinates $(p_r,r)$, they give rise to arcs $\mathcal{C}_\pm,\mathcal{C}_{\pm,i}\subset \Upsilon,i=1,2.$ These are smooth simple arcs connecting $\partial \Upsilon$ to $p\in \text{Fix}(\mathcal{N})$. Notice that $p$ separates $\text{Fix}(\mathcal{N})$ into two closed line segments, denoted by $D_1$ and $D_2$, see Figure \ref{fig:CurvesCD}.

\begin{rem}\label{rem: continuous extension} In the discussion above, we have used that $\bar g$ and $\check g$ continuously extend to $\Upsilon$. To see this, we observe that the tangent space of $\Sigma_0,\Sigma_{1/2}\subset \{z=0\}$ at their boundaries $\zeta_e\subset \{z=p_z=0\}$ is spanned by the flow and the vector $\partial_{p_z}$. The transverse linearized flow along $\zeta_e$ is governed by the equation
\begin{equation}\label{linearpzz}
\left(\begin{array}{c} \dot a_{p_z} \\ \dot a_z \end{array} \right) = \left(\begin{array}{cc}0 & -\frac{4\alpha^{-1}(1+2\alpha)}{r_e(t)^3} \\ 1 & 0 \end{array} \right) \left(\begin{array}{c} a_{p_z} \\ a_z\end{array} \right),
\end{equation}
where $(a_{p_z} \ a_z)^T \equiv a_{p_z} \partial_{p_z} + a_z \partial _z$. Hence $a_z=0, a_{p_z} \neq 0 \Rightarrow \dot a_z \neq 0$, and this fact implies that the maps $\bar g$ and $\check g$ can be  continuously approximated by the linearized flow along $\zeta_e$. The maps $\bar g$ and $\check g$ along $\zeta_e$ correspond to points for which the linearized flow applied to $\partial_{p_z}$ rotates $\pi$ and $2\pi$, respectively.
\end{rem}

\begin{figure}[ht]
\centering
\includegraphics[width=0.7\textwidth]{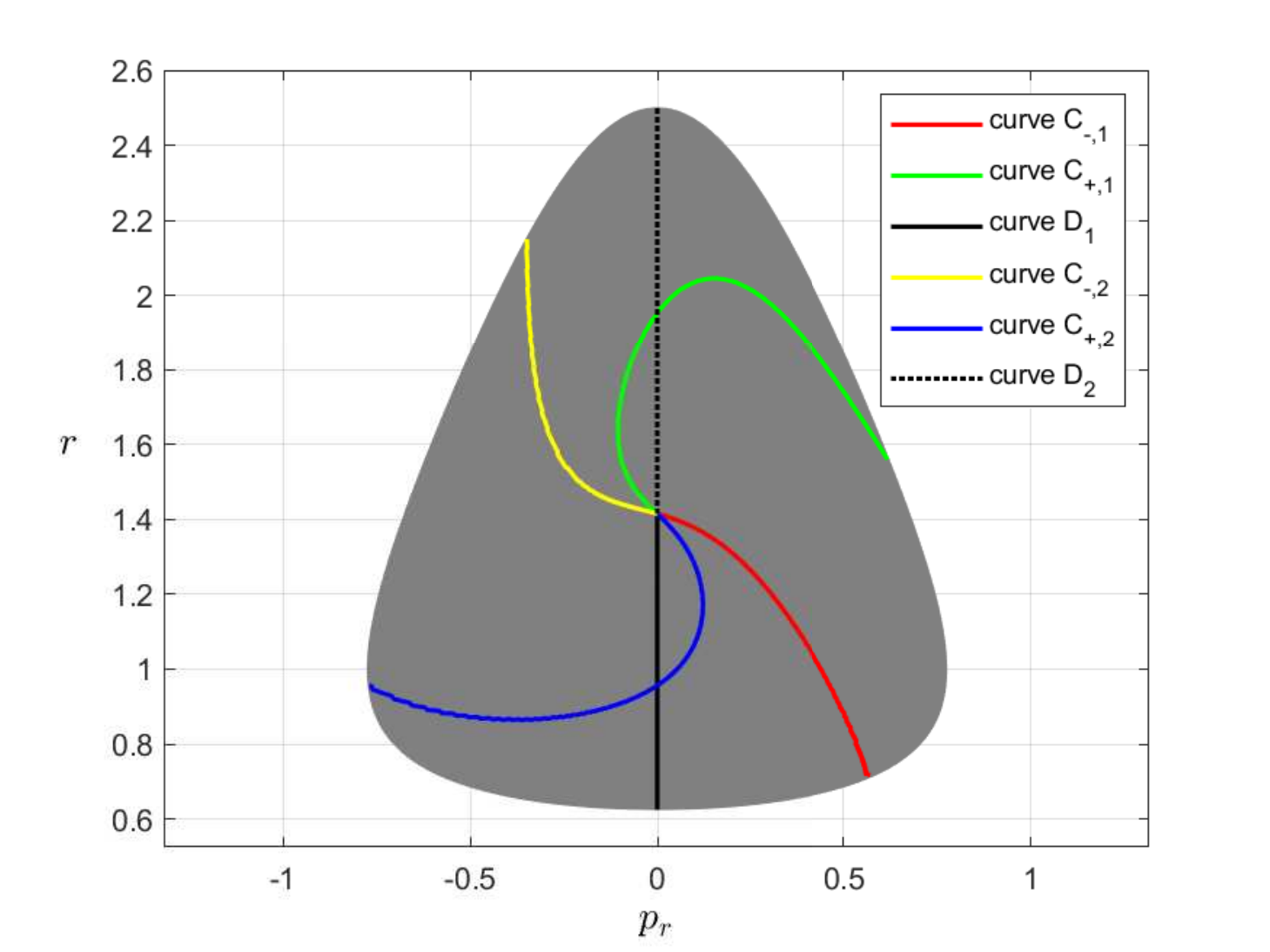}
\caption{ $\calc_{\pm, i}$ and $D_i$ with $\beta = \mathfrak{e}=0.6$. }
\label{fig:CurvesCD}
\end{figure}




\begin{lem} \label{lem:PerOrbitSigma}
Let $\zt\subset \mathfrak{M}\setminus (\zeta_e\cup \zeta_z)$ be an  orbit with $\zt(0) \in \Sigma_0^o$. Let $p_0=\mathcal{P}_0(\zeta(0))\in \Upsilon \setminus (\partial \Upsilon\cup \{p\})$. Then
\begin{enumerate}

\item[(i)] if $p_0 \in \check{g}^n(\calc_\pm) \cap D$ for some $n>0$, then $\zt$ is a $z$-symmetric  brake orbit.

\item[(ii)] if $p_0 \in \check{g}^{n}(\calc_{\pm, i_0}) \cap \calc_{\mp, i_1}$ for some $n>0$, with $i_0 \ne i_1$, then $\zt$ is a non-$z$-symmetric type-I periodic brake orbit.

\item[(iii)] if $p_0 \in \check{g}^n(\calc_{\pm, i_0}) \cap \calc_{\pm, i_1}$ for some $n>0,$ with $i_0 \ne i_1$, then $\zt$ is a type-II periodic brake orbit. In particular $\zeta$ is non-$z$-symmetric.

\item[(iv)] if $p_0 \in \check{g}^n(D_{i_0}) \cap D_{i_1}$ for some $n>0$, with $i_0 \ne i_1$, then $\zt$ is a  $z$-symmetric non-brake periodic orbit.
\end{enumerate}
\end{lem}

\begin{proof}
(i) By assumption, $p_r(0)=z(0)=0$ and $\zt(t_0) \in \cb_\pm$ for some $t_0<0$. The $z$-symmetry and the reversibility of the flow implies that $\zeta$ is a $z$-symmetric brake orbit.

(ii) By assumption,  $\zt(t_0) \in \cb_{\pm, i_0}$ and $\zt(t_1) \in \cb_{\mp, i_1}$ for some $t_0<t_1<0$, with $i_0 \neq i_1$. The reversibility of the flow imply that $\zt$ is  a non-$z$-symmetric type-I brake orbit. Similarly, we prove (iii).

(iv) By assumption, there exists $t_0<0$ such that $ p_r(t_0)=z(t_0)=0$ and $p_r(t)^2+z(t)^2\neq 0,\ \forall t \in (t_0, 0). $ Since $i_0\neq i_1$, we also have $(r(0),0)\neq(r(t_0),0)$. The reversibility and the $z$-symmetry of the flow imply that $\zt$ is a $z$-symmetric periodic trajectory, with least period $2|t_0|$. If there exists  $t_1 \in (t_0, 0)$ with $\zt(t_1) \in \cb$, then the $z$-symmetry and the reversibility of the flow imply that $\zt(-t_1) \in \cb$ and thus $\zeta(\R) = \zeta([t_1,-t_1])$ does not contain $\zeta(t_0)$, a contradiction.
\end{proof}


Let $\tilde g_s: \Upsilon \to \Upsilon, s>0,$ be the family of area-preserving homeomorphisms (diffeomorphisms of $\Upsilon^o$), satisfying $\tilde g_{s+1}=\tilde g_s\circ \tilde g_1,\forall s>0,$ where $\tilde g_s$ is given in (\ref{equ: s-page reaching map}). Observe that $\tilde g_s$ also extends as a homeomorphism of $\Upsilon$ for every $s$.

We denote by ${\rm Rot}(p)\in \R$ the rotation number of $p$ associated with the family $\tilde g_s$ computed with respect to the frame $\{\partial_{p_r},\partial_r\}$. More precisely, the family $\tilde g_s, s\in [0,1]$ determines a family of symplectomorphisms $D\tilde g(p):\R^2 \to \R^2$ starting from the identity, and ${\rm Rot}(p)$ is the rotation number induced by such family. Similarly, we denote by ${\rm Rot}(\partial \Upsilon)\in \R$ the rotation number of $\check g|_{\partial \Upsilon}$  associated with the family $\tilde g_s,s\in[0,1]$. Indeed, the family $\tilde g_s|_{\partial \Upsilon},s\in [0,1],$ lifts to a family  $\tilde f_s:\R \to \R, s\in[0,1],$ starting from the identity, and ${\rm Rot}(\partial \Upsilon)$ is the rotation number of $\tilde f_1$.

\begin{lem}\label{lem:Rot}We have
$$
{\rm Rot}(p) = \rho_z - 1 \quad \mbox{ and } \quad {\rm Rot}(\partial \Upsilon) = \frac{1}{\rho_e-1},
$$
where $\rho_z$ and $\rho_e$ are the rotation numbers of  $\zeta_z$ and  $\zeta_e$, respectively.
\end{lem}

\begin{proof}
Recall that the global   frame $\{X_1,X_2\}$, transverse to the flow, is given by
$$
\begin{aligned}
 X_1 & = \partial_zV \partial_{p_r} - \partial_rV \partial_{p_z} +p_z \partial_r - p_r \partial_z,\\
X_2 & = -p_z \partial_{p_r} + p_r \partial_{p_z} + \partial_zV \partial_r - \partial_rV \partial_z.
\end{aligned}
$$
Since $(\partial_zV)^2 + p_z^2$ never vanishes along $\zeta_z$, the frame $\{X_1|_{\zeta_z},X_2|_{\zeta_z}\}$ non-trivially projects to a loop of frames $\{\hat X_1,\hat X_2\} \subset {\rm span}\{\partial_{p_r},\partial_r\}$, whose winding with respect to the frame $\{\partial_{p_r},\partial_r\}$ is $-1$ along the period of $\zeta_z$. Indeed, during  half of its period, while $\zeta_z$ moves from $\{z=0\}$ to a brake point and then back to $z=0$,  $\partial_zV$ does not change sign and $p_z$ is symmetric, so that $(\partial_zV,p_z)$ is an arc from $(0,a)$ to $(0,-a)$, monotone in the clockwise direction.   We conclude that $\rho_z = {\rm Rot}(p) + 1.$

Now the global frame $X_1,X_2$ restricted to $\zeta_e$ becomes $X_1|_{\zeta_e}=-\partial_rV\partial_{p_z}-p_r\partial_z$ and $X_2|_{\zeta_e} = p_r \partial_{p_z} - \partial_r V \partial_z$. We first observe that  $(\partial_rV,p_r)\in \R^2$ never vanishes along $\zeta_e$ and its winding number is $-1$. So the rotation number of $\zeta_e$ with respect to the frame $\{\partial_{p_z},\partial_z\}$ is $\rho_e - 1$.

For each $s$, the tangent space of $\Sigma_s$ along $\zeta_e$ is generated by the flow and the vector $v_s:=\cos (2\pi s) \partial_{p_z} + \sin(2\pi s) \partial_z$. The vector $v_{0}$ is taken under the first hitting map $\Sigma_{0} \to \Sigma_{s}$ to a positive multiple of $v_{s}$ for every $s>0$. The vectors $v_s$ represent infinitesimal curves inside $\Sigma_s$. Varying $s$ from $0$ to $1$, while $v_s$ makes a full turn with respect to the frame $\{\partial_{p_z},\partial_z\}$, we obtain via the linearized flow the family $\tilde g_s|_{\partial \Upsilon},s\in[0,1],$ which determines the rotation number  ${\rm Rot}(\partial \Upsilon)$. On the other hand, along the period of $\zeta_e$, while the base point along $\partial \Upsilon$ gives one positive full turn, the linearized flow  determines the rotation number $\rho_e-1$ with respect to the frame $\{\partial_{p_z},\partial_z\}$. As rotation numbers, we conclude that
$$
\frac{1}{{\rm Rot}(\partial \Upsilon)} = \frac{\rho_e-1}{1}.
$$
\end{proof}

We are ready to prove Theorem \ref{thm infy z brake}.

\subsection{Proof of Theorem \ref{thm infy z brake}} \label{subsec: return map with twist condition}
Consider the mapping $\check g:\Upsilon \to \Upsilon$ representing the first return map $\Sigma_0 \to \Sigma_0$. Recall from last section that the family $\tilde g_s:\Upsilon \to \Upsilon, s\in[0,1],$ with $\tilde g_0 = {\rm Id}$ and $\tilde g_1 = \check g$, determines rotation numbers ${\rm Rot}(\partial \Upsilon)$ and ${\rm Rot}(p)$, where $p\in \Upsilon \setminus \partial \Upsilon$ satisfies $\check g(p) = p \in {\rm Fix}(\mathcal{N})$ is the fixed point associated with the intersection of $\zeta_z$ with $\Sigma_0$.

Our standing assumption that the Hopf link $\zeta_z \cup \zeta_e$ is non-resonant implies that
\begin{equation}\label{rotdif}
    \rho_z-1 \neq (\rho_e -1)^{-1}.
\end{equation}
Hence, from Lemma \ref{lem:Rot}, we obtain ${\rm Rot}(p) \neq {\rm Rot}(\partial \Upsilon)$.

Consider the curves $\mathcal{C}_{\pm,i},D_i,i=1,2,\subset \Upsilon$ with endpoints in $\{p\}$ and $\partial \Upsilon$ as in last section, and consider $\mathcal{C}_\pm = \mathcal{C}_{\pm,1}\cup \mathcal{C}_{\pm,2}$ and $D = D_1 \cup D_2$. Condition \eqref{rotdif} gives
$$
\lim_{n \to \infty} \# (\check g^n(\mathcal{C}_\pm) \cap D ) = +\infty,
$$
Lemma \ref{lem:PerOrbitSigma}-(i) implies that there are infinitely many $z$-symmetric brake orbits. This proves Theorem \ref{thm infy z brake}-(i). The cases (ii), (iii) and (iv) are similar.



\subsection{Proof of Theorem \ref{thm: rational rotation implies infy many}}\label{subsec: return map with Rational rotation}

In this section, we consider parameters $(\beta,\mathfrak{e})\in\mathcal{D}$ for which the rotation number $\rho_e$ of the Euler orbit $\zeta_e$ is rational. Recall that $\rho_e>2$ as proved in Proposition \ref{prop: rho_e} below, and that $p=\mathcal{P}_0(\zeta_{z}(0))\in \Upsilon \setminus \partial \Upsilon, \zeta_z(0) \in \Sigma_0 \subset \{z=0\},$ is a symmetric fixed point of $\check g=\tilde g_1:\Upsilon \to \Upsilon$, see \eqref{def of tilde gs}.

Suppose that $\rho_e$ is rational.  By Lemma \ref{lem:Rot},  $\mathrm{Rot}(\partial\Upsilon)$ is rational and thus $\partial \Upsilon$ contains at least one periodic point.  A famous theorem of Franks \cite{Fr92} states that a homeomorphism of the closed or open disk preserving a finite area form and which has at least two periodic points must in fact have infinitely many interior periodic points. J. Kang \cite{kang18} extended this result with the additional hypothesis of reversibility, see \cite{LWY} for a refinement.

\begin{thm}[{Kang \cite[Corollary 1.2 and Theorem 1.3]{kang18}}]\label{thm: Kang} Let $f$ be a reversible homeomorphism of the closed or open disk that  preserves a finite area form.
\begin{itemize}
    \item[(i)] If $f$ has an interior symmetric fixed point and another periodic point, then $f$ has infinitely many symmetric periodic points in the interior of the disk.
    \item[(ii)] If $f$ has an interior symmetric fixed point and another periodic point with odd period, then $f$ has infinitely many symmetric periodic points with odd period in the interior of the disk.
\end{itemize}
\end{thm}

We know $\check g$ has a symmetric fixed point. Since $\rho_e\in \Q$, $\check g$ also has a periodic point on $\partial \Upsilon$. We conclude from Theorem \ref{thm: Kang}-(i) that $\check g$ has infinitely many symmetric periodic points in $\Upsilon \setminus \partial \Upsilon$. Proposition \ref{prop: symmetric orbits and periodic points}-(ii) implies the existence of infinitely many $z$-symmetric periodic orbits in $\mathfrak{M}$, and such periodic orbits  hit $\{z=p_r=0\}$. This proves Theorem \ref{thm: rational rotation implies infy many}-(i).

Assume now that $\rho_e = \frac{p}{q} \in \Q$,  where $p>0$ is odd and $q>0$ is even. Lemma \ref{lem:Rot} says that the rotation number of $\check g|_{\partial \Upsilon}$, with respect to the family $\tilde g_s$  is ${\rm Rot}(\partial \Upsilon)=\frac{1}{\rho(\zeta_e)-1} = \frac{q}{p-q}\in \Q$.  Since $q$ is even and $p-q$ is odd, we find a periodic point on $\partial \Upsilon$ with odd period. Theorem \ref{thm: Kang}-(ii) implies that $\check g$ has infinitely many symmetric periodic points in $\Upsilon \setminus \partial \Upsilon$ with odd period. Proposition \ref{prop: symmetric orbits and periodic points}-(iii) implies the existence of infinitely many $z$-symmetric brake orbits in $\mathfrak{M}$. This proves Theorem \ref{thm: rational rotation implies infy many}-(ii).

\begin{figure}[ht]
\centering
\includegraphics[width=0.7\textwidth]{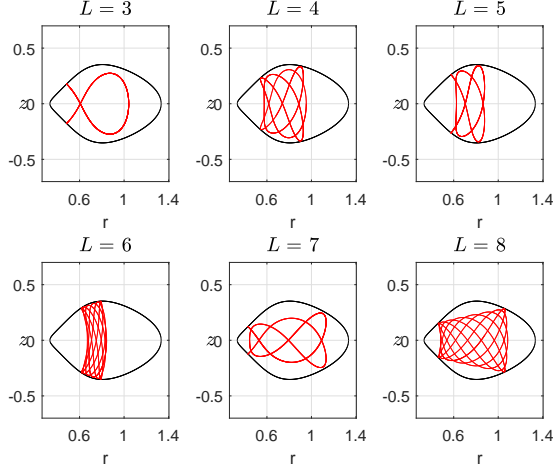}
\caption{Different types of $z$-symmetric orbits,  distinguished by the intersection number $L$ with $\Sigma_0$.}
\label{fig:z-sym}
\end{figure}

\subsection{The Euler orbit} Recall that $\zeta_e(t)=(p_{re},0,r_e,0)(t) \in \mathfrak{M}$ satisfies
$$
\ddot r_e = \frac{\varpi^2}{r_e^3} - \frac{\beta^{-1}}{r_e^2}.
$$
The preservation of energy gives
\begin{equation}\label{rde}
\frac{\dot r_e^2}{2} + \frac{\varpi^2}{2r_e^2}-\frac{\beta^{-1}}{r_e} = -1.
\end{equation}
and the linearized flow along $\zeta_e(t)=(p_{re},0,r_e,0)(t) \in \mathfrak{M}$, restricted to the plane spanned by $\partial_{p_z}$ and $\partial_z$, satisfies
$$
\dot \xi_1 = J_0\left(\begin{array}{cc}1 & 0 \\ 0 & (7 + \beta^{-1})r_e(t)^{-3} \end{array} \right) \xi_1,
$$
where $J_0=\left(\begin{array}{cc} 0 & -1\\ 1 & 0 \end{array}\right)$ and $\beta^{-1} = 1+4\alpha^{-1}$.

It will be convenient to reparametrize solutions using a new time coordinate $\theta \in \R / 2\pi \Z$. Let
$$
\tht(t):=\int_{0}^{t}\frac{\varpi}{r_e(\tau)^2}d\tau, \quad \forall t\in [0,T_e],
$$
where $T_e>0$ is the period of $\zeta_e$. We see that $\dot \theta(t)>0$ for every $t$.

\begin{lem}
$\theta(T_e) = 2 \pi.$
\end{lem}

\begin{proof}
Denote by $r_{\min},r_{\max}>0$ the minimum and the maximum values of $r_e$, and observe from \eqref{rde} that
$
r_{\min} r_{\max} = \frac{\varpi^2}{2}.
$
Then
\begin{equation}\label{intr}
\begin{aligned}
\theta(T_e) & = \int_0^{T_e} \frac{\varpi}{r_e(\tau)^2 \dot r_e(\tau)} \dot r_e(\tau)\ d\tau\\
& = 2\int_{r_{\rm min}}^{r_{\rm max}} \frac{\varpi}{r\sqrt{2(r-r_{\rm min})(r_{\rm max}-r)}} \ dr\\
& = 2\varpi \frac{-\sqrt{2}\arctan\left(\sqrt{\frac{r_{\min}(r_{\rm max}-r)}{r_{\max}(r-r_{\rm min})}} \right)}{\sqrt{r_{\rm min}r_{\rm max}}}\bigg|_{r_{\rm min}}^{r_{\rm max}} =2\pi.
\end{aligned}
\end{equation}

\end{proof}

Now we consider $r_e=r_e(\theta)$ and $\xi_1=\xi_1(\theta),\ \theta \in \R / 2\pi \Z.$ Using that
$
r_e' = \dot r_e t' = \dot r_e r_e^2\varpi^{-1},
$ and $r_e'' = \ddot r_e t'^2 + \dot r_e t''$, we end up with the second order equation satisfied by $r_e$
\begin{equation}\label{rell}
r_e'' = r_e -\frac{r_e^2}{\varpi^2 \beta} + 2\frac{r_e'^2}{r_e}.
\end{equation}

\begin{lem}Assume that $r_e(0) = r_{ \min}=\frac{\varpi^2\beta}{1+\mathfrak{e}}$ and $r_e'(0)=0$. Then
$$
r_e(\theta) = \frac{\varpi^2 \beta}{1+\mathfrak{e} \cos \theta}, \quad \forall \theta \in \R / 2 \pi \Z,
$$
where $\mathfrak{e} = \sqrt{1 + 2 h \varpi^2 \beta^2}$ is the eccentricity of the corresponding Keplerian orbit.
\end{lem}
\begin{proof}Let $u=\frac{1}{r_e} - \frac{1}{\varpi^2 \beta}\Leftrightarrow r_e = \frac{\varpi^2\beta}{\varpi^2\beta u +1}$. Then $u(0) = \frac{\mathfrak{e}}{\varpi^2 \beta}$ and $u'(0)=0$. A straightforward computation using \eqref{rell} gives $u'' = - u$. Hence $u= \frac{\mathfrak{e}\cos \theta}{\varpi^2 \beta}$ and the lemma follows.\end{proof}

Now using that $\xi_1'(\theta) = \dot \xi_1(t) t'(\theta)=\dot \xi_1(t)r_e(\theta)^2\varpi^{-1}$, $t=t(\theta)$, we obtain
\begin{equation}\label{xi11}
\xi_1' = J_0\left(\begin{array}{cc}r_e(\theta)^2 \varpi^{-1} & 0 \\ 0 & (7 + \beta^{-1})r_e(\theta)^{-1}\varpi^{-1} \end{array} \right) \xi_1,
\end{equation}

To further simplify the above linear system, we still need a time-dependent transformation of $\xi_1$. Let
\begin{equation}\label{xi22}
\xi_2(\theta) = \mathcal{R}(\theta) \xi_1(\theta), \quad \forall \theta \in \R / 2\pi \Z,
\end{equation}
where $\mathcal{R}: \R / 2\pi \Z \to {\rm Sp}(2)$ is the smooth path of $2 \times 2$ symplectic matrices given by
\begin{equation}
 \label{equ: translation matrix} \mathcal{R}(\theta):=\left(
\begin{array}{cccc}
  \frac{r_e}{\sqrt{\varpi}} &  \frac{-\sqrt{\varpi}r'_e}{r_e^2} \\
 0 & \frac{\sqrt\varpi}{ r_e} \end{array} \right)=\begin{pmatrix}
 \frac{r_e}{\sqrt{\varpi}}  &   - \frac{\sqrt{\varpi}}{r_e} \frac{\mathfrak{e} \sin \tht}{1 + \mathfrak{e} \cos \tht} \\
 0 & \frac{\sqrt\varpi}{ r_e} \end{pmatrix}= \begin{pmatrix}
 \frac{r_e}{\sqrt{\varpi}}  &   - \frac{\mathfrak{e} \sin \tht}{\sqrt{\varpi}^3\beta} \\
 0 & \frac{\sqrt\varpi}{ r_e} \end{pmatrix}.
\end{equation}
Notice that the Maslov index of $\mathcal{R}$ is $0$ since $\mathcal{R}(\theta)$ preserves $\R(0,1)^T$ for every $\theta$.  Straightforward computations using \eqref{xi11}, \eqref{xi22} and \eqref{equ: translation matrix} give
\bea \label{linear euler z new co}
\xi_2'=J_0 \left(
\begin{array}{cccc}
 1 & 0\\
 0 & 1+7\beta(1+ \mathfrak{e}\cos\theta)^{-1} \end{array} \right)\xi_2.
 \eea

From the equation above we immediately obtain the following proposition.

\begin{prop}\label{prop: rho_e} The rotation number $\rho_e$ of the Euler orbit is $>2$, for every $(\beta,\mathfrak{e})\in (0,1) \times [0,1)$.
\end{prop}

\begin{proof} Writing $\xi_2 = (a \ \ b)^T$ and using polar coordinates $a+ib=ue^{i\phi}$, $u>0$, we obtain from \eqref{linear euler z new co} that $\dot \phi(\theta) = \cos^2\phi + (1+7\beta(1+\mathfrak{e}\cos \theta)^{-1}) \sin^2 \phi = 1 + 7\beta(1+\mathfrak{e}\cos \theta)^{-1} \sin^2 \phi>1,$ except for $\phi=k \pi,k\in \Z$. Hence $\phi(2\pi) - \phi(0) >2\pi$ for any initial condition $r(0)>0,\theta(0)\in \R$. In particular, the rotation number associated with the symplectic path generated by \eqref{linear euler z new co} is $\rho_{p_z,z}>1$. Now since $(p_{re}(\theta),r_e(\theta)),\theta \in [0,2\pi],$ is a simple close curve whose tangent vector rotates precisely one positive full turn, its contribution to $\rho_e$ is $\rho_{p_r,r}=+1$. Hence $\rho_e=\rho_{p_z,z}+\rho_{p_r,r} > 2$.
\end{proof}

\subsection{Proof of Theorem \ref{thm:Dk}} The proof of Theorem \ref{thm:Dk}-(i) and (ii) is motivated by the works on the stability of elliptic relative equilibria \cite{HLS14}, \cite{Zhou17a}.

We start with the proof of (i). Let us assume that for some $(\beta_*,\mathfrak{e}_*)\in (0,1)\times [0,1)$ the rotation number $\rho_e = k>0$ is a positive integer and $\zeta_e$ is degenerate, that is there exists a non-trivial  periodic solution $\xi_2(\theta)=[y(\theta) \ x(\theta)]^T,\theta \in \R / 2\pi \Z.$ Using \eqref{linear euler z new co}, we see that
\begin{equation}\label{xfourier}
\ddot x = - (1+7\beta(1+\mathfrak{e}\cos \theta)^{-1})x, \quad \forall \theta,
\end{equation}
Consider the Fourier expansion of $x$
$$
x(\theta) = \sum_{n \geq 0} a_n \cos n \theta + b_n \sin n \theta, \quad \forall \theta.
$$
Multiplying both sides of  \eqref{xfourier} by $(1+\mathfrak{e}\cos \theta)$, and using that
\begin{equation}\label{symmcs}
\cos \theta \cos n \theta = \frac{1}{2}(\cos(n+1)\theta + \cos (n-1) \theta), \ \cos \theta \sin n \theta = \frac{1}{2} (\sin(n+1)\theta + \sin (n-1) \theta),
\end{equation}
for every $n\geq 1$, we obtain the relations for $a_n$
$$
\begin{aligned}
& \frac{\mathfrak{e}}{2}a_1 = (1+7\beta)a_0+\frac{\mathfrak{e}}{2}a_1\Rightarrow a_0=0,\\
& \frac{\mathfrak{e}}{2}\left(((n+1)^2-1)a_{n+1}+((n-1)^2-1)a_{n-1}\right)=(1+7\beta-n^2)a_n,\quad  \forall n \geq 1.
\end{aligned}
$$
Hence $a_2 = \frac{14\beta}{3\mathfrak{e}}a_1$, and $a_{n+1}$ depends linearly on $a_n$ and $a_{n-1}$ for every $n\geq 2$.

Because of the symmetry in \eqref{symmcs}, similar relations holds for $b_n$
$$
\begin{aligned}
& 0 = (1+7\beta)b_0\Rightarrow b_0=0,\\
& \frac{\mathfrak{e}}{2}\left(((n+1)^2-1)b_{n+1}+((n-1)^2-1)b_{n-1}\right)=(1+7\beta-n^2)b_n,\quad  \forall n \geq 1.
\end{aligned}
$$
Hence  $b_2 = \frac{14\beta}{3\mathfrak{e}}b_1$, and $b_{n+1}$ depends linearly on $b_n$ and $b_{n-1}$ for every $n\geq 2$. We may assume that $a_1 =1\neq 0$. Because the relations for $a_n$ and $b_n$ are the same, we see that $\hat a_n:=\mu a_n, \hat b_n := \nu a_n, n\geq 1,$ with $\hat a_0=\hat b_0=0$, determine $2\pi$-periodic solutions to \eqref{xfourier} for every $\mu,\nu \in \R.$ Hence the space of such periodic solutions is two-dimensional and the transverse linearized map along $\zeta_e$ is the identity map.

Next we show that for $\beta\neq \beta_*$ sufficiently close to $\beta_*$ and $\mathfrak{e}=\mathfrak{e}_*$ the rotation number of $\zeta_e$ is not equal to $k$, that is $\zeta_e$ is a nondegenerate elliptic orbit. To prove it, we see from \eqref{linear euler z new co} that an argument $\phi$ of $\xi_2$ satisfies
\begin{equation}\label{equ of phi}
\dot \phi = 1+\frac{7\beta}{1+\mathfrak{e}\cos \theta} \sin^2 \phi, \quad \forall \theta,
\end{equation}
which is strictly increasing in $\beta\in [0,1]$, for fixed $\theta$ and $\phi$, except for $\phi = n\pi, n\in \Z.$ In that case, we compute
$$
\dot \phi = 1,  \quad \ddot \phi =0 \quad \mbox{ and } \quad \dddot \phi =\frac{14 \beta}{1+\mathfrak{e} \cos \theta}>0, \quad \mbox{ if } \phi=n\pi, n\in \Z.
$$
We see that $\dddot \phi$ is also strictly increasing in $\beta$ if $\phi =n \pi, n\in \Z$. We summarize this discussion in the following lemma.

\begin{lem}\label{lembeta} The argument variation $\phi(2\pi) - \phi(0)$ is strictly increasing in $\beta$ for any fixed $\phi(0)\in \R$. \end{lem}

Continuing with the proof of (i), because the first return map along $\zeta_e$ is the identity map for $(\beta_*,\mathfrak{e}_*)$ we conclude that the rotation number of $\zeta_e$ is strictly larger than $k$ for $\beta-\beta_*>0$ sufficiently close to $0$ and $\mathfrak{e}=\mathfrak{e}_*$, and strictly smaller than $k$ for $\beta-\beta_*<0$ sufficiently close to $0$ and $\mathfrak{e}=\mathfrak{e}_*$.

Finally, for $\beta=0$ and any $\mathfrak{e} \in [0,1)$, \eqref{linear euler z new co}  admits a two-dimensional space of periodic solutions, and $\rho_e \to 2^+$ as $\beta \to 0^+$. Since $\rho_e$ continuously depends on $(\beta,\mathfrak{e})$, the argument above shows that there exists no pair $(\beta,\mathfrak{e})\in (0,1)\times [0,1)$ so that $\zeta_e$ has a positive rotation number and is positive hyperbolic.  The proof of Theorem \ref{thm:Dk}-(i) is complete.

Now we prove (ii). Let us assume that for some $(\beta_*,\mathfrak{e}_*)\in (0,1)\times [0,1)$ the rotation number $\rho_e = k+\frac{1}{2}>0$ for some positive integer $k>1$. In this case, $\zeta_e$ is nondegenerate, and there exists a $4\pi$-periodic solution $\xi_2(\theta)=[y(\theta) \ x(\theta)]^T,\theta \in \R / 2\pi \Z,$ satisfying
$$
\xi_2(\theta+2\pi) = -\xi_2(\theta), \quad \forall \theta.
$$
We conclude that
\begin{equation}\label{xfourier2}
\begin{aligned}
& \ddot x  = - (1+7\beta(1+\mathfrak{e}\cos \theta)^{-1})x, \quad \forall \theta,\\
& x(\theta+2\pi)  =-x(\theta), \quad \forall \theta.
\end{aligned}
\end{equation}

If $\mathfrak{e}=0$, then whenever $\beta=\frac{(n+1/2)^2-1}{7},\ n\in\mathbb{N}_+$, the space of such solution is two dimensional and generated by
\begin{equation}\label{e=0}
x_1(\theta)=\cos(\sqrt{1+7\beta}\theta),\quad x_2(\theta)=\sin(\sqrt{1+7\beta}\theta).
\end{equation}

Assume $\mathfrak{e}>0$. The  Fourier expansion of $x$ has the form
$$
x(\theta) = \sum_{n \geq 0} a_n \cos (n+1/2) \theta + b_n \sin \left(n+1/2\right) \theta, \quad \forall \theta.
$$
Multiplying both sides of  \eqref{xfourier} by $(1+\mathfrak{e}\cos \theta)$, and using that
\begin{equation}\label{symmcs2}
\begin{aligned}
&\cos \theta \cos \left(n+1/2\right) \theta = \frac{1}{2}\left(\cos \left(n+3/2\right) \theta + \cos \left(n- 1/2\right) \theta\right),
\\ &\cos \theta \sin \left(n+ 1/2 \right) \theta = \frac{1}{2} \left(\sin\left(n+1/2\right)\theta + \sin \left(n-1/2\right) \theta \right),
\end{aligned}
\end{equation}
for every $\theta\in \R,n\geq 1$, we obtain the relations for $a_n$

$$
\begin{aligned}
& a_1 = \frac{6+56\beta+3\mathfrak{e}}{5\mathfrak{e}}a_0,\\
& \frac{\mathfrak{e}}{2}\left(((n+3/2)^2-1)a_{n+1}+(\left(n-1/2\right)^2-1)a_{n-1}\right)=\left(1+7\beta-(n+1/2)^2\right)a_n,
\end{aligned}
$$
for every $n\geq 1$. Hence $a_{n+1}$ depends linearly on $a_n$ and $a_{n-1}$ for every $n\geq 2$.

Using \eqref{symmcs2} again, we obtain the following relations for $b_n$
$$
\begin{aligned}
& b_1 = \frac{6+56\beta-3\mathfrak{e}}{5\mathfrak{e}}b_0,\\
& \frac{\mathfrak{e}}{2}\left(((n+3/2)^2-1)b_{n+1}+(\left(n-1/2\right)^2-1)b_{n-1}\right)=\left(1+7\beta-(n+1/2)^2\right)b_n,
\end{aligned}
$$
for every $n\geq 1$. Hence $b_{n+1}$ depends linearly on $b_n$ and $b_{n-1}$ for every $n\geq 2$.

\begin{lem}One of the two possibilities hold:
\begin{itemize}
    \item[(i)] $(a_n)$ vanishes identically and $(b_n)$ does not vanish identically.
    \item[(ii)] $(b_n)$ vanishes identically and $(a_n)$ does not vanish identically.    \end{itemize}
\end{lem}

\begin{proof}
    Notice that except for the expressions for $a_1$ and $b_1$, the induction formulas for $a_n$ and $b_n$ coincide for every $n\geq 2$. Indeed, we see that both $y_n=(a_{n+1}, a_n)^T$ and $y_n=(b_{n+1}, b_n)^T$ satisfy
    $$
    y_n = T_n(y_{n-1}), \quad \forall n,
    $$
    where  $T_n:\R^2 \to \R^2$ is the invertible linear mapping given by
\begin{equation}\label{hyper system}
T_n(y):=
\begin{pmatrix}\frac{1+7\beta-(n+1/2)^2}{\frac{\mathfrak{e}}{2}((n+3/2)^2-1)} & -\frac{(n-1/2)^2-1}{(n+3/2)^2-1} \\ 1 & 0 \end{pmatrix}y, \quad \forall y\in \R^2, \quad \forall n,
\end{equation}
 with respective initial conditions $y_0=(a_1,a_0)\in \R(\frac{6+56\beta+3\mathfrak{e}}{5\mathfrak{e}},1)$ and $y_0=(b_1,b_0)\in \R(\frac{6+56\beta-3\mathfrak{e}}{5\mathfrak{e}},1)$.

We assume by contradiction that both $(a_n)$ and $(b_n)$ do not vanish identically. Then the initial conditions $(a_1,a_0)$ and $(b_1,b_0)$ are linearly independent, and since   $a_n,b_n \to 0$ as $n \to \infty$,  we conclude that  $y_n:=T_n(y_{n-1}), n\geq 1,$ converges to $0$ for every initial condition $y_0 \in \R^2$.

Observe that   $T_n$ converges to the hyperbolic linear mapping $T_\infty :\R^2 \to \R^2$ given by
\begin{equation}\label{hyper system1}
T_\infty(y):=
\begin{pmatrix}-2/\mathfrak{e} & -1 \\ 1 & 0 \end{pmatrix}y, \quad \forall y\in \R^2.
\end{equation}
The mapping $T_\infty$ admits an unstable direction $v_u:= \R(-\frac{1+\sqrt{1-\mathfrak{e}^2}}{\mathfrak{e}},1)$ associated with the eigenvalue $\frac{-1-\sqrt{1-\mathfrak{e}^2}}{\mathfrak{e}}\in(-\infty,-1)$ and a stable direction $\R(-\frac{1-\sqrt{1-\mathfrak{e}^2}}{\mathfrak{e}},1)$ associated with the eigenvalue $\frac{\sqrt{1-\mathfrak{e}^2}-1}{\mathfrak{e}}\in(-1,0)$. Hence there exists a closed cone $C_0 \subset \R^2$ centered at $v_u$ and $\nu>1$ so that for $n$ sufficiently large
$$
T_n(C_0) \subset C_0,\ \ \forall n \gg 0 \quad \mbox{ and } \quad |T_n(y)| \geq \nu |y|,\ \  \forall y\in C_0.
$$
 Hence we find  initial conditions $y_0\in \R^2$ so that the sequence $y_n:=T_n(y_{n-1}), n\geq 1,$  is divergent,  a contradiction. We conclude that one of the sequences  $(a_n)$ or $(b_n)$ vanishes and the other never vanishes. \end{proof}

\begin{lem}\label{anbn}
    The linearized first return map along $\zeta_e$ is a shear with eigenvalue $-1$. In particular, $d\phi_{T_e}(\zeta_e(0)):\xi|_{\zeta_e(0)} \to \xi|_{\zeta_e(0)}$ is not $-\text{Id}$.
\end{lem}

\begin{proof} Assume there exists a non-trivial $4\pi$-periodic solution $\tilde x\neq x$ solving \eqref{xfourier2}. Recall the sequences $(a_n),(b_n)$ from the Fourier expansion of $x$ and consider the corresponding sequences $(\tilde a_n),(\tilde b_n)$ for the Fourier expansion of $\tilde x$.
By Lemma \ref{anbn} either $(a_n)$ or $(b_n)$ vanishes identically, and the same holds for $(\tilde a_n),(\tilde b_n)$. If both $(a_n)$ and $(\tilde b_n)$ do not vanish identically, we obtain from the linearity of \eqref{xfourier2} a solution $\hat x$ whose Fourier coefficients $(\hat a_n),(\hat b_n)$ do not vanish identically, a contradiction to Lemma \ref{anbn}. The same holds with $(\tilde a_n)$ and $(b_n)$. We conclude that either both $(a_n)$ and $(\tilde a_n)$ vanish identically or $(b_n)$ and $(\tilde b_n)$ vanish identically. The induction formulas for the Fourier coefficients imply that there exists $\nu \neq 0$ so that $\tilde a_n = \nu a_n$ and $\tilde b_n=\nu b_n$ for every $n$. In particular, the space of solutions to \eqref{xfourier2} is one-dimensional and  $d\phi_{T_e}(\zeta_e(0))$ is a shear with eigenvalue $-1$. \end{proof}

Continuing with the proof of (ii), since the first return map $d\phi_{T_e}(\zeta_e(0))$ is a shear, the rotation interval $\{\phi(2\pi)-\phi(0): \phi(0)\in \R\}$ determined by the solutions of \eqref{equ of phi} is a non-trivial interval $I_*\subset \R$ of length $0<\ell(I_*) < \pi$ so that $\pi +2k\pi$ is a boundary point. Fixing $\mathfrak{e}=\mathfrak{e}_*$ and varying $\beta$ we denote by $I_\beta$ the corresponding rotation interval for the solutions of \eqref{equ of phi}. Suppose $I_*$ lies to the left of $\pi+2k\pi$. Then Lemma \ref{lembeta} implies that $I_\beta$  contains $\pi + 2k \pi$ as an interior point for every $\beta-\beta_*>0$ sufficiently small. This implies that $\zeta_e$ is negative hyperbolic for such values of $\beta$. Analogously, if $I_*$ lies to the right of $\pi + 2k \pi$, then $\pi + 2k \pi$ is an interior point of $I_\beta$ if $\beta-\beta_*<0$ is sufficiently small. In this case, $\zeta_e$ is also negative hyperbolic for those values of $\beta$. In each case, we find a non-trivial open interval of $\beta$ either to the left  or to the right of $\beta_*$ so that the Euler orbit is negative hyperbolic.

The above argument shows that if there exists $(\beta_*,\mathfrak{e}_*)\in (0,1) \times [0,1)$ so that the Euler orbit has rotation number $\rho_e = k + 1/2,$ then there exists an open interval $\mathcal{I}$ so that for every $\beta \in \mathcal{I}$ and $\mathfrak{e} = \mathfrak{e}_*$, the Euler orbit associated with $(\beta,\mathfrak{e})$ is negative hyperbolic.

It remains to show that for any integer $k\geq 2$ there exists  $(\beta,\mathfrak{e})\in (0,1) \times [0,1)$ so that the Euler orbit has rotation number $k+1/2$.

\begin{lem}For every integer $k\geq 2$, there exists $(\beta,\mathfrak{e})\in (0,1) \times [0,1)$ so that the Euler orbit has rotation number $k+1/2$.
\end{lem}

\begin{proof} Consider $\phi=\phi(\theta),\theta \in [0,\pi),$  a solution to
$$
\dot \phi = 1 + \frac{7}{1+\cos \theta}\sin^2 \phi \geq 1,
$$
starting from any $\phi(0)\in \R.$
We claim that
\begin{equation}\label{limphi}
\lim_{\theta \to \pi} \phi(\theta) = +\infty.
\end{equation}
To prove this claim, we assume by contradiction that $\lim_{\theta \to \pi} \phi(\theta) = A \in \R.$ This limit should exist since $\dot \phi>1$ and thus $\phi$ is monotone increasing. We show that $A = k\pi$ for some $k\in \Z$. Indeed, if $\sin^2 A \neq 0$, then $\dot \phi> 1+ \frac{7}{1+\cos \theta}\frac{\sin^2 A}{2}$ for $\theta$ sufficiently close to $\pi$. Since $ \int_{\pi - \epsilon}^\pi \frac{c}{1+\cos \theta} \ d\theta = +\infty$ for any $c>0$ and any  $\epsilon>0$ small, we obtain $A-\phi(0)=\int_0^\pi \dot \phi \ d\theta = +\infty$, a contradiction.

From the periodicity of $\dot \phi$ with respect to $\phi$ we may assume that $A = 0$. Since $\lim_{\theta \to \pi} \phi(\theta) = 0,$ we conclude that
$$
\dot \phi > 1 + \frac{\phi^2}{2(\theta - \pi)^2},
$$
for every $\theta$ sufficiently close to $\pi$. Hence any solution $\phi_0=\phi_0(\theta)$ to the equation $\dot \phi = 1 + \frac{\phi^2}{2(\theta - \pi)^2}$, with $\phi_0(\pi-\epsilon) = \phi(\pi - \epsilon)$, $\epsilon>0$ small, satisfies
$$
\phi_0(\theta) = (\theta - \pi)(1 -  \tan((c_1 - \log (\pi - \theta))/2)) <\phi(\theta), \quad \forall \pi-\epsilon<\theta<\pi.
$$
However, a direct computation shows that $\phi_0$ blows up for some $\theta_*\in(\pi-\epsilon,\pi).$ This contradiction proves our claim and the limit in \eqref{limphi} holds.

By continuous dependence of solutions with respect to parameters, we find $\mathfrak{e}_*\in [0,1)$ close to $1$ so that the solution to \eqref{equ of phi} with $\beta=1$ and $\mathfrak{e}=\mathfrak{e}_*$ is such that $\phi(2\pi) - \phi(0) > 2\pi (k+1/2)$ for any initial condition $\phi(0)$. This implies, in particular, that $
\rho_e > k + 1/2$. Since $\rho_e = 2$ for $\beta=0$ we find $\beta_*\in (0,1)$ so that for $\beta= \beta_*$ and $\mathfrak{e}=\mathfrak{e}_*$, the rotation number of the Euler orbit is $k+1/2.$ \end{proof}

To complete the proof of Theorem \ref{thm:Dk}-(ii) we show that, for any fixed $\mathfrak{e}\in(0,1/3)$, there exist non-trivial intervals $I_{5/2},I_{7/2}\subset (0,1),$ depending on $\mathfrak{e}$, so that $\rho_e=5/2$ if $\beta\in I_{5/2}$,  and $\rho_e=7/2$ if $\beta\in I_{7/2}$. Moreover, such intervals $I_{5/2}$ and $I_{7/2}$ converge to $5/28$ and $3/4$ as $\mathfrak{e} \to 0^+$, respectively.

For $\mathfrak{e}=0$, we see from \eqref{e=0}  that $\rho_e=\rho_{p_r,r}+\rho_{p_z,z}=1+\sqrt{1+7\beta}$ for every $\beta.$
Hence $\rho_e$ is strictly increasing in $\beta\in[0,1]$ from $2 < 5/2$ to $1+2\sqrt{2} > 7/2$. For $\beta=5/28$ and $3/4$, respectively, we have $\rho_e=5/2$ and $\rho_e=7/2$. If $\beta=0$, then $\rho_e=2$ is independent of $\mathfrak{e}$. If $\beta=1$ and $\mathfrak{e}\in (0,1/3)$, then \eqref{equ of phi} gives $\dot \phi \geq 1+\frac{21}{4} \sin^2 \phi$. Comparing with the linear system, we obtain $\rho_e>1 +\sqrt{1+7/(1+1/3)}=7/2$. By the continuity of $\rho_e$ with respect to parameters and the monotonicity of $\rho_e$ with respect to $\beta$, we obtain from the argument above that for each $\mathfrak{e} \in (0,1/3)$ there exist unique non-trivial intervals $I_{5/2}$ and $I_{7/2}$ with the desired properties. For $(\beta,\mathfrak{e})$ on such intervals, the Euler orbit is negative hyperbolic. Since hyperbolicity is preserved under small perturbation of parameters, we conclude that $\mathcal{D}_2 \cap \mathcal{D}$ and $\mathcal{D}_3 \cap \mathcal{D}$ are non-trivial open subsets. The proof of Theorem \ref{thm:Dk}-(ii) is finished.

Now we prove Theorem \ref{thm:Dk}-(iii).

We  show that $\mathcal{D}_{{\rm odd/even}}\subset (0,1)\times [0,1)$ is dense. This implies, in particular, that $\mathcal{D}_{\mathcal{R}}\subset (0,1)\times [0,1)$ is dense. Recall that $\mathcal{D}_{k+1/2}$ is the collection of all $(\beta,\mathfrak{e})\in (0,1)\times[0,1)$ so that $\zeta_e$ is negative hyperbolic with rotation number $k+1/2$.
For any $(\beta, \mathfrak{e})\in \mathcal{D}\setminus (\cup_{k>1}\overline{\mathcal {D}_{k+1/2}})$,  $\zeta_e$ is either degenerate or nondegenerate elliptic. So the eigenvalues of the linearized first return map are $\lambda_\pm = e^{\pm\eta \sqrt{-1}}\in \mathbb{U}$. If $\eta \in \Q$ then some $k_0$-iterate of $\zeta_e$ is degenerate with eigenvalue $1$.  In that case, we know from Theorem \ref{thm:Dk}-(i) that the linearized first return map is necessarily the identity map, and thus  we see from \eqref{equ of phi} that the rotation number of $\zeta_e^{k_0}$ is strictly increasing in $\beta$. This implies, in particular, that $\rho_e$ is also strictly increasing in $\beta$. To deal with the case $\eta \in \R \setminus \Q$, we first observe that for  $\tilde \beta>\beta$ and fixed $\mathfrak{e}$, there exists $\delta>0$ so that the corresponding solutions $\tilde \phi$ and $\phi$ of \eqref{equ of phi} starting from the same initial condition satisfy $\tilde \phi(2k\pi)> \phi(2k\pi) + \delta$ for every $k\in \N$.
If $\eta \in \R \setminus \Q$ for certain $(\beta,\mathfrak{e})$, then for a certain sequence $k_n \to +\infty$ the linearized first return map along $\zeta_e^{k_n}$ converges to the identity map. This implies that the length of the rotation interval associated with $\zeta_e^{k_n}$ converges to $0$. Hence, taking $\tilde \beta > \beta$ with fixed $\mathfrak{e}$, we obtain from the uniform shift of the rotation interval that for $n$ sufficiently large the corresponding rotation intervals are separated by a positive distance, and thus the rotation number of $\zeta_e$ for $\tilde \beta$ is strictly larger than that for $\beta$. The continuity of $\rho_e$ with respect to parameters implies that there exists $\tilde \beta$ arbitrarily close to $\beta$ so that $\rho_e\in \Q$.
We have proved that the subset $\mathcal{D}_{{\rm odd/even}}\subset (0,1)\times [0,1)$ of parameters $(\beta,\mathfrak{e})$ for which $\rho_e=p/q\in \Q$ with $p>0$ odd and $q>0$ even is  dense.

The proof of Theorem \ref{thm:Dk} is now complete.

\begin{rem}\label{rem: alternative proof}
We briefly discuss an operator theory approach to Theorem \ref{thm:Dk}. Fix $\omega \in \mathbb{U} \subset \C$, let $f(\theta):= (1+\mathfrak{e}\cos \theta)^{1/2}$, and let
$$
A_\omega \cdot x (\theta) :=f(\theta)[-d^2/d\theta^2-1](x(\theta)f(\theta))-7\beta x(\theta),\quad \forall \theta,
$$
be the self-adjoint operator acting on curves $x\in W^{2,2}([0,2\pi],\mathbb{C})$ satisfying $\omega (\dot x,x)(0)=(\dot x, x)(2\pi)$. Notice that $A_\omega$ is bounded from below.

 Let $v(A_\omega) :=\dim_{\mathbb{C}}\ker_{\mathbb{C}}(A_\omega)\in\{0,1,2\}$, and let $\mathcal{M}(A_\omega)$ be the Morse index of $A_\omega$, i.e., the total multiplicity of the negative spectrum. From \eqref{more equa mas omega}, we know that
\begin{equation}\label{index equ}
\mathcal{M}(A_\omega)=i_{\omega}(\gamma) \quad \mbox{ and } \quad v(A_\omega)=\nu_{\omega}(\gamma),
\end{equation}
where $\gamma:[0,2\pi]\to {\rm Sp}(2)$ is the fundamental solution of \eqref{linear euler z new co} starting from the identity,  and $(i_{\omega},\nu_{\omega})$ is given by Definition \ref{def:Maslov-type index}.

If $v(A_\omega)=\nu\in \{1,2\}$ for some $(\beta,\mathfrak{e})=(\beta_*,\mathfrak{e}_*)$, then the solutions of
\begin{equation}\label{2nd ode with pm 1 bdd}
\ddot x=-x-\frac{7\beta_*}{1+\mathfrak{e}_*\cos\theta}x, \quad
\omega(\dot x,x)(0)=(\dot x,x)(2\pi),
\end{equation}
form a $\nu$-dimensional space.
In that case, we say  that $\gamma$ is $\omega$-degenerate for $(\beta_*,\mathfrak{e}_*)$. Actually, if $\omega=\pm 1$, then the multiplicity $\nu$ is related to the `coexistence' problem for the Ince's equation. More precisely,  $y(u):=x(2u)$, satisfies a special case of  Ince's equation
\begin{equation}\label{ince equation}
(1+\mathfrak{e}_*\cos2u)\ddot y(u)+(4(1+7\beta_*)+4\mathfrak{e}_*\cos2u)y(u)=0.
\end{equation}
Applying Theorems 7.1 and 7.3 from \cite{MW66},
we conclude that if $\gamma$ is $1$-degenerate for $(\beta_*,\mathfrak{e}_*)$, then $\nu=2$, and if $\gamma$ is $-1$-degenerate for $(\beta_*,\mathfrak{e}_*)$, then $\nu=1$. Moreover, fixing $\mathfrak{e}=\mathfrak{e}_*$ and taking $\beta$ slightly larger than $\beta_*$, the first identity in \eqref{index equ} implies that   $i_\omega(\gamma)$ and $\mathcal{M}(A)$ change by the same amount.
By \eqref{def of mean index}, the  mean index $\hat i(\gamma)$ is the average of $i_{\omega}(\gamma),\ \omega \in \mathbb{U}$. Since $i_\omega(\gamma)$ is non-decreasing in $\beta$, we conclude from \eqref{def: mean index and rotation number}  that $\hat i(\gamma)$ and $\rho_e=1+\hat i(\gamma)/2$ are non-decreasing in $\beta$ as well. If the transverse linearized first return map along $\zeta_e$ admits the eigenvalues $\omega, \bar \omega \in \mathbb{U} \setminus \{\pm1\}$ for $(\beta_*,\mathfrak{e}_*)$, then  \eqref{mean index 2} implies that $\hat i(\gamma)$ is strictly increases in $\beta$.

The spectrum of $A_\omega$ is a discrete and unbounded subset of $\R$. Moreover, it is strictly decreasing in $\beta\in \R$, and there exists a sequence $\beta_n \to +\infty$ such that $\gamma$ is $\omega$-degenerate for $(\mathfrak{e}_*,\beta_n)$. The continuity of $A_\omega$ in $\beta$ and $\mathfrak{e}$ implies the existence of a sequence of  disjoint curves
$$\Gamma_n^1:=\{(\mathfrak{e},\beta_n(\mathfrak{e})),\ \mathfrak{e}\in[0,1)\},\quad
\Gamma_{n,\pm}^{-1}:=\{(\mathfrak{e},\beta_{n,\pm}(\mathfrak{e})),\ \mathfrak{e}\in(0,1)\},$$
so that $\gamma$ is $1$-degenerate with rotation number $n+1$ for $(\mathfrak{e},\beta) \in \Gamma^1_n$, and $-1$-degenerate with rotation number $n+3/2$ for $(\mathfrak{e},\beta) \in \Gamma^{-1}_{n,\pm}$. The curve $\Gamma^1_n$ intersects the $\beta$-axis at $((n^2-1)/7,0)$, and both curves $\Gamma^{-1}_{n,\pm}$ approach the $\beta$-axis at $(((n+1/2)^2-1)/7,0)$, see Figure \ref{picture of Gamma}. The curves $\Gamma^1_1, \Gamma^{-1}_{1,-}$ approach the point $(0,1)$ as $\mathfrak{e} \to 1$, and for every $n > 1$, the curves $\Gamma^{-1}_{1,+},\Gamma^1_n,\Gamma^{-1}_{n,\pm}$ approach the point $(1/28,1)$ as $\mathfrak{e} \to 1$, see \cite{HOT}. For each $n$, $\Gamma^{-1}_{n,-}$ and $\Gamma^{-1}_{n,+}$ bound an open set of parameters for which the Euler orbit is negative hyperbolic. It is unknown whether the functions $\beta_n(\mathfrak{e}),\beta_{n,\pm}$ are monotone in $\mathfrak{e}.$

\begin{figure}[ht]
\centering
\includegraphics[width=0.7\textwidth]{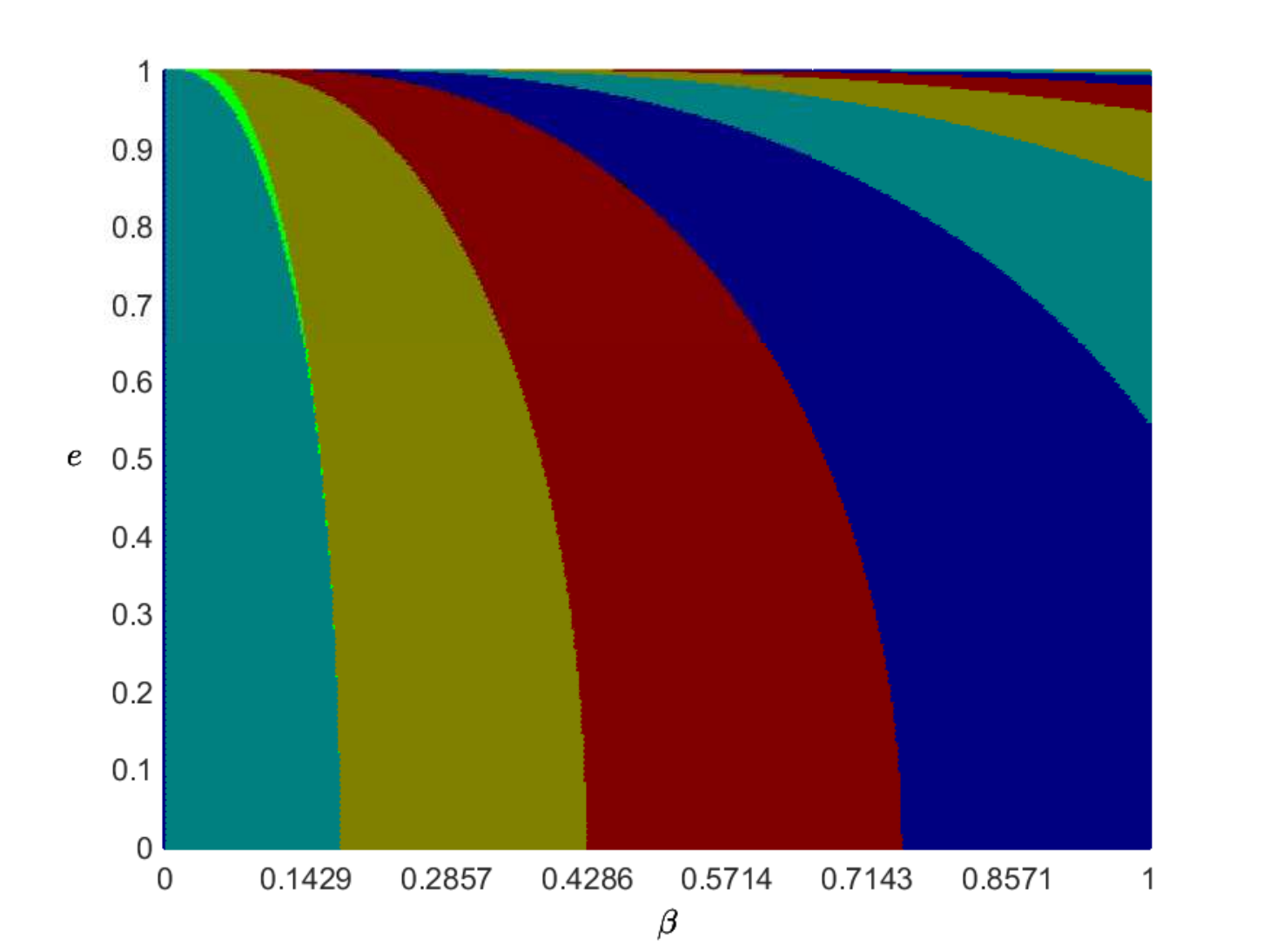}
\caption{The curves $\Gamma_n^1$ and $\Gamma_{n,\pm}^{-1}$ in the square $[0,1] \times [0,1]$. The curve $\Gamma^1_1$ is contained in the $\mathfrak{e}$-axis, and the curves $\Gamma_{1,\pm}^{-1}$ bound an open set (green region) where the Euler orbit is negative hyperbolic with rotation number $5/2$.}
\label{picture of Gamma}
\end{figure}
\end{rem}

\section{Convexity of Energy Surfaces}\label{sec: convexity}

In this section, we study  the convexity of the energy surface $\mathfrak{M}$. Recall that the parameters $\varpi,\alpha>0$ for which the energy surface $\mathfrak{M} = H^{-1}(-1)$ is diffeomorphic to the three-sphere satisfy
\begin{equation}\label{intervalvarpi}
\frac{1}{2} < \varpi^2 < \frac{(1 + 4\alpha^{-1})^2}{2}.
\end{equation}

We prove that for each fixed $\alpha>0$ there exists $\frac{1}{2}< \hat \varpi^2 <  \frac{(1 + 4\alpha^{-1})^2}{2}$ so that if  $\hat \varpi^2 < \varpi^2 < \frac{(1 + 4\alpha^{-1})^2}{2}$, then $\mathfrak{M}$ is strictly convex and if $\frac{1}{2} < \varpi^2 < \hat \varpi^2$, then $\mathfrak{M}$ is not convex.

In \cite{Sa04} we find the following criterion for an energy surface of a mechanical system to be strictly convex.

\begin{thm}[\cite{Sa04}] \label{thm: convexity}Let $\mathfrak{M} =H^{-1}(-1) \subset \R^4$ be a regular sphere-like energy surface of a mechanical Hamiltonian $H = \frac{p_r^2+p_z^2}{2} + V(r,z)$. Denote by $\mathcal{H}=\{V\leq -1\} \subset \R^2$ the disk-like Hill region given by the projection of $\mathfrak{M}$ to the $(r,z)$-plane. Then $\mathfrak{M}$ is strictly convex if and only if
\begin{equation}\label{criterion formula}
\begin{aligned}
\Delta:= & -2(1+V)(\partial_{rr}V\partial_{zz}V - (\partial_{rz}V)^2) \\ & + \partial_{rr}V(\partial_zV)^2 + \partial_{zz}V (\partial_rV)^2 - 2\partial_{rz}V\partial_r V \partial_z V
\end{aligned}
\end{equation}
is positive on $\mathcal{H}$. Moreover, if $\Delta$ is somewhere negative in $\mathcal{H}$, then $\mathfrak{M}$ is not convex, that is $\mathfrak{M}$ bounds a subset of $\R^4$ that is not convex.
\end{thm}

Under the  linear change of coordinates $r \mapsto r/\alpha $ and $z\mapsto z/(\alpha \sqrt{1+2\alpha} )$,
which does not alter the convexity properties of $\mathfrak{M}$, we may assume that

\begin{equation}\label{The potential after scalling}
V=\frac{\varpi^2\alpha^2}{2 r^2}-\frac{\alpha}{r}-\frac{4}{\sqrt{r^2+z^2}}.
\end{equation}

A direct computation using  \eqref{The potential after scalling} and \eqref{criterion formula} gives

\begin{equation}\label{quadratic form}
\Delta=A(r,z)\varpi^4+B(r,z)\varpi^2+C(r,z),
\end{equation}
where
\begin{equation}
\begin{aligned}
A(r,z) & =\frac{8(2z^2-r^2)\alpha^4}{(r^2+z^2)^{5/2}r^6},\\
B(r,z) & =\frac{24(r^2-2z^2)}{(r^2+z^2)^{5/2}r^2}(\frac{\alpha^3}{r^3}-\frac{\alpha^2}{r^2})+\frac{48(2r^2-3z^2)\alpha^2}{(r^2+z^2)^3r^4},\\
C(r,z)& =\frac{4(r^2-2z^2)}{(r^2+z^2)^{5/2}r^2}(-\frac{3\alpha^2}{r^2}+\frac{4\alpha}{r})-\frac{192}{(r^2+z^2)^{7/2}}-\frac{96(r^2-z^2)\alpha-64r^3}{(r^2+z^2)^3r^3}.
\end{aligned}
\end{equation}
It will be convenient to use coordinates $(r,k),$ instead of $(r,z)$, where $k\geq 1$ satisfies
$$
z^2=(k^2-1)r^2.
$$
In this case, $\Delta$ is replaced with
\begin{equation}\label{quadratic form 1}
\widetilde \Delta:=\frac{k^7r^9}{4}\Delta=a(r,k)\varpi^4+b(r,k)\varpi^2+c(r,k),
\end{equation}
where
\begin{equation}\label{coefficients}
\begin{aligned}
a(r,k)
& =2\alpha^4k^2(2k^2-3),\\
 b (r,k) & =-6\alpha^2kr((\alpha k-rk+3)(2k^2-3)-1),\\
c(r,k) & =r^2(\alpha k(2k^2-3)-4)(12+3\alpha k-4kr).
\end{aligned}
\end{equation}
For simplicity we denote $\widetilde \Delta$ again by $\Delta = \Delta(r,k)$. We also keep denoting by $\mathcal{H}$ the Hill region in coordinates $(r,k)$. The parameter $k\geq 1$ will be referred to as the slope.

In what follows the parameter $\alpha>0$ is fixed, while $\varpi$ may  vary in the interval given in \eqref{intervalvarpi}.

If $\varpi^2$ is sufficiently close to $(1+4\alpha^{-1})^2/2$, then $\mathfrak{M}$ is sufficiently close to a nondegenerate minimum of $H$ and thus strictly convex. This means that $\Delta|_{\mathcal{H}}>0$. As $\varpi^2$ decreases from $(1+4\alpha^{-1})^2/2$  to $1/2$, we show that there exists a special value for $\varpi$, denoted $\hat \varpi = \hat \varpi(\alpha)$, so that strict convexity ($\Delta|_{\mathcal{H}}>0$) holds for $\varpi > \hat \varpi$ and non-convexity  ($\Delta<0$ somewhere in $\mathcal{H}$) holds for $\varpi < \hat \varpi.$ Moreover, the first point where $\Delta$ vanishes on $\mathcal{H}$  occurs in $\mathcal{H} \setminus \partial \mathcal{H}.$

Let us explain the strategy of finding $\hat \varpi$. For any fixed $(r,k)$, $k\geq 1, k^2\neq 3/2$, $\Delta$ is a quadratic function of $\varpi^2$. It turns out that its discriminant $b^2-4ac$ is non-negative and we can solve $\Delta=0$ for $\varpi^2$ to obtain functions
$$
w_\pm  := \frac{-b \pm \sqrt{b^2-4ac}}{2a}.
$$
For each $\varpi$, we consider the level curves
$$
\ell_\pm:=\{w_\pm(r,k) = \varpi^2\}.
$$
If $\mathcal{H}$ does not intersect $\ell_\pm$, then $\Delta>0$ on $\mathcal{H}$ and thus $\mathfrak{M}$ is strictly convex. It is worth mentioning that  $\Delta$ is always positive on $ \partial \mathcal{H} \cap \{z=0\}$.

As $\varpi$ decreases,  strict convexity is lost for the first time if one of the following situations occurs:
\begin{itemize}
    \item[a)]  $\ell_\pm$ intersects $\partial \mathcal{H}$. This case will be ruled out by Lemma \ref{lem: w^2 for the first point on the boundary} below. Indeed, we show that if $\Delta=0$ somewhere in $\partial \mathcal{H}$, then there exist nearby points in the interior of $\mathcal{H}$ where $\Delta<0$. In particular, convexity of $\mathfrak{M}$ was already lost.

    \item[b)]  $\ell_\pm$ intersects  $\mathcal{H}\setminus \partial \mathcal{H}$. In this case, we show that $\ell_+$ or $\ell_-$ intersects only at the interior of $\mathcal{H}$ and localizing such intersection points is reduced to the study of $\Delta =0$ (local maximum) and $\nabla w_+=0$ (or $\nabla w_-=0$) in the interior of $\mathcal{H}$.
    \end{itemize}


In the first step, we identify a slope for which $\Delta|_{\partial \mathcal{H}}$ vanishes for the first time as $\varpi$ decreases. Moreover, we show that convexity is first lost in the interior of $\mathcal{H}$  as $\varpi$ decreases.

\begin{lem}\label{lem: w^2 for the first point on the boundary}
Fix $\alpha>0$. There exist $\frac{\sqrt{11+\sqrt{17}}}{2\sqrt{2}}< k_0 < \sqrt{5/2}$ and  $\frac{1}{\sqrt{2}} < \varpi_0 < \frac{1+4\alpha^{-1}}{\sqrt{2}}$  such that

\begin{itemize}
\item[(i)] if $\varpi>\varpi_0$, then $\Delta>0$ everywhere on $\partial\mathcal{H}$.

\item[(ii)] if $\varpi<\varpi_0$, then $\Delta<0$ at some point $(r_0,k_0)\in \partial\mathcal{H}$.

\item[(iii)] if $\varpi=\varpi_0$, then $\min \Delta|_{\partial \mathcal{H}}=0$ and $\min \Delta|_{\mathcal{H} \setminus \partial \mathcal{H}}<0$.
\end{itemize}
\end{lem}
The proof of Lemma \ref{lem: w^2 for the first point on the boundary} is left to Appendix \ref{subsec: proof of lemma for convexity}.

Recall that $\hat \varpi=\hat \varpi(\alpha)$ is defined as
\begin{equation}\label{hatvarpi}
\hat \varpi:= \sup \{\varpi\in(1/\sqrt{2},(1+4\alpha^{-1})/\sqrt{2}): \min \Delta|_{\mathcal{H}} \leq 0\}.
\end{equation}

We know from Lemma \ref{lem: w^2 for the first point on the boundary}  that
$\varpi_0 < \hat \varpi < (1+4\alpha^{-1})/\sqrt{2},$ that is strict convexity is lost in the interior of $\mathcal{H}$ as $\varpi$ decreases. Hence we are led to the study of
 the level curves $w_\pm =\varpi^2$ solving $\Delta=0$.

 If $k^2=3/2$, then $a=0$ and thus
$$
\Delta=3\sqrt{6}\alpha^2r\varpi^2-r^2(48+2\sqrt{6}(3\alpha-4r)).
$$
Hence
 $$
 w_+=\frac{-8r^2+6\alpha r+8\sqrt{6}r}{3\alpha^2} \mbox{ and } w_- \mbox{ is not well-defined}.
 $$
 The maximum value of $w_+$  is attained at $r=\frac{3\alpha+4\sqrt{6}}{8}$.

If $k^2 \neq 3/2$, then $b^2-4ac = 4\alpha^4k^2r^2I$
where
\begin{equation}\label{formula of I}
I:=9((3+\alpha k-rk)(2k^2-3)-1)^2
-2(2k^2-3)(\alpha k(2k^2-3)-4)(12+3\alpha k-4kr).
\end{equation}
In this case, we obtain
\begin{equation}\label{equ: w pm}
w_\pm=\frac{3r}{2a^2k}\bigg(\alpha k-rk+3+\frac{-3\pm \sqrt{I}}{3(2k^2-3)}\bigg).
\end{equation}

\begin{lem}\label{lem: roots of Delta}
Assume that $k^2\neq 3/2$. Then the following assertions hold:
\begin{itemize}
\item[(i)] $I\geq 0$. Moreover, $I=0$ precisely at a unique point $(r,k), k^2> 3/2,$  satisfying
$$
\alpha=\frac{4}{k(2k^2-3)} \quad  \mbox{
 and  } \quad r=\frac{6(k^2-1)}{k(2k^2-3)}.
 $$

\item[(ii)] If $(r,k)$ is fixed, then $\Delta$, as a function of $\varpi^2$, is increasing  at $\varpi^2=w_+$ and decreasing at $\varpi^2=w_-$.

\item[(iii)]Fix $r>0$. Then $w_+$ smoothly extends near $k^2=3/2$,  $w_-$ is smooth in $k^2\neq 3/2,$ and
$\lim_{k^2\rightarrow{3/2}^\pm} w_-=\mp\infty.$

\item[(iv)]    If $I\neq 0$ and $\partial_r w_{\pm}=0$, then
\begin{equation}\label{pkw}
\partial_k w_\pm=\frac{18rF_1F_2}{\alpha^2k^2(2k^2-3)^2(I+rk(2k^2-3)F_0)},
\end{equation}
where
\begin{equation}\label{formula of F1 and F2}
\begin{aligned}
F_0 & := 74+15\alpha k-54k^2-10\alpha k^3-27kr+18k^3r,\\
F_1 & :=3\alpha+8k-\alpha k^2(3-2k^2) +4(1-2k^2)r,\\
F_2 & :=-16+12k^2-\alpha k(3-2k^2) +2(3-2k^2)rk.
\end{aligned}
\end{equation}
\item[(v)] If $I=0$, then   $F_2=0$ and $F_1 <0$.

\end{itemize}
\end{lem}

\begin{lem}\label{lem: the first zero appear on the red curve}
Assume $\varpi = \hat \varpi$, and let $(\hat r, \hat k)\in \mathcal{H}\setminus \partial \mathcal{H}$ solve $\Delta=0$. Then
$$
F_1(\hat r, \hat k) =0,
$$
where $F_1$ is defined in \eqref{formula of F1 and F2}.
\end{lem}

The proof of Lemmas \ref{lem: roots of Delta} and \ref{lem: the first zero appear on the red curve} are left to Appendix \ref{subsec: proof of lemma for convexity}.

\begin{figure}[ht]
\centering
\includegraphics[width=1\textwidth]{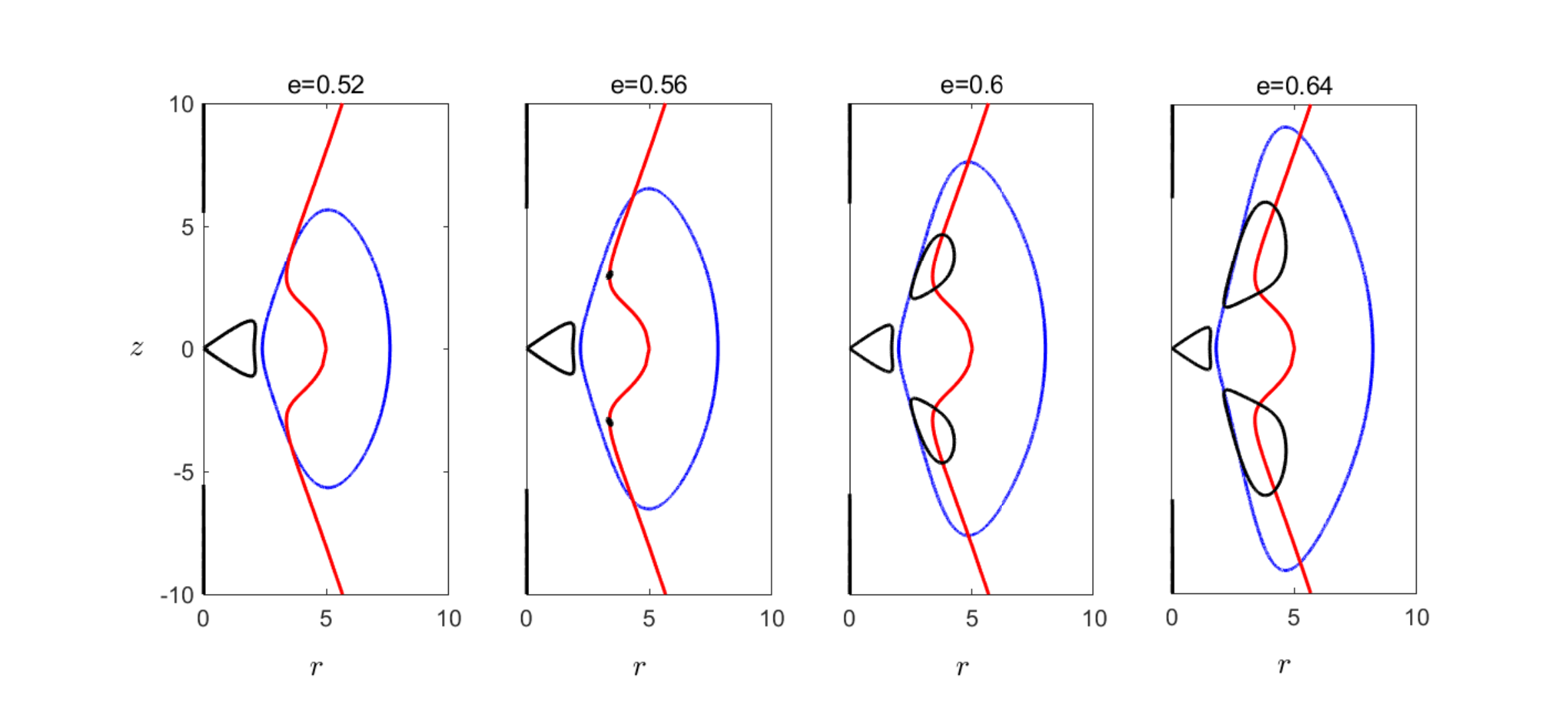}
\caption{The blue curve is the boundary of the Hill region on the $(r,z)$-plane. The black curve is the zero set of $\Delta$. The red curve is the set points where $F_1=0$  ($\alpha=6$ and increasing values of $\mathfrak{e}$, or, equivalently, decreasing values of $\varpi$).}
\label{picture of first zeros points}
\end{figure}


Now we study the behavior of the discriminant $I$ and the functions $w_\pm$ along the curve
$$
\ell_1:=\{F_1=0\},
$$
where $F_1$ is given in \eqref{formula of F1 and F2}. Notice that $\ell_1$ is the graph of a function $r=r_1(k), k\geq 1.$

\begin{lem}\label{lem: differential of w^2 along F1=0} Let
$$w_{1,\pm}(k)  :=w_\pm(r_1(k),k),\quad  \forall k\geq 1,$$ where
\begin{equation}\label{equ: r1}
r_1(k): = -\frac{3\alpha+8k-\alpha k^2(3-2k^2)}{4(1-2k^2)}, \quad \forall k\geq 1.
\end{equation}
Then following assertions hold:

\begin{itemize}
\item[(i)] The derivative of $w_{1,\pm}$ with respect to $k$ is
$$
\begin{aligned}
 w_{1,\pm}' & =\frac{6(k^2-1)F_3F_4}{\alpha^2 k^2(2k^2-1)^2\left(\pm D-(5-2k^2)\sqrt{I_1}\right)\sqrt{I_1}}\\
 & = \frac{3(k^2-1)F_3(\pm D+(5-2k^2)\sqrt{I_1})}{8\alpha^2 k^2(2k^2-1)^3(2k^2-3)^2\sqrt{I_1}},
\end{aligned}
$$
where
\begin{equation}\label{formula of F3 and F4}
\begin{aligned}
I_1 & := 16(2k^2-1)^2 I(r_1(k),k),\\
D & :=8(51-116k^2+88k^4-24k^6)\\
&+\alpha k(2k^2-5)(2k^2-3)(13-17k^2+6k^4),\\
F_3 & := 256k^3(k^2-1)+\alpha^2 k^3(2k^2-3)^3(1+2k^2)\\
&-2\alpha(-9+81k^2-54k^4-20k^6+8k^8),\\
F_4 & := \alpha^2 k^2(k^2-1)(2k^2-5)^3+32(17-34k^2+16k^4)\\
&-8\alpha k(2k^2-5)(13-22k^2+8k^4).
\end{aligned}
\end{equation}

\item[(ii)] $F_3$ has a unique root $1<\hat k_0 <\sqrt{3/2}$, which  increasingly depends on $\alpha$.

\item[(iii)] $F_4$ has precisely three roots
$$\frac{\sqrt{17+\sqrt{17}}}{4} < \hat k_1 < \sqrt{5/2}< \hat k_2 <  \hat k_3.$$
Moreover, $\hat k_0 < \hat k_1$,  $\hat k_1,-\hat k_2, -\hat k_3$  increasingly depend on $\alpha$, and satisfy
\begin{equation}\label{evaluation of D at three k's}
D(\hat k_1),D(\hat k_2)<0, D(\hat k_3)>0.
\end{equation}
\item[(iv)] $\hat k_1\in ( (17+\sqrt{17})^{1/2}/4,\sqrt{5/2})$ is in a smooth bijection with $\alpha\in (0,+\infty)$ according to the relation
$$
\alpha=\frac{4(13-22\hat k_1^2+8\hat k_1^4+(2\hat k_1^2-1)^{\frac{3}{2}})}{\hat k_1(2\hat k_1^2-5)^2(\hat k_1^2-1)}, \quad \forall \hat k_1.
$$
\item[(v)] $w_{1,-} > w_{1,+}$
 on $[1,\sqrt{3/2})$ and $w_{1,-} < w_{1,+}$
 on $(\sqrt{3/2},+\infty)$.

\item[(vi)] $\hat k_1$ is the unique local maximum of $w_{1,+}$, and  $\hat k_0$ and $\hat k_3$ are minimum points of $w_{1,+}$. $\hat k_2$ is the unique local maximum of $w_{1,-}$, and $\hat k_0$ is a local minimum of $w_{1,-}$.
\end{itemize}
\end{lem}
See the proof in Appendix \ref{subsec: proof of lemma for convexity} and see the graph of $w_{1,\pm}$ in Figure \ref{wpm} for $\alpha=6$.
\begin{figure}[ht]
\centering
\includegraphics[width=0.6\textwidth]{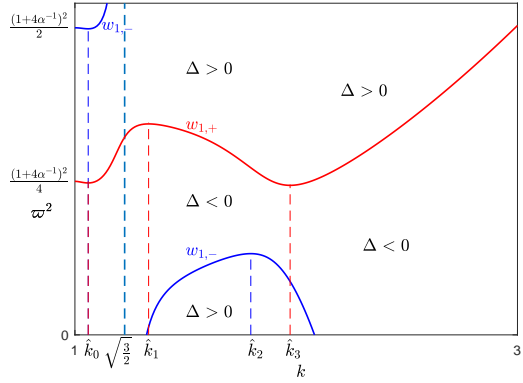}
\caption{The restricted functions $w_{1,+}$ and $w_{1,-}$ corresponding to $\Delta=0$, for $\alpha=6$. }
\label{wpm}
\end{figure}

Our next task is to compute $\hat \varpi$, as defined in \eqref{hatvarpi}, and identify the point $(\hat r,\hat k)\in \mathcal{H} \setminus \partial \mathcal{H}$ where convexity is lost for the first time as $\varpi$ decreases from $(1+4\alpha^{-1})/\sqrt{2}$ to $1/\sqrt{2}$.
In the next Lemma, we show that $(\hat r,\hat k) = (r_1(\hat k_1),\hat k_1)\in \mathcal{H} \setminus \partial \mathcal{H}$.

\begin{lem}\label{lem: the first zero appear inside}
For any fixed $\alpha>0$, we have
\begin{equation}\label{formula varpihat}
\hat \varpi^2 =  w_{1,+}(\hat k_1)
   = \frac{(\alpha \hat k_1(2\hat k_1^2-5)-8)(8\hat k_1+\alpha(3-3\hat k_1^2+2\hat k_1^4))}{4\alpha^2\hat k_1(5-12\hat k_1^2+4\hat k_1^4)}.
\end{equation}
Moreover, the following statements hold:
\begin{itemize}
    \item[(i)]  if $\varpi = \hat \varpi$, then  $(r_1(\hat k_1),\hat k_1)\in \mathcal{H} \setminus \partial \mathcal{H}$ is the only point in $\mathcal{H}$ solving $\Delta =0$.
    \item[(ii)]  if $1/\sqrt{2}< \varpi < \hat \varpi$, then $\min \Delta|_{\mathcal{H}} <0.$
    \end{itemize}
  Here, $w_{1,+}$, $r_1$ and $\hat k_1$ are as in Lemma \ref{lem: differential of w^2 along F1=0}.
\end{lem}

\begin{figure}[ht]
\centering
\includegraphics[width=0.55\textwidth]{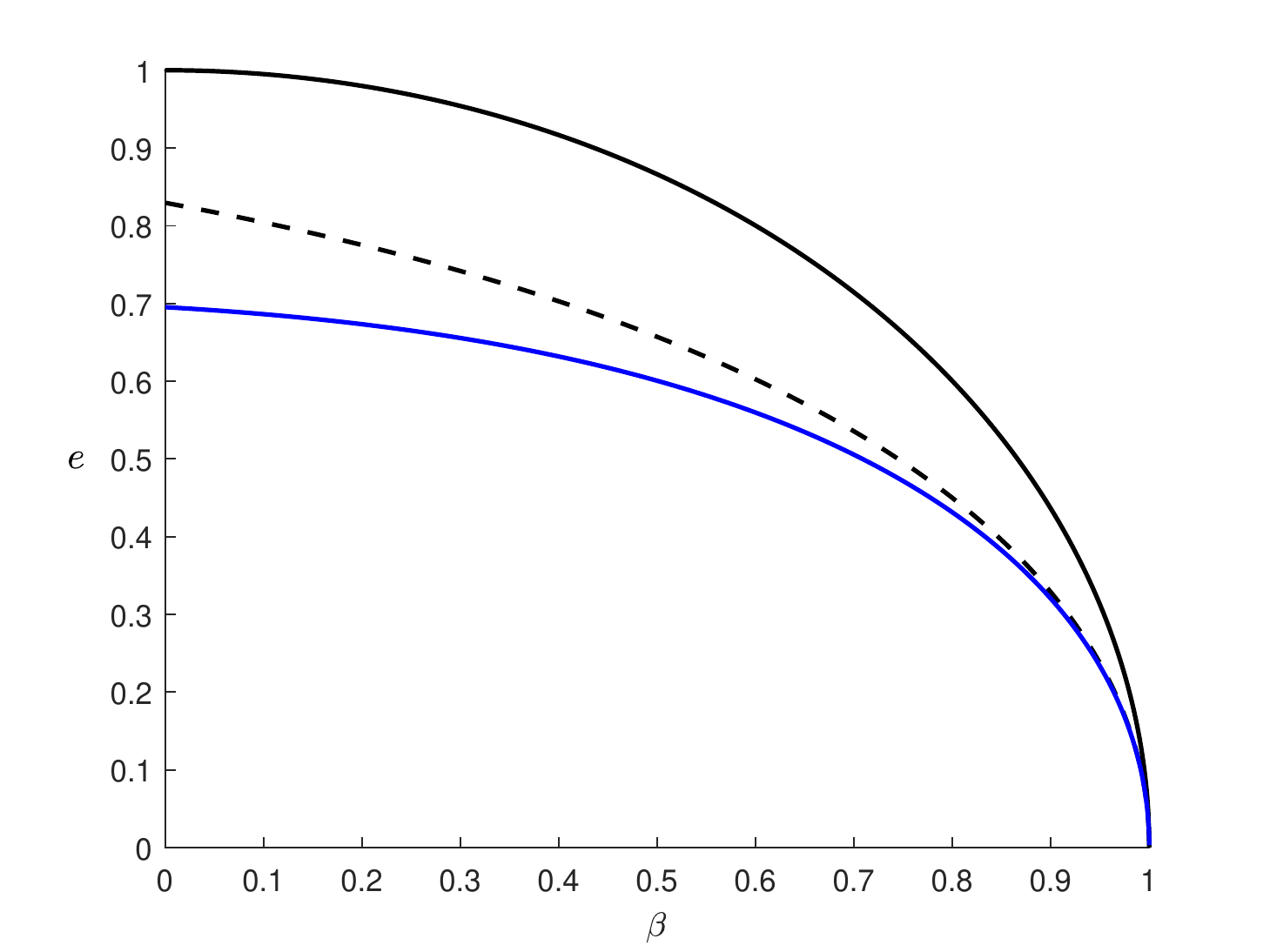}
\caption{The blue curve is the graph  $\mathfrak{e} = \mathfrak{e}_{\text{conv}}(\beta), \beta \in [0,1],$  corresponding to the loss of convexity in the interior of $\mathcal{H}$ as $\mathfrak{e}$ increases. The dashed curve corresponds to the loss of convexity on the boundary of $\mathcal{H}$.}
\label{picture of econv and bconv}
\end{figure}

See the proof in Appendix \ref{subsec: proof of lemma for convexity}. Now we are ready to prove Theorem \ref{thm_convexity}.

\subsection{Proof of Theorem \ref{thm_convexity}}

For each fixed $\alpha>0$, we know from  Lemmas \ref{lem: differential of w^2 along F1=0} and \ref{lem: the first zero appear inside} that the energy surface $\mathfrak{M}$ changes from strictly convex to non-convex precisely at $\varpi=\hat \varpi$, where
\begin{equation}\label{varpik1}
\hat \varpi^2=\frac{(\alpha \hat k_1(2\hat k_1^2-5)-8)(8\hat k_1+\alpha(3-3\hat k_1^2+2\hat k_1^4))}{4\alpha^2\hat k_1(5-12\hat k_1^2+4\hat k_1^4)}\in(1/2,(1+4\alpha^{-1})^2/2),
\end{equation}
and $\hat k_1$ is the unique point
in the interval $( (17+\sqrt{17})^{1/2}/4,\sqrt{5/2})$
satisfying
\begin{equation}\label{alphak1}
\alpha=\frac{4(13-22\hat k_1^2+8\hat k_1^4+(2\hat k_1^2-1)^{\frac{3}{2}})}{\hat k_1(2\hat k_1^2-5)^2(\hat k_1^2-1)}\in (0,+\infty).
\end{equation}
Therefore, the curve $\hat k_1 \mapsto (\alpha(\hat k_1), \hat \varpi(\hat k_1))$ separates strict convexity of $\mathfrak{M}$ from non-convexity of $\mathfrak{M}$ in the parameter space $(\alpha, \varpi)$. This means that if $(\alpha,\varpi)$ satisfies $\alpha = \alpha(\hat k_1)$ and $\varpi > \hat \varpi(\hat k_1),$ then $\mathfrak{M}$ is strictly convex. Otherwise, if $\alpha = \alpha(\hat k_1)$ and $\varpi < \hat \varpi(\hat k_1),$ then $\mathfrak{M}$ is not convex.

Defining a new parameter $
\nu =: (2\hat k_1^2-1)^{1/2}, \nu \in ((1+\sqrt{17})/4,2),$
and using the relations $\beta = \alpha/(\alpha + 4)$ and $\mathfrak{e} = (1-2\varpi^2 \beta^2)^{1/2}$, we obtain a new curve $\nu \mapsto (\beta_{\text{conv}},\mathfrak{e}_{\text{conv}})(\nu)\in \mathcal{D}$ corresponding to the convexity transition of $\mathfrak{M}$ in the  parameter space $(\beta,\mathfrak{e})$.

A straightforward computation using \eqref{varpik1} and \eqref{alphak1} gives
$$
\beta_{\text{conv}}(\nu)  = \frac{2\sqrt{2}(2\nu^2-\nu-2)}{G(\nu)} \quad \mbox{ and } \quad
\mathfrak{e}_{\text{conv}}(\nu)  = \big(1- 2W(\nu)\big)^{1/2},
$$
where
$$
\begin{aligned}
    W(\nu) & := \frac{\nu^2(12 +4\nu - 9 \nu^2-2\nu^3+3\nu^4)}{G(\nu)^2},\\
    G(\nu) & : = 2\sqrt{2}(-2 - \nu +2\nu^2)+(2-\nu)^2(1+\nu)(2+\nu)(1+\nu^2)^{1/2}.
\end{aligned}
$$
It is immediate to check that $G$ is positive on $[(1+\sqrt{17})/4,2)$, and
$$
\begin{aligned}
(\beta_{\text{conv}},\mathfrak{e}_{\text{conv}})((1+\sqrt{17})/4) & = (0,(7+\sqrt{17})/16),\\ (\beta_{\text{conv}},\mathfrak{e}_{\text{conv}})(2) & = (1,0).
\end{aligned}
$$
The image of  $(\beta_{\text{conv}},\mathfrak{e}_{\text{conv}})$ is depicted in Figure \ref{picture of econv and bconv}.

 Since $\varpi$ is strictly decreasing in $\mathfrak{e}$, we conclude that if $(\beta,\mathfrak{e})\in \mathcal{D}$ satisfies $\beta = \beta_{\text{conv}}$ and $\mathfrak{e}< \mathfrak{e}_{\text{conv}}$, then $\mathfrak{M}$ is strictly convex. Otherwise,  if $\beta = \beta_{\text{conv}}$ and $\mathfrak{e}> \mathfrak{e}_{\text{conv}}$, then $\mathfrak{M}$ is not convex.

Since $\beta$ is strictly increasing with $\alpha$, it is also strictly increasing with $\nu$. In fact,
$$
\beta_{\text{conv}}'(\nu)=-\frac{2\sqrt{2} W_1(\nu)}{(1+\nu^2)^{1/2}G(\nu)^2}>0,
$$
where
$$
\begin{aligned}
W_1(\nu) : = &  -3(2-\nu)^2(2-\nu)-3(2-\nu)^3(2-\nu)\nu\\
  & -(34-25\nu+7\nu^2+6\nu^3)(2-\nu)\nu^2
\end{aligned}
$$
is negative on $ [(1+\sqrt{17})/4,2)$, see the proof of Lemma \ref{lem: the first zero appear inside}.

 From $\beta'_{\text{conv}}>0$, the curve $(\beta_{\text{conv}},\mathfrak{e}_{\text{conv}})$ can be expressed as a graph
 $$
 \mathfrak{e} = \mathfrak{e}_{\text{conv}}(\beta), \quad \beta \in [0,1].
 $$ Below, we show that
  $\mathfrak{e}_{\text{conv}}(\beta)$ is a strictly decreasing and concave function.
To show that $\mathfrak{e}_{\text{conv}}(\beta)$ is strictly decreasing, we compute
$$
\mathfrak{e}_{\text{conv}}'(\beta)=\frac{\mathfrak{e}_{\text{conv}}'(\nu)}{\beta'(\nu)}=\frac{W_2(\nu)}{\sqrt{2}G(\nu) \mathfrak{e}_{\text{conv}}(\nu)}<0, \quad \forall \beta = \beta_{\text{conv}}(\nu),
$$
where
$$
W_2(\nu)  :=  -4+\nu^2-\nu^4+2\sqrt2\sqrt{1+\nu^2}
$$
is negative on $[(1+\sqrt{17})/4,2)$, see the proof of Lemma \ref{lem: the first zero appear inside}.

Finally, we compute
$$
\begin{aligned}
\mathfrak{e}_{\text{conv}}''(\beta)& =\frac{1}{\beta'(\nu)}\frac{d}{d\nu}( \mathfrak{e}_{\text{conv}}'(\nu)/\beta'(\nu))\\
& =\frac{W_2'G\mathfrak{e}_{\text{conv}}-W_2G'\mathfrak{e}_{\text{conv}}-W_2G\mathfrak{e}_{\text{conv}}'}{\sqrt{2}G^2\mathfrak{e}_{\text{conv}}^2\beta'}(\nu)<0,
\end{aligned}
$$
for every $\beta = \beta_{\text{conv}}(\nu)\in [0,1].$

To prove it, we compute
$$
\begin{aligned}
(W_2'G-W_2G')(\nu) &= (S_1(\nu)+(2+2\nu^2)^{1/2}S_2(\nu))(1+\nu^2)^{-1/2}\\
\end{aligned}
$$
where
$$
\begin{aligned}
S_1(\nu)&:=24-48\nu+24\nu^2-96\nu^3-27\nu^4+12\nu^5-5\nu^6+3\nu^7+\nu^9,\\
S_2(\nu)& :=4(2-\nu)(1+\nu^2)(4\nu^2+5\nu-2)>0.
\end{aligned}
$$
Clearly, $S_2$ is positive on $((1+\sqrt{17})/4,2)$, and thus
$$
\begin{aligned}
(W_2'G-W_2G')(\nu) & <
(S_1(\nu)+(1+\nu^2)S_2(\nu))(1+\nu^2)^{-1/2}\\
& =(\nu^2-1)S_3(\nu)(1+\nu^2)^{-1/2},
\end{aligned}
$$
where
$$
S_3(\nu): = -8-12\nu^2+16\nu^3+7\nu^4-12\nu^5+\nu^7.
$$

Since
$$
\begin{aligned}
S_3(2-s)= & -(s-1)(s-13)s^5-(s-1)(59s-108)s^3\\
& -4(18-36s+9s^2+11s^3),
\end{aligned}
$$
is negative for every $s\in(0,1)$, we conclude that $S_3$ is negative on $((1+\sqrt{17})/4,2)$. Hence $\mathfrak{e}_{\text{conv}}''(\beta)<0$ for every $\beta \in [0,1].$
The proof of Theorem \ref{thm_convexity} is complete.

\appendix

\section{The Maslov-type index for symplectic paths} \label{sec: Maslov-type index}
The Maslov-type index (or the $\omega$-index) is a generalized version of the Conley-Zehnder index. It was developed by Long during the 1990s, see \cite{Lon02}. In this section, we briefly introduce this general concept and explain the Conley-Zehnder index and the mean index. Moreover, we provide an alternative proof of Theorem \ref{thm:Dk} using the Maslov-type index theory, see Remark \ref{rem: alternative proof}.

Denote by $\mathrm{Sp}(2n)$ the symplectic group, i.e. the collection of $M\in {\rm GL}(2n, \mathbb{R})$ so that $M^TJ_0M=J_0$, where
$$
J_0 = \left(\begin{array}{cc} 0_n & -I_{n}\\ I_n & 0_n \end{array} \right),
$$
 $I_n$ is the $n \times n$ identity matrix, and $0_n$ is the $n\times n$ zero matrix.

Fix $T>0$ and consider the set $\mathcal{P}_{T}(2n)$ of symplectic paths $\gamma:[0,T] \to {\rm Sp}(2n)$ starting from  $\gamma(0)=I_{2n} \in \Sp(2n)$. For every $\omega$ in the unit circle $\mathbb{U}\subset \C$, we denote by
$\Sp(2n)_\omega^0$ the collection of $M\in\Sp(2n)$ possessing the eigenvalues $\omega$ and $\bar\omega$. This set is a boundaryless singular hypersurface of $\Sp(2n)$ and we let $\Sp(2n)^*_{\omega}$ be the complement of $\Sp(2n)_\omega^0$. The nowhere vanishing vector field
$$\mathcal{V}:=\frac{d}{dt}Me^{tJ_0}\big|_{t=0},\quad \quad
 \forall  M\in \Sp(2n)_\omega^0,$$
is everywhere transverse to $\Sp(2n)_\omega^0$,  and hence determines a positive co-orientation of $\Sp(2n)_\omega^0$.

For any continuous paths $\xi,\eta:[0, T] \rightarrow \operatorname{Sp}(2n)$ with $\xi(T)=\eta(0)$, we denote their concatenation path by $\eta * \xi$. Let
$$
\xi_{n}(t):=\begin{pmatrix}
L(t) & 0_n \\
0_n & L(t)^{-1}
\end{pmatrix},\quad L(t):=(2-t/T)I_n,\quad \quad \forall 0 \leq t \leq T.
$$

The Maslov-type index is then introduced as follows:

\begin{defi}\label{def:Maslov-type index} For every $\omega\in \mathbb{U}$, and $\gamma\in \mathcal{P}_{T}(2n)$, the $\omega$-index of $\gamma$ is defined as the algebraic intersection number
\[i_\omega(\gamma):=(e^{-\vep J_0}(\gamma*\xi_n))\cdot \Sp(2n)_\omega^0, \]
where $\vep>0$ is small enough. The $\omega$-nullity of $\gamma$ is defined as $$\nu_\omega(\gamma):=\dim_{\mathbb{C}}\ker_{\mathbb{C}}(\gamma(T)-\omega I_{2n}).$$
\end{defi}

Given a symplectic path $\gamma\in\mathcal{P}_{T}(2n)$ and $m\in\mathbb{Z}^+$, we define the $m$-th iteration $\gamma^m: [0,mT]\to\Sp(2n)$ as
$$
\gamma^{m}(t):=\gamma(t-j T) \gamma(T)^{j}, \quad \forall j T \leqslant t \leqslant(j+1) T,\quad \forall j=0,1, \ldots, m-1.
$$
The mean index of $\gamma$ is defined by
\begin{equation} \label{def of mean index}
\hat{i}(\gamma)=\lim_{m\to\infty}\frac{i_1(\gamma^m)}{m}=\frac{1}{2 \pi} \int_{0}^{2 \pi} i_{e^{\sqrt{-1} \theta}}(\gamma)\, d \theta,
\end{equation}
where the last identity follows from Bott iteration formula.

If $\omega=1$, then the $\omega$-index $i_1(\gamma)$ coincides with the generalized Conley-Zehnder index ${\rm CZ}(\gamma)$ as defined in \cite{convex}.

An alternative way of understanding the Conley-Zehnder index in $\Sp(2)$ is as follows: let $v\in \mathbb{C}\setminus \{0\}$, and  let $\eta(t)$ be a continuous argument of $\gamma(t)v$. Let $\Delta(v):=(\eta(T) - \eta(0))/(2\pi)$ be the net variation of $\eta$ on the interval $[0,T]$. Define
$$
I_\gamma:= \{\Delta(v): v \in \mathbb{C} \setminus \{0\}\} \subset \mathbb{R}.
$$
Then $I_\gamma$ is a closed interval of length $|I_\gamma|<1/2$. Therefore, there exists $k\in \mathbb{Z}$ such that for every $\epsilon>0$ sufficiently small, either $I_\gamma-\epsilon\in (k,k+1)$ or $k\in I_\gamma-\epsilon$. In the former case, we define the generalized Conley-Zehnder index of $\gamma$ by $\mathrm{CZ}(\gamma):=2k+1,$ and in the latter case,  by $\mathrm{CZ}(\gamma): = 2k.$ Recall that $\gamma$ is called nondegenerate if the linear map $\gamma(T)$ does not have $1$ as an eigenvalue. In that case, the boundary of $\partial I_\gamma \cap \mathbb{Z}=\emptyset.$

The rotation number of $\gamma$ is defined as
\begin{equation}\label{def: mean index and rotation number}
\rho(\gamma) :=\frac{\hat i(\gamma)}{2}= T\lim_{t\to \infty}\frac{\eta(t)}{2\pi t}.
\end{equation}

In \cite{Lon02}, Long  provided an explicit formula for the mean index. Firstly, any $M\in \Sp(2)$ is symplectically similar to one of the following normal forms:
\begin{equation*}
D(\lambda)=\begin{pmatrix}
\lambda & 0 \\
0 & \lambda^{-1}
\end{pmatrix},\
N_{1}(\nu, a)=\begin{pmatrix}
\nu & a \\
0 & \nu
\end{pmatrix},\
R(\theta)=\begin{pmatrix}
\cos \theta & -\sin \theta \\
\sin \theta & \cos \theta
\end{pmatrix},
\end{equation*}
where $\lambda\in \mathbb{R}\setminus\{0\}$, $\nu=\pm 1, a=\pm 1,0$, and $\theta \in(0, \pi) \cup(\pi, 2 \pi)$. As a particular case of Corollary 8.3.2 in \cite{Lon02}, we have
\begin{equation} \label{mean index 2}
\hat{i}(\gamma)=\left\{\begin{array}{ll}
i_1(\gamma)-1+\theta/\pi, & \text {if } \gamma(T)\approx R(\theta), \theta \in(0, \pi) \cup(\pi, 2 \pi).  \\
i_1(\gamma)+1, & \text {if }\gamma(T)\approx N_1(1,a), a=1,0. \\ i_1(\gamma), & \text{otherwise.}
\end{array}\right.
\end{equation}

Finally we introduce the Morse index theorem needed in Section \ref{subsec: return map with Rational rotation}. Consider the second order linear system $\ddot{x}=R(t)x$, where $R(t+T)=R(t)$ is a continuous family of $n \times n$ symmetric matrices. Let $\mathcal{A}:=-\frac{d^2}{dt^2}+R(t)$ be a self-adjoint operator on $W^{2,2}([0,T],\C^n) \subset L^2([0,T], \mathbb{C}^n)$ with self-adjoint boundary conditions.

Let $v(\mathcal{A}):=\dim\ker(\mathcal{A})$ be the nullity of $\mathcal{A}$, and let $\mathcal{M}(\mathcal{A})$ be the Morse index of $\mathcal{A}$, i.e. the total (finite) number of negative eigenvalues of $\mathcal{A}$.

By standard Legendre transformation, the system $\ddot{x}=R(t)x$ becomes into
\bea\label{ham equ}
\dot{z}(t)=J_0\mathcal{B}(t)z(t),\quad \mathcal{B}(t)=\begin{pmatrix}I_n & 0_n \\ 0_n & -R(t)\end{pmatrix},
\eea
where $z= (x,\dot x).$
Let $\gamma(t)$ be the fundamental solution of \eqref{ham equ} starting from $\gamma(0)=I_{2n}$.

\begin{thm}[{Long \cite[Theorem 7.3.4]{Lon02}}] Let $\mathcal{A}_{\omega}:=\mathcal{A}|_{\Lambda_\omega},$ where
$\Lambda_\omega:=\{x\in W^{2,2}([0,T],\C^n): x(T)=\omega x(0) \mbox{ and } \dot x(T) = \omega \dot x(0)\}.$
Then
\begin{equation}\label{more equa mas omega}
 \mathcal{M}(\mathcal{A}_{\omega})=i_\omega(\gamma),\quad v(\mathcal{A}_{\omega})=\nu_\omega(\gamma).
\end{equation}
\end{thm}


\section{Proof of Lemmas \ref{lem: w^2 for the first point on the boundary},  \ref{lem: roots of Delta}, \ref{lem: the first zero appear on the red curve},  \ref{lem: differential of w^2 along F1=0} and \ref{lem: the first zero appear inside}}\label{subsec: proof of lemma for convexity}

 Some algebraic manipulations below were performed by Mathematica \cite{Math}.

\begin{proof}[Proof of Lemma \ref{lem: w^2 for the first point on the boundary}]
In $(r,k)$ coordinates, the Hill region $\mathcal{H}$ is determined by
$$
V=\frac{\varpi^2\alpha^2}{2r^2}-\frac{\alpha}{r} -\frac{4}{kr}\leq -1,$$
and hence if $(r,k)\in \mathcal{H}$, then $k\in[1,\frac{4}{\alpha(\sqrt{2}\varpi-1)}]$ and $r\in [r_L,r_R]$, where
\begin{equation}\label{formula of rL and rR}
r_L : = \frac{4+\alpha k-J}{2k} \quad \mbox{ and } \quad
r_R :=\frac{4+\alpha k + J}{2k},
\end{equation}
and
$$
J : =\sqrt{(4+\alpha k)^2-2\varpi^2\alpha^2k^2}\in[0,\sqrt{(4+\alpha)^2-\alpha^2k^2}).
$$
If $k=\frac{4}{\alpha(\sqrt{2}\varpi-1)}$, then $r_L=r_R$ and $J=0$.

Using that
\begin{equation}\label{w2J}
  \varpi^2 = \frac{(4+\alpha k)^2 - J^2}{2\alpha^2 k^2},
\end{equation} and that $4+\alpha k\pm J>0$,
we find
\begin{equation}\label{def: Delta_L and Delta_R}
\begin{aligned}
\Delta_L & :=\Delta(r_L,k)\cdot(4+\alpha k-J)^{-2}4k^2\\
& =-(2k^2-3)J^2-4(k^2-1)J+4(4+\alpha k)(k^2-1),\\
\Delta_R  & :=\Delta(r_R,k)\cdot(4+\alpha k+J)^{-2}4k^2\\
&=-(2k^2-3)J^2+4(k^2-1)J+4(4+\alpha k)(k^2-1).
\end{aligned}
\end{equation}

If $k=1$, then $\Delta_L= \Delta_R = J^2>0$ since, otherwise, if $J=0$, we obtain from \eqref{w2J} that $\varpi^2 = (1+4\alpha^{-1})^2/2$, a contradiction. This estimate along the boundary will be improved below.

 If $k>1$, then
\begin{equation}\label{DeltaLR}
\begin{aligned}
\Delta_L&=0  \Leftrightarrow J=J_{L,\pm}:=\frac{2(k^2-1)}{2k^2-3}\left(-1\pm\sqrt{1+\frac{(4+\alpha k)(2k^2-3)}{k^2-1}}\right),\\
\Delta_R& =0  \Leftrightarrow J=J_{R,\pm}:=\frac{2(k^2-1)}{2k^2-3}\left(1\pm\sqrt{1+\frac{(4+\alpha k)(2k^2-3)}{k^2-1}}\right).
\end{aligned}
\end{equation}

If $1<k^2<3/2$ and $J_{L,\pm},J_{R,\pm}$ are not real numbers, then $\Delta_L,\Delta_R >0$. Otherwise, if $J_{L,\pm},J_{R,\pm}$ are real numbers, then
$$
J_{R,+}\leq J_{R,-}< 0< J_{L,+}\leq J_{L,-}.
$$
Since $0\leq J < 4+\alpha k$, we see from \eqref{def: Delta_L and Delta_R} that
$$
\Delta_R>\Delta_L> 4(k^2-1)(-J + 4 + \alpha k) >0.
$$

If $k^2=3/2$, then, as in the previous case, we have $\Delta_R > \Delta_L > 0$.

If $\sqrt{3/2}< k\leq \frac{4}{\alpha(\sqrt{2}\varpi-1)}$, then $J_{L,\pm},J_{R,\pm}$ are real numbers and
$$
J_{L,-}< J_{R,-}< 0 < J_{L,+}< J_{R,+}.
$$

We conclude from the computations above that
\begin{equation}\label{boundary1}
\Delta|_{\partial \mathcal{H}} > 0 \Leftrightarrow 0\leq J < J_{L,+}, \quad  \forall \sqrt{3/2} <k \leq \frac{4}{\alpha(\sqrt{2}\varpi-1)}.
\end{equation}

Using \eqref{w2J}, we translate \eqref{boundary1} into
\begin{equation}\label{boundary2}
\Delta|_{\partial \mathcal{H}} > 0 \Leftrightarrow \varpi^2 > W(k), \quad  \forall \sqrt{3/2} <k \leq \frac{4}{\alpha(\sqrt{2}\varpi-1)},
\end{equation}
where
$$
W(k):=\frac{(4+\alpha k)^2-J_{L,+}^2}{2\alpha^2k^2}.
$$
Similarly,  for any $\sqrt{3/2} <k \leq \frac{4}{\alpha(\sqrt{2}\varpi-1)}$, we have
\begin{equation}\label{boundary2b}
\Delta(r_L(k),k)<0 \Leftrightarrow \varpi^2 < W(k).
\end{equation}

We shall study the maximum value of $W$ for $k\in [\sqrt{3/2},4\alpha^{-1}/(\sqrt{2}\varpi-1)]$.

We claim that  $W$ has only one critical point $k_0$ in $[\sqrt{3/2},+\infty)$ which is a global maximum, and satisfies  $k_0^2\in(\frac{11+\sqrt{17}}{8},5/2)$ and
\begin{equation}\label{equ of k0}
4(13-22k_0^2+8k_0^4)-\alpha k_0(k_0^2-1)(2k_0^2-5)^2=0.
\end{equation}
To see this, we compute
\begin{eqnarray*}
W'=-\frac{6(4+\alpha k)^2(D_1+2(2k^2-3)D_2^{1/2})}{\alpha^2k^3(k^2-1)D_2^{1/2}(1+D_2^{1/2})^3},
\end{eqnarray*}
where
$$
\begin{aligned}
D_1 & =-26-5\alpha k+12k^2+2\alpha k^3,\\
D_2 & =1+\frac{(4+\alpha k)(2k^2-3)}{k^2-1}.
\end{aligned}
$$
Direct computation shows that $D_1,D_2$ are both strictly increasing when $k^2>3/2$, and also $D_1(\sqrt{3/2})=-8-\sqrt{6}\alpha<0\Rightarrow W'(\sqrt{3/2})>0$. Therefore, $W'$ possesses only one zero  $k_0$ in $[\sqrt{3/2},+\infty)$, which is a maximum of $W$. We can further check that
$-D_1(k_0)^2 + 4(2k_0^2-3)^2D_2(k_0)$ is equal to
$$
 \frac{4+\alpha k_0}{k_0^2-1}(4(13-22k_0^2+8k_0^4)-\alpha k_0(k_0^2-1)(2k_0^2-5)^2).
$$
It shows that $k_0$ satisfies \eqref{equ of k0}.
Since $13-22k_0^2+8k_0^4=0$ at $k_0=\sqrt{\frac{11+\sqrt{17}}{8}}$ and $$D_1(\sqrt{\frac{11+\sqrt{17}}{8}})<0,\ D_1(\sqrt{5/2})>0,$$ we have
$W'(\sqrt{\frac{11+\sqrt{17}}{8}})>0$ and $W'(\sqrt{5/2})<0$. Hence  $k_0^2\in(\frac{11+\sqrt{17}}{8},5/2)$.

Note that $k_0$ does not depend on $\varpi$ and we define
$$
\varpi_0:= \sqrt{W(k_0)}.
$$

Now if $\varpi > \varpi_0,$ then we obtain from \eqref{boundary2} that $\Delta|_{\partial \mathcal{H}}>0.$ This proves (i).

If $\varpi< \varpi_0$, then we claim that $k_0 < 4\alpha^{-1}/(\sqrt{2}\varpi-1).$ Indeed, using \eqref{equ of k0}, we obtain
\begin{equation}\label{def: alpha(k0) and w^2b(k0)}
\alpha=\frac{4(13-22k_0^2+8k_0^4)}{k_0(k_0^2-1)(2k_0^2-5)^2} \Rightarrow \varpi_0^2=W(k_0)=\frac{(2k_0^2-3)^4(2k_0^2-1)}{2(13-22k_0^2+8k_0^4)^2}.
\end{equation}
Using \eqref{def: alpha(k0) and w^2b(k0)}, we now check that $J(k_0)>0$ and this implies that $k_0 < 4\alpha^{-1}/(\sqrt{2}\varpi -1).$ Indeed,
$$
\begin{aligned}
J(k_0)^2 &  = (4+\alpha k_0)^2 - 2\varpi^2\alpha^2k_0^2 \\ & > (4+\alpha k_0)^2 - 2\varpi^2_0\alpha^2k_0^2
\\ & = \frac{16(2k_0^2-3)^2}{5-2k_0^2} >0.
\end{aligned}
$$
We conclude from \eqref{boundary2b} that there exists $r_0>0$ so that $(r_0,k_0) \in \partial \mathcal{H}$ and $\Delta(r_0,k_0)<0$. This proves (ii).

One easily checks from \eqref{def: alpha(k0) and w^2b(k0)} that $\alpha$ is strictly increasing in $k_0\in (\sqrt{\frac{11+\sqrt{17}}{8}},\sqrt{5/2})$ and thus there exists a bijective correspondence between $\alpha\in (0,+\infty)$ and $k_0 \in (\sqrt{\frac{11+\sqrt{17}}{8}},\sqrt{5/2})$.

To prove (iii), we first compute $r_L(k_0)$ using  \eqref{formula of rL and rR} and \eqref{def: alpha(k0) and w^2b(k0)},  assuming that $\varpi = \varpi_0$,
$$
r_L(k_0)=\frac{2(2k_0^2-3)^3}{k_0(5-2k_0^2)^2(k_0^2-1)}.
$$
Then we compute
$$
\partial_k\Delta(r_L(k_0),k_0)=\frac{192(2k_0^2-3)^6(2k_0^2-1)}{k_0^3(k_0^2-1)^2(5-2k_0^2)^5}>0,
$$
using $\alpha$ and $\varpi=\varpi_0$ as in \eqref{def: alpha(k0) and w^2b(k0)}.
Hence $\Delta(r_L(k_0),k)<0$ for $k_0-k>0$ sufficiently small. Since $V(r_L(k_0),k_0)=-1$ and
$
V=\frac{\varpi^2_0\alpha^2}{2r^2}-\frac{\alpha+4/k}{r}$
 is strictly increasing with $k$, we conclude that $V(r_L(k_0),k)<-1$ for any $1\leq k<k_0$.  Hence $\Delta$ is somewhere negative in $\mathcal{H} \setminus \partial \mathcal{H}$ close to $(r_L(k_0),k_0)$. The proof of (iii) is complete. \end{proof}

\begin{proof}[Proof of Lemma \ref{lem: roots of Delta}] We rewrite $I$ as a quadratic form on $r$
\begin{equation*}
\begin{aligned}
I &=9(2k^3-3k)^2r^2+2(2k^3-3k)(74-54k^2-5\alpha(2k^3-3k))r\\
&-6(4+\alpha k)(2k^2-3)(\alpha k(2k^2-3)-4)+9((3+\alpha k)(2k^2-3)-1)^2.
\end{aligned}
\end{equation*}
A direct computation at the minimum point of this quadratic form gives
$$r=\frac{74-54k^2-5\alpha(2k^3-3k)}{-9(2k^3-3k)} \Rightarrow I=\frac{2}{9}(\alpha k(2k^2-3)-4)^2\geq 0.$$
Hence $I=0$ is equivalent to $\alpha=\frac{4}{k(2k^2-3)}$ and
$$
r=\frac{74-54k^2-5\alpha(2k^3-3k)}{-9(2k^3-3k)}=\frac{6(k^2-1)}{k(2k^2-3)}.$$ Since $k^2>3/2$, the point $(r,k)$ is unique. This proves (i).

Item (ii) follows immediately from the derivative of $\Delta$ with respect to $\varpi^2$ computed at $w_\pm$
$$
\partial_{\varpi^2} \Delta(w_\pm) = 2aw_\pm+b = \pm 2\alpha^2kr \sqrt{I}\geq 0.
$$

To see (iii), we rewrite \eqref{equ: w pm} as
$$
w_+=\frac{-2c}{b+\sqrt{b^2-4a c}} \quad \mbox{  and  } \quad w_-=\frac{-b-\sqrt{b^2-4a c}}{2a}.
$$
Since $b>0$ near $k^2=3/2$, we see that $w_+$ has a smooth continuation near $k^2=3/2$. Moreover, since $a=2\alpha^4k^2(2k^2-3)$, we see that $w_- \to \mp \infty$ as $k^2 \to 3/2^\pm$.

To prove (iv) we compute
 $$
 \partial_r w_{\pm}=\frac{3}{2\alpha^2k}\left(\alpha k-2rk+3+\frac{-3\sqrt{I}\pm (I+ rk(2k^2-3)F_0)}{3(2k^2-3)\sqrt{I}} \right),
 $$
 where $F_0$ is as in the statement.
Assuming that $\partial_rw_\pm=0$, we obtain
$$
\sqrt{I}=\pm \frac{-I - rk(2k^2-3)F_0}{(2k^2-3)(3\alpha k-6rk + 9) -3}.
$$
Replacing the expression above into the expression for $\partial_k w_\pm$, we obtain \eqref{pkw}.

If $I=0$, then  (i) gives $k^2>3/2$ and using the values of $\alpha$ and $r$ we obtain $F_2=0$. Moreover,
$$
F_1=\frac{12(2k^2-1)(k^2-1)}{k(3-2k^2)}<0, \quad \forall k^2>3/2.$$
This proves (v).
\end{proof}

\begin{proof}[Proof of Lemma \ref{lem: the first zero appear on the red curve}]
Firstly, if $\varpi=\hat \varpi$ and $(\hat r,\hat k)\in \mathcal{H} \setminus \partial \mathcal{H}$ solves $\Delta=0$, then either  $\varpi^2 = w_+(\hat r,\hat k)$ or $\varpi^2 = w_-(\hat r,\hat k)$.
From the definition of $\hat \varpi$ and Lemma \ref{lem: w^2 for the first point on the boundary}, $(\hat r,\hat k)$ is a local maximum of $w_+$ or $w_-$.
We have to consider the following cases:
\begin{itemize}
\item[(a)] $\hat k^2 = 3/2$. In this case, $w_-$ is not well-defined and, since $w_+(r,\sqrt{3/2})=\frac{-8r^2+(6\alpha+8\sqrt{6})r}{3\alpha^2}$ has a unique critical point, we have $\hat r=\frac{3\alpha+4\sqrt{6}}{8}$. A direct computation gives $F_1(\hat r,\hat k)=0$.

\item[(b)] $I(\hat r,\hat k)=0.$ In this case, Lemma \ref{lem: roots of Delta}-(v) gives $F_2(\hat r,\hat k)=0$, and we are led to the next case.

\item[(c)] $F_2(\hat r,\hat k)=0$. From \eqref{formula of F1 and F2}, we obtain $\hat r=\frac{-16-3\alpha \hat k+12\hat k^2+2\alpha \hat k^3}{2\hat k(2\hat k^2-3)}$. Then a direct computation gives
\begin{equation}
\partial_r w_\pm(\hat r,\hat k) = \frac{6(1-\hat k^2)}{\alpha^2\hat k(2\hat k^2-3)} \quad  \mbox{ or } \quad\partial_r w_\pm(\hat r,\hat k) =  \frac{12(1-\hat k^2)}{\alpha^2\hat k(2\hat k^2-3)}.
\end{equation}
From $\partial_r w_\pm(\hat r,\hat k)=0$  we conclude that $(\hat r,\hat k)=(\frac{\alpha+4}{2},1)$, and one easily checks that $F_1(\hat r,\hat k)=0$.
\end{itemize}

The proof of Lemma \ref{lem: the first zero appear on the red curve} is complete.
\end{proof}

\begin{proof}[Proof of Lemma \ref{lem: differential of w^2 along F1=0}]
By \eqref{equ: w pm} and \eqref{equ: r1}, we have
\begin{equation}\label{formula of tilde w}
w_{1,\pm}=\frac{(8k+\alpha(3-3k^2+2k^4))(G_1\pm \sqrt{I_1})}{32\alpha^2k(2k^2-3)(2k^2-1)^2}.
\end{equation}
where $G_1:=120-240k^2+96k^4-3\alpha k(-21+47k^2-28k^4+4k^6)$. Taking the derivative of $w_{1,\pm}$ with respect to $k$, we obtain
\begin{equation}\label{differential of w wrt k 1}
w_{1,\pm}'=\frac{3(k^2-1)F_3(\pm D+(5-2k^2)\sqrt{I_1})}{8\alpha^2 k^2(2k^2-1)^3(2k^2-3)^2\sqrt{I_1}}.
\end{equation}
Using that
\begin{equation}\label{simplify F3}
D^2- (5-2k^2)^2I_1=16(2k^2-3)^2(2k^2-1)F_4,
\end{equation}
we can multiply by $\pm D-(5-2k^2)\sqrt{I_1}$  both the numerator and the denominator of \eqref{differential of w wrt k 1}, to obtain the expression for $w_{1,\pm}'$ in (i).

To prove (ii), we first compute
\begin{equation*}
\begin{aligned}
F_3' & =k\cdot(256k(-3+5k^2)+\alpha^2k(3-2k^2)^2(-9-12k^2+44k^4)\\
&-4\alpha(3+2k^2)(27-54k^2+16k^4)).
\end{aligned}
\end{equation*}
Since $27-54k^2+16k^4<0$ on $[1,\sqrt{2}]$, we conclude that $F_3'>0$ on $[1,\sqrt{2}]$.

For every $k>\sqrt{2}$,  $F_3'/k$ is a quadratic form on $\alpha$ with negative discriminant
\begin{eqnarray*}
-432(2k^2-1)^2(k^4(k^2-2)(512k^2-448)+52k^4+252k^2-243).
\end{eqnarray*}
 Hence $F_3'$ is everywhere positive on $[1,+\infty)$  and thus $F_3$ is strictly increasing from $F_3(1)=-3\alpha(4+\alpha)<0$ to $+\infty$ in $[1,+\infty)$.

 Since $F_3(\sqrt{3/2})=24(4\sqrt{6}+3\alpha)>0$, we conclude that
 \begin{equation}\label{signF3}
 \begin{aligned}F_3 < 0 \mbox{ on } (1,\hat k_0) \quad \mbox{ and } \quad  F_3>0 \mbox{ on } (\hat
 k_0,+\infty),
 \end{aligned}
 \end{equation}
 for some $1<\hat k_0 < \sqrt{3/2}.$

Regarding $F_3$ as a quadratic form $A\alpha^2+B\alpha + C$, for each $k\in (1,\sqrt{3/2})$, we have $A<0,C>0,$ and $B^2-4AC >0$. We see that for each $k\in (1,\sqrt{3/2})$, there exists a unique $\alpha>0$ such $\hat k_0=k$. Indeed, a direct computation shows that $\partial_\alpha F_3(\hat k_0)=2A\alpha +B = -\sqrt{B^2-4AC} <0$ and thus, since $F_3'(\hat k_0) >0$,  $\hat k_0$ increases with $\alpha$. Moreover, $\hat k_0 \to 1$ as $\alpha\to 0$ and $\hat k_0 \to \sqrt{3/2}$ as $\alpha \to +\infty$. This proves (ii).

Now we prove (iii). Regarding $F_4$  as a quadratic form on $\alpha$, we solve $F_4=0$ for $\alpha$ to obtain
\begin{equation}\label{roots of F3=0}
\alpha_\pm:=\frac{4(13-22k^2+8k^4\pm(2k^2-1)^{\frac{3}{2}})}{k(2k^2-5)^2(k^2-1)}.
\end{equation}

We claim that $\alpha_+$ is increasing on  $\big(\frac{\sqrt{17+\sqrt{17}}}{4},\sqrt{5/2}\big)$ from $0$ to $+\infty$, and both $\alpha_-<\alpha_+$ are decreasing on $(\sqrt{5/2},+\infty)$ from $+\infty$ to $0$. Everywhere else, $\alpha_\pm$ is negative. To prove the claim, we see from \eqref{roots of F3=0} that $\alpha_-<0$ on  $(1,\sqrt{5/2})$ and $\alpha_+(\frac{\sqrt{17+\sqrt{17}}}{4})=0$. Indeed, $13-22k^2+8k^4-(2k^2-1)^{3/2}<0$ on $(1,\sqrt{5/2})$ and $13-22k^2+8k^4+(2k^2-1)^{3/2}=0$ at $k=\frac{\sqrt{17+\sqrt{17}}}{4}$ as one readily checks. After differentiating $\alpha_\pm$ with respect to $k$ and replacing $k$ with  $\nu=(2k^2-1)^{\frac{1}{2}}$, we obtain
\begin{eqnarray*}
\alpha_\pm(\nu)=\frac{8\sqrt{2}(-2\mp \nu+2\nu^2)}{(\mp2+\nu)(\nu^2-4)(\pm1+\nu)(\nu^2+1)^{\frac{1}{2}}},
\end{eqnarray*}
and
\begin{eqnarray*}
\alpha_\pm'(\nu)=\frac{16(\pm12+12 \nu\pm  \nu^2-7 \nu^3\pm 4\nu^4+6\nu^5)}{(\pm2-\nu)(\nu^2-4)^2(\pm1+\nu)^2(\nu^2+1)}.
\end{eqnarray*}
Notice that  $k\in\big(\frac{\sqrt{17+\sqrt{17}}}{4},\sqrt{5/2}\big)$ corresponds to $ \nu\in(\frac{1+\sqrt{17}}{4},2)$. Since $\alpha_+'(\nu)>0$ on $\big(\frac{1+\sqrt{17}}{4},2\big)$, and $\alpha'_\pm<0$ on $(2,+\infty)$, the claim holds.

To prove that $\hat k_0< \hat k_1$, we show that $F_3(\hat k_1)>0$.  Taking $\alpha=\alpha_+(k)$ and $k = \sqrt{(\nu^2+1)/2}$, we compute
\begin{equation}
F_3=\frac{8\sqrt{2} \nu^4G_3(\nu)}{(\nu-2)^2(\nu+1)^2(\nu^2-4)^2(\nu^2+1)^{\frac{1}{2}}},
\end{equation}
where
$$
\begin{aligned}
G_3(\nu) & :=48 + (\nu - 1) (48 (2 - \nu)^2 + 27 (2 - \nu)^4 \nu + (12 (2 - \nu)^6 \\ & + (2 - \nu)^2) \nu^2 +
(87 (2 - \nu)^4 + (2 - \nu)^3) \nu^3 + (7 (2 - \nu)^4 \\ & + 3 (\nu - 1)^5 + 6 (\nu - 1)^4) \nu^4
 + (15 (2 - \nu)^4 + (2 - \nu)^2) \nu^5 \\ & + (300 - 358 \nu + 107 \nu^2) \nu^6).
\end{aligned}
$$ It turns out that $G_3>0$ on $[1,2]$.
Therefore, (ii) implies that $\hat k_1>\hat k_0$.

We conclude that  for any $\alpha>0$, there exist precisely three zeros of $F_4$ $$
\frac{\sqrt{17+\sqrt{17}}}{4}<\hat k_1 < \sqrt{5/2} < \hat k_2 < \hat k_3,
$$
where $\hat k_1 > \hat k_0$, and $\alpha_+(\hat k_1)=\alpha_+(\hat k_3)=\alpha_-(\hat k_2)=\alpha$ .

We may rewrite
$$
F_4=k^2(k^2-1)(2k^2-5)^3(\alpha-\alpha_-)(\alpha-\alpha_+) \quad \forall k>1, k\neq \sqrt{5/2}.
$$
By the monotonicity of $\alpha_\pm$, we can further obtain
\begin{equation}\label{sign of F4}
\begin{aligned}
F_4 & >0 \mbox{ on } (\hat k_1,\hat k_2)\cup(\hat k_3,+\infty), \\
F_4 & <0 \mbox{ on } [1,\hat k_1)\cup(\hat k_2,\hat k_3).
\end{aligned}
\end{equation}
Indeed, $F_4(1)<0$, $\alpha-\alpha_-$  changes sign at $\hat k_2$, and $\alpha - \alpha_+$ changes sign at $\hat k_1,\hat k_3$.

Now we study the equation $D=0$. Regarding $D$ as a linear function of $\alpha$, we solve $D=0$ for $\alpha$ to obtain
$$
\alpha_D:=\frac{8(-51+116k^2-88k^4+24k^6)}{k(2k^2-5)(2k^2-3)(13-17k^2+6k^4)}, \quad k\geq 1, k\neq \sqrt{3/2}, \sqrt{5/2}.
$$
We claim that  $\alpha_D$ is increasing on $(1,\sqrt{3/2})$ from $4/3$ to $+\infty$,  and  decreasing on $(\sqrt{5/2},+\infty)$ from $+\infty$ to $0$. Moreover, $\alpha_D<0$ on $(\sqrt{3/2},\sqrt{5/2})$. Indeed, $-51+116k^2-88k^4+24k^6>0$ is strictly increasing on $(1,+\infty)$, and $(5k-2k^3)(3-2k^2)(13-17k^2+6k^4)>0$ is strictly decreasing to $0$ on $(1,\sqrt{3/2})$. Hence $\alpha_D$ is increasing on $(1,\sqrt{3/2})$. The change of signs on the denominator implies that $\alpha_D<0$ on $(\sqrt{3/2},\sqrt{5/2})$. Since
$$a_D'=\frac{-8G_4(k^2)}{k^2(2k^2-5)^2(2k^2-3)^2(13-17k^2+6k^4)^2},$$
where
$$
\begin{aligned}
G_4(s):=& -9945+48219k-107798k^2+138476k^3\\
& -109232k^4+52848k^5-14496k^6+1728k^7.
\end{aligned}
$$
Since $G_4(5/2)=10240$ and
$$
\begin{aligned}
G_4'(s+2)= & 2067+23012s+80580s^2+133312s^3\\
& +120240s^4+58176s^5+12096s^6>0,
\end{aligned}
$$
we conclude that $a_D'<0$ on $(\sqrt{5/2},+\infty)$.
Therefore, for any $\alpha\geq 4/3$, there exist
$$1\leq k_1 < \sqrt{3/2} < \sqrt{5/2}< k_2,
$$
solving $\alpha_D(k_i)=\alpha,i=1,2$. If $0<\alpha<4/3$, then $k_1<1$.

We compute
$$
D=k(2k^2-5)(2k^2-3)(13-17k^2+6k^2)(\alpha-\alpha_D).
$$
By the monotonicity of $\alpha_D$, we can further obtain that
\begin{equation}\label{sign of D}
\begin{aligned}
D & >0 \mbox{ on } [1,k_1)\cup(k_2,+\infty), \\
D & <0 \mbox{ on } (k_1, k_2).
\end{aligned}
\end{equation}

A direct computation using $\alpha=\alpha_D(k)$ gives
\begin{eqnarray*}
F_4=-\frac{32(2k^2-1)^3G_5(k^2)}{(39-77k^2+52k^4-12k^6)^2}.
\end{eqnarray*}
where
$$
G_5(s):=153-390s+397s^2-188s^3+36s^4.
$$
Since
$$G_5(s+1)=8(s-1)^2+(41-44s+36s^2)s^2>0,\ \forall s\geq 0,$$
we conclude that $F_4<0, \forall k\geq 1.$  Therefore, for every $\alpha \geq 4/3$, we conclude that $F_4(k_1),F_4(k_2)<0$ and from  \eqref{sign of F4}, we obtain
$$k_1<\hat k_1<\sqrt{5/2}<\hat k_2<k_2<\hat k_3.$$
This implies \eqref{evaluation of D at three k's}. The proof of (iii) is complete.

Since $\alpha=\alpha_+(\hat k_1),$ $\alpha_+(\frac{\sqrt{17+ \sqrt{17}}}{4})=0$ and $\alpha_+(\sqrt{5/2})=+\infty,$ (iv) follows.

Now we prove (v). By Lemma \ref{lem: roots of Delta}-(v), $I>0$  if $F_1=0$, that is $w_{1,+}\neq w_{1,-}$. Item (v) then follows directly from \eqref{equ: w pm}.

Finally we prove (vi). From (iii) we know that
$1<\hat k_0<\hat k_1<\sqrt{5/2}<\hat k_2<\hat k_3$, where $\hat k_0$ solves $F_3=0$ and $\hat k_i,i=1,2,3,$ solves $F_4=0$. From \eqref{simplify F3}, we have $D^2-(5-2\hat k_i^2)^2I_1=0, \forall i=1,2,3.$ Using inequalities
 in \eqref{evaluation of D at three k's}, we conclude that
 \begin{equation}\label{equ: sqrtI1}
 D+(5-2\hat k_i^2)\sqrt{I_1}=0, i=1,3, \quad \mbox{ and } \quad D-(5-2\hat k_2^2)\sqrt{I_1}=0.
 \end{equation}
Moreover,
 $$
 D+(5-\hat k_2^2)\sqrt{I_1}<0, \
 D-(5-2\hat k_1^2)\sqrt{I_1}<0,\
 D-(5-2\hat k_3^2)\sqrt{I_1}>0.
 $$
We conclude that
\begin{equation}\label{sign of D plus I}
\begin{aligned}
D+(5-2 k^2)\sqrt{I_1}<0  & \mbox{ on } (\hat k_1,\hat k_3),\\
D+(5-2 k^2)\sqrt{I_1}>0 & \mbox{ on } (1, \hat k_1) \cup (\hat k_3,+\infty),\\
D-(5-2 k^2)\sqrt{I_1}<0 & \mbox{ on } [1,\hat k_2),\\
D-(5-2 k^2)\sqrt{I_1}>0 & \mbox{ on } (\hat k_2,+\infty).
\end{aligned}
\end{equation}

Finally, combining \eqref{differential of w wrt k 1}, \eqref{signF3} and \eqref{sign of D plus I}, we obtain
\begin{equation}\label{sign of w minus k}
\begin{aligned}
     w_{1,-}' <0 \quad & \mbox{ on } \quad [1,\hat k_0)\cup (\hat k_2,+\infty),\\
w_{1,-}' >0 \quad & \mbox{ on } \quad (\hat k_0,\sqrt{3/2})\cup(\sqrt{3/2},\hat k_2),\\
     w_{1,+}' <0 \quad & \mbox{ on } \quad [1,\hat k_0)\cup (\hat k_1,\hat k_3),\\
w_{1,+}' >0  \quad & \mbox{ on } \quad (\hat k_0,\hat k_1)\cup(\hat k_3, +\infty).
\end{aligned}
\end{equation}
Item (vi) follows.
\end{proof}

\begin{proof}[Proof of Lemma \ref{lem: the first zero appear inside}]
Assume that $\varpi = \hat \varpi$. By the definition of $\hat \varpi$, we have $\min \Delta|_{\mathcal{H}} = 0$, and if $(\hat r,\hat k)\in \mathcal{H}$ is such that $\Delta(\hat r,\hat k) =0$, then either $w_+(\hat r,\hat k)=\hat \varpi^2$ or $w_-(\hat r,\hat k)=
\hat \varpi^2$. Also, Lemma \ref{lem: w^2 for the first point on the boundary}  implies that $(\hat r,\hat k)\in \mathcal{H} \setminus \partial \mathcal{H}$ and thus $\nabla \Delta(\hat r,\hat k ) =0$. This implies that either $\nabla w_+(\hat r,\hat k)=0$ or $\nabla w_-(\hat r,\hat k) =0$, if $\hat k>1$. Moreover, by Lemma \ref{lem: the first zero appear on the red curve}, $F_1(\hat r,\hat k)=0$ and if $\hat k>1$, then $\hat k$ is a critical point of $w_{1,+}$ or $w_{1,-}$. Also, by Lemma \ref{lem: differential of w^2 along F1=0}-(iv) and (v), we conclude that the only candidates for $(\hat r, \hat k)$ are the points $(r_1(1),1), (r_1(\hat k_i),\hat k_i),i=1,2,3.$ Suppose $\hat k>1$ is a critical point of $w_{1,+}$ or $w_{1,-}$. If $\hat k$ is not a local maximum for $w_{1,+}$ and also not a local maximum for $w_{1,-}$, then there exist points $k$ arbitrarily close to $\hat k$ such that $w_{1,+}(k)$ or $w_{1,-}(k)$ is slightly greater than $\hat \varpi^2$. Hence, for $\varpi$ slightly greater than $\hat \varpi$, we obtain points in the interior of $\mathcal{H}$ solving $\Delta=0$, a contradiction with the definition of $\hat \varpi$. We conclude that either $\hat k\in \{1, \hat k_1,\hat k_2\}$.

Let us exclude the case $\hat k=\hat k_2$. By Lemma \ref{lem: differential of w^2 along F1=0}-(iii), (iv) and (v), we know that $\hat k_2 > \sqrt{5/2}$ and thus $w_{1,+}(\hat k_1) >w_{1,+}(\hat k_2) > w_{1,-}(\hat k_2)$.  If $\hat k = \hat k_2$, then $w_{1,-}(\hat k_2) =\hat \varpi^2$. However, for $\varpi^2 = w_{1,+}(\hat k_1)>\hat\varpi^2$, the point $(r_1(\hat k_1),\hat k_1)\in \mathcal{H} \setminus \partial \mathcal{H}$, see Claim \ref{claimF1} below, and $\Delta(r_1(\hat k_1),\hat k_1) =0.$ This contradicts the definition of $\hat \varpi$.  We conclude that $\hat k \in \{1,\hat k_1\}$.

Now we exclude the case $\hat k = 1$. In this case, $(\hat r,\hat k) = (r_1(1),1)=(\alpha/2+2,1)\in \mathcal{H}\setminus \partial \mathcal{H}$. We explicitly compute
$$
w_{1,+}(1) = \frac{(1+4\alpha^{-1})^2}{4}< w_{1,-}(1) = \frac{(1+4\alpha^{-1})^2}{2}.
$$ Claim \ref{claimF2} below gives $w_{1,+}(\hat k_1) > w_{1,+}(1)$. Hence, if $\hat k=1$, then $w_{1,+}(1)=\hat\varpi^2$ and we can use Claim \ref{claimF1} again to conclude that for $\varpi^2=w_{1,+}(\hat k_1) > \hat \varpi^2$,  $\Delta$ vanishes on $(r_1(\hat k_1),\hat k_1) \in \mathcal{H}\setminus \partial \mathcal{H}$, a contradiction with the definition of $\hat \varpi$.

We have proved that $(\hat r,\hat k) = (r_1(\hat k_1),\hat k_1)$ and, since $\hat k_1$ is not a critical point of $w_{1,-}$, we conclude that  $\hat \varpi^2 = w_{1,+}(\hat k_1)$.

\begin{claim}\label{claimF1} If $\varpi^2 = w_{1,+}(\hat k_1)$, then $(r_1(\hat k_1),\hat k_1)\in \mathcal{H} \setminus \partial \mathcal{H}$, that is $$V(r_1(\hat k_1), \hat k_1) < -1.$$
\end{claim}

\begin{proof}[Proof of Claim \ref{claimF1}] Recall from Lemma \ref{lem: differential of w^2 along F1=0} that $\hat k_1\in (\sqrt{17+\sqrt{17}}/4,\sqrt{5/2}).$

Using a new variable $\nu := (2\hat k_1^2-1)^{1/2}\in((1+\sqrt{17})/4,2)$ instead of $\hat k_1$, we compute $V(r_1(\hat k_1), \hat k_1)$ as
$$
V_1:=\frac{4(-4-2\nu+3\nu^2+\nu^3-\nu^4)}{12+4\nu-9\nu^2-2\nu^3+3\nu^4},\quad \forall
 \nu \in((1+\sqrt{17})/4,2).
 $$
 In the formula above, we have used that $\alpha$ and $\hat k_1$ satisfy the relation $\alpha=\alpha_+(\hat k_1)$, see \eqref{roots of F3=0}.
Taking the derivative  with respect to $\nu$, we obtain
$$
V_1'=\frac{4(4-\nu^2)(\nu^2-1)(\nu^2+2)}{(12+4\nu-9\nu^2-2\nu^3+3\nu^4)^2}>0,\quad \forall \nu\in ((1+\sqrt{17})/4,2).
$$
Since $V_1(2) = -1$, we conclude that $V(r_1(\hat k_1),\hat k_1) <-1$ as desired.  \end{proof}

\begin{claim}\label{claimF2}
$w_{1,+}(1)<w_{1,+}(\hat k_1)<w_{1,-}(1).$
\end{claim}
\begin{proof}[Proof of Claim \ref{claimF2}]
From \eqref{formula of tilde w} and \eqref{equ: sqrtI1}, we obtain
$$
w_{1,+}(\hat k_1)
   = \frac{(\alpha \hat k_1(2\hat k_1^2-5)-8)(8\hat k_1+\alpha(3-3\hat k_1^2+2\hat k_1^4))}{4\alpha^2\hat k_1(5-12\hat k_1^2+4\hat k_1^4)}.
$$
Using a new variable $\nu := (2\hat k_1^2-1)^{1/2}\in((1+\sqrt{17})/4,2)$ instead of $\hat k_1$, replacing $\alpha$ with $\alpha_+(\hat k_1)$ as in \eqref{roots of F3=0}, we compute
\begin{equation*}
w_{1,+}(\nu) = \frac{\nu^2 (12 + 4 \nu - 9 \nu^2 - 2 \nu^3 + 3 \nu^4)}{8 (2 + \nu - 2 \nu^2)^2},
\end{equation*}
and
\begin{equation}\label{formula of wb}
W:=\frac{w_{1,+}(\nu)}{(1+4\alpha^{-1})^2}=
\frac{\nu^2(12+4\nu-9\nu^2-2\nu^3+3\nu^4)}{G^2},
\end{equation}
where
\begin{equation}\label{func G}
G := 2\sqrt{2}(-2-\nu+2\nu^2)+(2-\nu)^2(1+\nu)(2+\nu)\sqrt{1+\nu^2}
\end{equation}
is positive for every $ \nu\in((1+\sqrt{17})/4,2).$

Then
\begin{equation}\label{wbk}
W'=\frac{2\nu W_1W_2}{\sqrt{1+\nu^2}G^3},
\end{equation}
where
$$
\begin{aligned}
W_1  := & -3(2-\nu)^2(2-\nu)-3(2-\nu)^3(2-\nu)\nu\\
  & -(34-25\nu+7\nu^2+6\nu^3)(2-\nu)\nu^2,\\
W_2  := & -4+\nu^2-\nu^4+2\sqrt2\sqrt{1+\nu^2}.
\end{aligned}
$$
Notice that $W_1 < 0$ for every $\nu\in((1+\sqrt{17})/4,2),$ $W_2((1+\sqrt{17})/4)<0$, and
$$
W_2'=-\frac{2\nu((2\nu^2-1)\sqrt{1+\nu^2}-\sqrt{2})}{\sqrt{1+\nu^2}}<0,\quad \forall \nu\in((1+\sqrt{17})/4,2).
$$
It follows that  $W_1$ and $W_2$ are both negative on $((1+\sqrt{17})/4,2)$.
In particular, $W$ is strictly increasing from $W((1+\sqrt{17})/4)=(95-7\sqrt{17})/256>1/4$ to $W(2)=1/2$. The claim follows.
\end{proof}
\noindent Finally, Lemma \ref{lem: the first zero appear on the red curve} and Claim \ref{claimF2} imply that $\Delta$ is positive everywhere in $\mathcal{H}$ except at $(r_1(\hat k_1),\hat k_1)$.  The proof of (i) is complete.

Now we prove (ii). Assume that $\varpi^2<\hat \varpi^2= w_{1,+}(\hat k_1)$. We claim  that
\begin{equation}\label{lastpart}
(1+4\alpha^{-1})^2/4 = w_{1,+}(1) > w_{1,-}(\hat k_2).
\end{equation}
Since $w_{1,+}(\hat k_1) > w_{1,+}(1)$ and $w_{1,-}(\hat k_2) > w_{1,-}(\hat k_1)$, this inequality implies that $\Delta(r_1(\hat k_1),\hat k_1) <0,$ for every $w_{1,+}(1)\leq \varpi^2 < \hat \varpi^2.$ Moreover, since $(r_1(1),1),$ $(r_1(\hat k_1),\hat k_1) \in \mathcal{H}\setminus \partial \mathcal{H}$ for every $\varpi < \hat \varpi$ and $\Delta (r_1(1),1)<0$ for every $\varpi^2 < w_{1,+}(1),$ we conclude that $\min \Delta|_{\mathcal{H}}<0$ for every $\varpi < \hat \varpi,$ as desired.

To prove the claim, we use \eqref{formula of tilde w} and \eqref{equ: sqrtI1}, replacing $\alpha$ with $\alpha_-(\hat k_2)$ as in \eqref{roots of F3=0}, to obtain
$$
w_{1,-}(\nu) = \frac{\nu^2 (12 - 4 \nu - 9 \nu^2 + 2 \nu^3 + 3 \nu^4)}{8 (-2 + \nu + 2 \nu^2)^2},
$$
where $\nu = (2\hat k_2^2-1)^{1/2}, \forall \nu\in (2,+\infty).$ Replacing $\alpha$ with $\alpha_-(\hat k_2)$ in $w_{1,+}(1)=(1+4\alpha^{-1})^2/4,$ we obtain from a straightforward computation that
$$
w_{1,+}(1)=\frac{(2\sqrt{2}(-2+\nu + 2\nu^2)+(-2 +\nu)(-1+ \nu)(2+ \nu)^2(1+\nu^2)^{1/2})^2}{32(-2 + \nu + 2\nu^2)^2}
$$
Then
$$
\begin{aligned}
w_{1,+}(1) - w_{1,-}(\nu) & = \frac{(\nu-2)(f_1+f_2\sqrt{2+2\nu^2})}{32(-2 + \nu + 2 \nu^2)^2},
\end{aligned}
$$
where
$$
\begin{aligned}
f_1 & =-48 + 24 \nu + 72 \nu^2 + 12 \nu^3 - 10 \nu^4 - 23 \nu^5 - 22 \nu^6 -
 2 \nu^7 + 4 \nu^8 + \nu^9,\\
f_2 & = 4 (2 + \nu)^2  (2 - 3 \nu - \nu^2 + 2 \nu^3)
 \end{aligned}
$$
Observe that $f_2>0, \forall \nu >2$ and hence
$$
\begin{aligned}
w_{1,+}(1) - w_{1,-}(\nu)
& > \frac{(\nu-2)(f_1+f_2(\nu + 1/\nu))}{32(-2 + \nu + 2 \nu^2)^2}.
\end{aligned}
$$
Denoting $f_3(\nu):=\nu f_1(\nu) + f_2(\nu)(\nu^2 + 1),$
we obtain
$$
\begin{aligned}
f_3(\nu + 2)  = & 1280 + 7280 \nu + 18296 \nu^2 + 26252 \nu^3 + 23712 \nu^4\\ & +
  14118 \nu^5 + 5633 \nu^6 + 1490 \nu^7 + 250 \nu^8 + 24 \nu^9 + \nu^{10},
\end{aligned}
$$
which is clearly positive for every $\nu >0$. The claim is proved and the proof of Lemma \ref{lem: the first zero appear inside}.
\end{proof}

\hfill\newline
\noindent{\bf Acknowledgement.}
This work is partially supported by the National Key R\&D Program of China (2020YFA0713303). XH is partially supported by NSFC (\# $12071255$). YO is partially supported by the Qilu Young Scholar Program of Shandong University, and NSFC (\# 11801583). GY is partially supported by NSFC (\#12171253) and the Nankai Zhide Foundation. All authors sincerely thank Professor Jian Wang for the useful discussion about reversible maps. LL thanks the support of School of Mathematics at Shandong University and the support of BICMR at Peking University. LL and PS acknowledge the support of NYU-ECNU Institute of Mathematical Sciences at NYU Shanghai and the 2022 National Foreign Experts Program.





\end{document}